\documentclass[11pt]{article}

\usepackage[utf8]{inputenc}
\usepackage[T1]{fontenc}

\usepackage{amsmath}
\usepackage{graphicx,psfrag,epsf}
\usepackage{enumerate}
\usepackage[numbers]{natbib}
\usepackage{url}

\usepackage[left=1.2in,right=1.2in, top=0.7in, bottom=1.0in]{geometry} 

\usepackage{multirow}

\usepackage{graphicx}
\usepackage{amssymb, amsmath, amsthm}
\usepackage{epstopdf}

\usepackage{txfonts}

\usepackage[font=small,labelfont=bf]{caption}

\usepackage{comment}

\usepackage[font=footnotesize,labelfont=bf]{subcaption}
\usepackage{booktabs}
\usepackage{csquotes}
\usepackage{url}

\newtheorem{theorem}{Theorem}
\newtheorem{lemma}[theorem]{Lemma}

\DeclareMathOperator{\traj}{\operatorname{traj}}

\newcommand{\deltaSNR}{\mathrm{\Delta SNR}}
\newcommand{\dB}{\mathrm{\,dB}}
\newcommand{\Pos}{\mathrm{Pos}}

\usepackage{epstopdf}
\usepackage{tikz}
\usepackage{pgfplots}
\usepackage{pgfplotstable}

\usepackage{rotating}
\newcommand\tabrotate[1]{\begin{turn}{90}\rlap{#1}\end{turn}}

\title{
Wavelet Sparse Regularization for Manifold-Valued Data
}
\author{Martin Storath%
\thanks{Image Analysis and Learning Group, Universit\"at Heidelberg, Germany.},
Andreas Weinmann\thanks{Department of Mathematics and Natural Sciences, Hochschule Darmstadt, and Institute of Computational Biology, Helmholtz Zentrum München, Germany. 
}} 
\date{}

\newcommand{\N}{\mathbb{N}}

\def\supp{\operatorname{supp}}

\newcommand{\argmin}{\operatorname{\arg}\min}

\newcommand{\dist}{\mathrm{dist}}
\newcommand {\mean}{\mathrm{mean}}
\newcommand {\diam}{\mathrm{diam}}

\newcommand {\pt}{\mathrm{pt}}
\newcommand {\prox}{\mathrm{prox}}

\newlength\figureheight
\newlength\figurewidth
\begin{document}
\maketitle

\begin{abstract}
	In this paper, we consider the sparse regularization of manifold-valued data with respect to an interpolatory wavelet/multiscale transform. We propose and study variational models for this task and provide results on their well-posedness.  
	We present algorithms for a numerical realization of these models in the manifold setup. 
	Further, we provide experimental results to show the potential of the proposed schemes for applications. 		
\end{abstract}
{\small


}

\section{Introduction}

In various problems of applied sciences the data take values in a manifold. 
Examples are circle and sphere-valued data. 
They appear in interferometric SAR imaging \cite{massonnet1998radar}, as wind directions \cite{schultz1990circular}, and as orientations of flow fields \cite{adato2011polar, storath2017fastMedian}.
Further, in color image processing
they appear in connection with color spaces such as HSI, HCL, as well as in chromaticity based color spaces
\cite{chan2001total,vese2002numerical,kimmel2002orientation, lai2014splitting}.
Other examples are data with values in the special orthogonal group $\mathrm{SO}(3)$ which may express camera positions
or orientations of aircrafts \cite{rahman2005multiscale},
Euclidean motion group-valued data \cite{rosman2012group} representing, e.g., poses as well as shape-space data \cite{Michor07,berkels2013discrete}.
Another prominent manifold is the space of positive (definite) matrices $\Pos_n$ of dimension $n$
endowed with the Fisher-Rao metric  \cite{radhakrishna1945information}. 
For each $n \in \N,$ $\Pos_n$ is a Cartan-Hadamard manifold which has nice differential-geometric properties.
On the application side, $\Pos_3$ is the data space in diffusion tensor imaging \cite{pennec2006riemannian} 
which allows the quantification of diffusional characteristics of a specimen non-invasively \cite{basser1994mr,johansen2009diffusion}
and which is therefore helpful in the context of neurodegenerative pathologies such as schizophrenia \cite{foong2000neuropathological}  and autism 
\cite{alexander2007diffusion}. Another application of positive matrices are deformation tensors; cf. \cite{rahman2005multiscale}.

There are various directions on processing manifold valued data.
Manifold-valued partial differential equations have for instance been considered in   
\cite{tschumperle2001diffusion, chefd2004regularizing}.
In particular, finite element methods for manifold-valued data are the topic of \cite{grohs2013optimal,sander2015geodesic}.
Work on statistics on Riemannian manifolds can be found in 
\cite{oller1995intrinsic,bhattacharya2003large,fletcher2004principal,%
	bhattacharya2005large,pennec2006intrinsic,fletcher2007riemannian}.  
Optimization problems for manifold-valued data are for example the topic of \cite{absil2009optimization,absil2004riemannian, grohs2016varepsilon} 
and of \cite{hawe2013separable} with a view towards learning in manifolds. We also mention related work on optimization in Hadamard spaces \cite{bacak2014convex, bavcak2013proximal}.
Nonsmooth variational methods for segmentation and denoising 
using Potts and Mumford-Shah models for manifold-valued data were considered in \cite{weinmann2015mumford,storath2017jump}.  
TV functionals for manifold-valued data were considered in~\cite{GM06,GM07,GMS93}; 
algorithmic approaches for TV minimization were considered in ~\cite{SC11,LSKC13,grohs2016total,weinmann2014total, steidl16half_quadratic, storath2016exact}.
Examples of applications of TV regularization for medical imaging tasks 
can be found in \cite{baust2015total,stefanoiu2016joint,baust2016combined}.
Recently, TV type regularization with indirect data terms, which in particular may be used to realize deconvolution models in the manifold setting, has been considered in \cite{storath2018variational}.

Interpolatory wavelet transforms for linear space data have been investigated by D.~Donoho in \cite{donoho1992interpolating}.
Their analogues for manifold-valued data have been introduced by Ur Rahman, Donoho and their coworkers in \cite{rahman2005multiscale}. 
Such transforms have been analyzed and developed further in \cite{grohs2009interpolatory,grohs2010stability,weinmann2012interpolatory}. 
Typically, the wavelet-type transforms employ an (interpolatory) subdivision scheme
to predict the signal on a finer scale. 
The `difference' between the prediction and the actual data on the finer scale 
is realized by vectors living in the tangent spaces of the predicted signal points which point to the actual signal values; 
i.e., they yield actual signal values after application of a retraction such as the exponential map. 
These tangent vectors then serve as detail coefficients.
Subdivision schemes for manifold-valued data have been considered in  
\cite{wallner2005convergence, grohs2008smoothness,xie2008smoothness,weinmannConstrApprox, wallner2011convergence}.
Interpolatory wavelet transforms and subdivision are discussed in more detail in 
Section~\ref{sec:Preliminaries} below which deals with preliminaries. 
All the aforementioned approaches in this paragraph consider explicit schemes, i.e., the measured data is processed in a forward way using the 
analogues of averaging rules and differences in the manifold setting. 
In contrast, we here consider an implicit approach based on a variational formulation.

In this paper, we consider wavelet sparse regularization for manifold-valued data.
Let us briefly describe the univariate situation here; 
details and the multivariate setup are considered in Section~\ref{sec:model}.
Let $f \in \mathcal M^K$ be data living in the manifold $\mathcal M$ defined on a discrete grid of length $K.$ 
We consider the problem (which, for $q=1, \alpha=1$, is a variant of the LASSO \cite{tibshirani1996regression, chambolle1998nonlinear}  for manifold-valued data)
\begin{align}\label{eq:VarProblemIntro}
\argmin_{u \in \mathcal M^N }  \  \dist(\mathcal A(u),f)^p + \mathcal W_\lambda^{\alpha,q} (u).
\end{align}
Here, $u$ denotes the argument to optimize for; 
it may be thought of as the underlying signal 
generating the response $\mathcal A(u)$
where $\mathcal A $ is a (necessarily) nonlinear operator which models for instance a system's response.
In the simplest case of pure denoising, $\mathcal A$ may be implemented as the identity on $\mathcal M^N.$
Further instances of $\mathcal A$ are detailed in Section~\ref{sec:model}; 
examples are the manifold valued analogues of convolution operators.  
The nearness between $\mathcal A(u)$ 
and the data $f$ is measured by $\dist(\mathcal A(u),f)^p = \sum_{i=1}^{K}\dist(\mathcal A(u)_i,f_i)^p$
which denotes the $p$th power of the distance in $\mathcal M^K.$
The parameter vector $\lambda = (\lambda_1,\lambda_2)$ regulates the trade-off between the data fidelity, i.e., the distance to the data 
and the regularizing term $\mathcal W^{\alpha,q}.$
Let us more precisely describe the term $\mathcal W_\lambda^{\alpha,q}$ which is the central topic of the paper.
We let  
\begin{equation}
 \mathcal W_\lambda^{\alpha,q}(u) = 	
 \lambda_1 \cdot  \sum_{n,r} 2^{r q \left(\alpha+\tfrac{1}{2}-\tfrac{1}{q} \right)}\| d_{n,r}(u) \|^q_{\hat u_{n,r}} 
 + \lambda_2 \cdot \sum_n \dist(\tilde u_{n-1,0},\tilde u_{n,0})^q.
\end{equation}
Here the symbol 
$\| \cdot \|_{\hat u_{n,r}}$ denotes the norm 
induced by the Riemannian scalar product in the point 
$\hat u_{n,r},$ which is the point where the detail $d_{n,r}(u)$
is a tangent vector at.
The parameter $\alpha$ is a smoothness parameter and the parameter $q\geq 1$
stems from a norm type term. 
The details $d_{n,r}$ at scale $r$ of the interpolatory wavelet transform for manifold valued data are given by 
\begin{equation}
d_{n,r} =  d_{n,r}(u) =  2^{-r/2}\left( \tilde u_{n,r}  \ominus  \hat u_{n,r} \right),
\qquad \hat u_{n,r} = \mathrm{S} \tilde u_{n,r-1}.  
\end{equation}
Here $\tilde u_{n,r-1} = u_{2^{R-r+1}n}$  and $\tilde u_{n,r} = u_{2^{R-r}n}$  ($R$ the finest level)  
denote the thinned out target $u$ at scale $r-1$ and $r,$ respectively.
The coarsest level  is denoted by  $\tilde u_{n,0} = u_{2^{R}n}.$  
$\mathrm{S} \tilde u_{n,r-1}$ denotes the application of an interpolatory subdivision scheme $\mathrm S$ for manifold-valued data  	 
to the coarse level data $\tilde u_{\cdot,r-1}$ evaluated at the index $n$ which serves as prediction for $\tilde u_{n,r}.$
Then, $\tilde u_{n,r} \ominus \mathrm{S} \tilde u_{n,r-1} = \exp^{-1}_{\mathrm{S} \tilde u_{n,r-1}} \tilde u_{n,r}$ 
denotes the tangent vector sitting in $\hat u_{n,r} = \mathrm{S} \tilde u_{n,r-1}$ which points to $\tilde u_{n,r}.$ 
Here, we use the symbol $\exp^{-1}$ to denote the inverse of the Riemannian exponential function $\exp.$
We point out that we use the normalization of the coefficients $d_{n,r}$ employed in \cite{donoho1992interpolating} in the linear case.
Details are discussed in Section~\ref{sec:Preliminaries}.
The second term addresses the coarsest scale; it measures the $q$th power of the distance between the target   items on the coarsest scale.  
As already pointed out, the case $q=1,\alpha=1$ in \eqref{eq:VarProblemIntro}, 
corresponds to  the manifold analogue of the LASSO or $\ell^1$-sparse regularization which, in the linear case, is addressed by (iterative) soft thresholding \cite{donoho1995noising}.
This choice $q=1,\alpha=1$ is particularly interesting since it favors solutions $u$ which are sparse w.r.t.\ the considered wavelet expansion.
The manifold analogue of $\ell^0$-sparse regularization is obtained by using the regularizer
\begin{equation}
\mathcal W_\lambda^0(u) = 	\lambda_1  \ \#   \{(n,r) \ : \ d_{n,r}(u) \neq 0\}
 \ + \ \lambda_2 \ \#  \{n \ : \ \tilde u_{n-1,0} \neq \tilde u_{n,0}  \}.
\end{equation}
The operator $\#$ is used to count the number of elements in the corresponding set. 
Note that so the number of non-zero detail coefficients 
of the wavelet expansion is penalized.
Similar to the linear case 
\cite{weaver1991filtering,donoho1995noising,chambolle1998nonlinear},
potential applications of the considered sparse regularization techniques are denoising and compression.

\subsection{Contributions}

The contributions of the paper are as follows: \emph{(i)}
 we introduce and study a model for wavelet sparse regularization in the manifold setup; 
 \emph{(ii)} we provide algorithms for the proposed models;
\emph{(iii)} we show the potential of the proposed algorithms by applying them to 
concrete manifold-valued data.
Concerning \emph{(i)}, we propose a variational scheme employing manifold valued interpolatory wavelets in the regularizing term. In particular, we consider a sparsity promoting $\ell^1$ type term 
as well as an $\ell^0$ type term.  
We obtain results on the existence of minimizers for the proposed models.
Concerning \emph{(ii)}  we provide the details for an algorithmic realization of the proposed variational
model. 
In particular, we apply the concepts of a generalized forward backward-scheme with Gauss-Seidel type update and a trajectory method as well as the well-established concept of a cyclic proximal point algorithm.  
To implement these schemes we derive expressions for 
the (sub)gradients and proximal mappings of the atoms of the wavelet regularizing terms.
This includes the manifold analogues of $\ell^1$ and $\ell^0$ sparse wavelet regularization.
Concerning \emph{(iii)}, we provide a detailed numerical study of the proposed scheme.
We provide experiments with data living on the unit circle, in the two-dimensional sphere as well as in the space of positive matrices.

\subsection{Outline of the paper}

The paper is organized as follows. 
The topic of Section~\ref{sec:model} is to derive a model for wavelet sparse regularization for manifold valued data and show its well-posedness.
In Section~\ref{sec:algorithm} we consider the algorithmic realization of the proposed models.
In Section~\ref{sec:Experiments} we provide a numerical investigation of the proposed algorithms.
Finally, we draw conclusions in Section~\ref{sec:Conclusion}.

\section{Model}\label{sec:model}

 In Section~\ref{sec:Preliminaries} we provide basic information on subdivision schemes and interpolatory multiscale transforms
 needed to define a variational model for wavelet sparse regularization for manifold valued data.
 In Section~\ref{sec:WaveletManifold} we give a detailed formulation of the manifold analogue of the variational problem for wavelet sparse regularization. 
 In  Section~\ref{sec:well-posedness}, we obtain well-posedness results for the variational problem, i.e., we show that all considered problems have a minimizer.

\subsection{Preliminaries: Subdivision schemes, interpolatory multiscale transforms}\label{sec:Preliminaries}

We here provide information on subdivision schemes and interpolatory multiscale transforms we need later on. 

\paragraph{Subdivision schemes.} We consider the multivariate setup. The purpose of a subdivision scheme is to refine a grid function.
More precisely, given a function $p = (p_i)_{i \in \mathbb Z^s}$ which lives on a coarser $s$ dimensional grid,
a subdivision scheme $\mathrm{S}$ assigns a function $\mathrm{S}p$ living on a finer $s$ dimensional grid. 
Here, we use the multiindex notation $i=(i_1,\ldots,i_s) \in \mathbb Z^s$ to specify the points in the domain of the grid functions.
In case of dyadic refinement and a linear subdivision scheme, 
the assignment of $\mathrm{S}p$ to data $p$ is done via the rule
\begin{equation}\label{eq:defSubdivOperator}
\mathrm{S} p_n = 
\sum_{k \in \mathbb N^s} a_{n- 2 k} p_k, \qquad \text{ for } \ n \in \mathbb Z^s.     
\end{equation} 
Here, $a=(a_i)_{i \in \mathbb Z^s}$ is called the mask of the subdivision scheme $\mathrm{S}.$ 
We always assume that $\sum_{i \in \mathbb Z^d} a_i =1,$ which means that $\mathrm{S}$ reproduces constants. 
For simplicity, we will use dyadic refinement in the following. 
We point out, that our approach in this paper is also amenable for  refinement using general dilation matrices; references dealing with subdivision schemes for more general dilation matrices are
\cite{HanSol,HanConvergence,HanComputingSmooth,weinmann2012subdivision}. 
In the following we are interested in interpolatory subdivision schemes. 
To explain the notion of an interpolatory scheme we notice that the mask $a$ on an $s$-variate domain actually encodes 
$2^s$ convolution operators $a^{(l)}=(a_{l + 2i})_{i},$ $l \in \{0,1\}^s.$  
A scheme is interpolatory if the convolution operator corresponding to the coarse level grid $(l=0)$
is the identity operator. 
In formulas, this means that $a_{2 i} = \delta_{i0},$ for all $i \in \mathbb Z^s,$ were $\delta$ denotes the Kronecker symbol which equals one if and only if the indices are equal, and zero else.   
Examples of interpolatory schemes are the well-known four-point scheme \cite{dyn19874} or the schemes of Delaurier and Dubuc \cite{deslauriers1989symmetric} 
with the particular instances of the first order interpolatory scheme and the third order interpolatory Deslaurier-Dubuc scheme given by the masks 
\begin{equation}\label{eq:TheBasicSchemesEmployed}
	a = (\ldots,0, \tfrac{1}{2},1,\tfrac{1}{2},0,\ldots), \quad \text{ and } \quad
    a = (\ldots,0, -\tfrac{1}{16},0,\tfrac{9}{16}, 1,\tfrac{9}{16},0,-\tfrac{1}{16},0, \ldots),
\end{equation}
respectively.
These are univariate schemes which can be adapted such as to work in any multivariate domain of 
dimension $s$ by a tensor product construction: the mask $a'$ of the multivariate scheme is given by
 $a'_n = \prod_{i=1}^{s} a_{n_i}$ where $a$ is the mask of the univariate scheme, and $n \in \mathbb Z^s$ is a multiindex.

In recent years, there has been a lot of interest in nonlinear subdivision schemes 
and in particular in subdivision schemes dealing with geometric or manifold-valued data 
\cite{wallner2005convergence, grohs2008smoothness,xie2008smoothness,weinmannConstrApprox, wallner2011convergence}. 
For geometric data, typically geometric analogues of linear subdivision schemes are considered. 
This means that the coefficients stored in the mask $a$ are used to define a geometric scheme employing a manifold construction replacing the averaging operation in affine space. 
Examples of such constructions include geodesic averaging \cite{wallner2005convergence}, 
$\log$-$\exp$ schemes \cite{rahman2005multiscale} and intrinsic means \cite{weinmannConstrApprox}.
The latter have turned out to be particularly suitable 
\cite{wallner2011convergence, grohs2010general, grohs2010stability}. For this reason, and due to their variational nature, we concentrate on geometric schemes based on intrinsic means. 
The intrinsic mean variant of the subdivision scheme $\mathrm{S}$ is given by  
\begin{equation}\label{eqDefSubdiv}
\mathrm{S} p_n = \mean(a_{n-2\cdot}, p).     
\end{equation}
The intrinsic mean in the manifold $\mathcal M$ is given by  
\begin{equation}\label{eq:DefIntrinsicMeanAnalogue}
\mean(a_{n- 2\cdot}, p) =   \argmin_{v \in \mathcal{M}} \ \sum_j  a_{n-2j} \ \dist(v,p_{j})^2.  
\end{equation}
It replaces the affine average on the right hand side of \eqref{eq:defSubdivOperator} (which is the affine space variant.) 
The coefficients of the scheme are the coefficients of mask $a$ which are the same for a linear scheme and its geometric analogue.
Although the subdivision operator $\mathrm S$ is nonlinear in the general manifold setting we omit the brackets in the notation $\mathrm S u$ since this is particularly convenient in connection with index notation.

The first definition of the intrinsic mean can be traced back to Fr\'echet; a lot of work has been done by Karcher \cite{karcher1977riemannian}. For details including an historic overview we refer to \cite{afsari2011riemannian}; see also \cite{kendall1990probability}. Due to its use as a means of averaging in a manifold,
it is employed as a basic building in various works; e.g., \cite{pennec2006riemannian,fletcher2007riemannian,tron2013riemannian}
as well as many references in the paragraph above. 
Various work deals with the numerical computation of the center; 
see, for instance \cite{afsari2013convergence,ferreira2013newton,arnaudon2013medians}.

The mean  $\mean(a_{i,\cdot},u)$ is not unique for all input constellations. 
However, for data living in a small enough ball uniqueness is guaranteed. 
For a discussion, precise bounds and an extensive list of reference concerning this interesting line of research we refer to Afsari \cite{afsari2011riemannian}. 
A reasonable way to deal with the nonuniqueness is to consider the whole set of minimizers in such a case.
In this paper we do so where -- in case of nonuniqueness -- we add an additional constraint originating from our variational formulation described in Section~\ref{sec:WaveletManifold}.

\paragraph{Interpolatory wavelet transforms.}

Interpolatory wavelet transforms for data living in linear spaces have been considered 
by D. Donoho in \cite{donoho1992interpolating}. 
To explain the idea, we consider grid functions $p_{n,r}, p_{n,r-1}$ 
at the finer scale $r$ and the coarser scale $r-1$
which are related via $p_{2n,r} = p_{n,r-1}$ for all multiindices $n$ and levels  $r$ 
(which might originate from point samples of a function $f$ 
defined on $\mathbb R^s$ with values in $\mathbb R.$)
The interpolatory wavelet transforms for real valued data in \cite{donoho1992interpolating}
are then defined by mapping the signal $p$ to the details 
\begin{equation}\label{eq:DefDetailsIntMult}
d_{n,r} =  d_{n,r}(p) =  2^{-sr/2}\left(p_{n,r} \ominus  \hat p_{n,r} \right), 
\qquad \hat p_{n,r} = \mathrm S p_{n,r-1},     
\end{equation}
where $n$ is a multiindex, $\ominus$ denotes the ordinary subtraction in a vector space, and  $\mathrm S$ denotes a linear interpolatory subdivision scheme.   
Please note the slight abuse of notation in the expression $\mathrm S p_{n,r-1}$ which represents the scheme $\mathrm S$
applied to the data $p_{\cdot,r-1}$ evaluated at the multiindex $n.$ The transformation then maps 
\begin{equation}
p_{n,R}  \mapsto (d_{n,R},\ldots, d_{n,1}, p_{n,0}),
\end{equation}
where $p_{\cdot,R}$ is a grid function at the finest scale used as input.

Interpolatory transforms for manifold-valued data have been introduced in \cite{rahman2005multiscale}. 
In the manifold situation, we interpret \eqref{eq:DefDetailsIntMult} as follows: 
(i) the symbol $\mathrm S$ denotes a geometric subdivision scheme $\mathrm S$,
here concretely the intrinsic mean analog given by \eqref{eq:DefIntrinsicMeanAnalogue} 
of a linear interpolatory scheme, e.g.,  a multivariate tensor product Deslaurier-Dubuc scheme; 
(ii) the symbol $\ominus$ is interpreted via the the inverse of the Riemannian exponential mapping $\exp,$
concretely, $ p_{n,r} \ominus \mathrm{S} p_{n,r-1} = \exp^{-1}_{\mathrm{S} p_{n,r-1}} p_{n,r}.$
We note that, in contrast to the real valued situation, where the tangent spaces at all points are naturally identified, the details now are tangent vectors with $d_{n,r}$ sitting in $\mathrm{S} p_{n,r-1}.$ 
Concerning analytic results for these transforms we refer to  \cite{grohs2009interpolatory,grohs2010stability,weinmann2012interpolatory}.

\subsection{Wavelet Sparse Regularization in the Manifold Setup -- Variational Formulation}\label{sec:WaveletManifold}

The previously discussed interpolatory wavelet transforms are explicit means of processing data
in the sense that analogues of averaging rules and differences in the manifold setting are applied to the data 
directly. In contrast, we here consider an implicit approach based on a variational formulation
where the interpolatory wavelet transform is used as a regularizing term.

In the following, we always consider a complete and connected Riemannian manifold $\mathcal M$ 
(with its canonical metric connection, its Levi-Civita connection). 
More precisely, we consider a manifold $\mathcal M$ with a Riemannian metric which is a smoothly varying inner product
$\langle \cdot,\cdot \rangle_p$
in the tangent space of each point $p$. 
According to the Hopf-Rinow theorem, the complete manifold $\mathcal M$ is geodesically complete in the sense
that any geodesic can be prolongated arbitrarily. 
The Levi-Civita connection is the only connection which is symmetric and which is compatible with the metric. For an account on Riemannian geometry we refer to the books \cite{spivak1975differential,do1992riemannian}. 

Extending the discussion started in the introduction, we derive a model for
wavelet sparse regularization for manifold-valued data in the multivariate situation.
We consider data $f \in \mathcal M^K,$ where $K=(K_1,\ldots,K_{s'}) \in \mathbb K^{s'}$ denotes a multiindex consisting of $s'$ components. We propose to consider the problem
\begin{align}\label{eq:VarProblemModelSec}
\argmin_{u \in \mathcal M^N }  \  \dist(\mathcal A(u),f)^p +  \mathcal W_\lambda^{\alpha,q} (u).
\end{align}
where the target variable $u \in \mathcal M^N = \mathcal M^{N_1} \times \ldots \times \mathcal M^{N_s}$ 
is a multivariate function defined on a regular grid, 
and, in contrast to \eqref{eq:VarProblemIntro}, $N=(N_1,\ldots,N_s) \in \mathbb N^s$ denotes a multiindex consisting of $s$ components representing the $s$-dimensional domain of $u.$ 
For instance, for a classical image $s=2,$ and for a volume $s=3.$ 
Further, we use the usual multiindex notation to denote
\begin{align}
\dist(\mathcal A(u),f)^p = \sum_{i=0}^{K-1}\dist(\mathcal A(u)_i,f_i)^p = 
 \sum_{i_1=0}^{K_1-1}\ldots \sum_{i_{s'}=0}^{K_{s'}-1} 
 \dist(\mathcal A(u)_{i_1,\ldots,i_{s'}},f_{i_1,\ldots,i_{s'}})^p
\end{align}
the $p$th power of the distance in $\mathcal M^K.$
The parameter vector $\lambda = (\lambda_1,\lambda_2)$ regulates the trade-off between the data fidelity, i.e., the distance to the data 
and the regularizing term $\mathcal W_\lambda^{\alpha,q}.$
Possible choices of the imaging operator $\mathcal A $ are discussed below.

Our central topic is the precise definition of the wavelet sparse regularizing term $\mathcal W_\lambda ^{\alpha,q}$ 
in an $s$-variate situation. 
We assume that the signal dimensions $N=(N_1,\ldots,N_s)$
and the scale level $R$ are related via $N=2^R \cdot N'$ with the scalar $R$ and the multiindices $N,N'.$
Then, the coarsest scale $r=0$ has size $N'=(N'_1,\ldots,N'_s)$ with $N'_\nu$ data points in the $\nu$th component,
$1\leq \nu \leq s,$ $\nu \in \mathbb N$.
We define $\tilde u$ by 
\begin{equation}\label{eq:DefSampling}
\tilde u_{n,r} = u_{2^{R-r}n},   \qquad  0 \leq r \leq R.
\end{equation} 
Here, $n$ denotes a multiindex, which, in dependence on the level $r,$ 
takes values in the range $ 0 \leq n <  2^r  N'.$
With the details 
$d_{n,r}(u) =  2^{-sr/2}\left(\tilde u_{n,r} \ominus  \hat u_{n,r} \right)$  
where $\hat u_{n,r} = \mathrm S \tilde u_{n,r-1}$  
defined by
\eqref{eq:DefDetailsIntMult} of the interpolatory wavelet transform, we have  
\begin{equation}\label{eq:DefWavRegMult}
\mathcal W_\lambda^{\alpha,q}(u) = 	
\lambda_1 \cdot \sum_{n,r} 2^{r q \ \left(\alpha+\tfrac{s}{2}-\tfrac{s}{q} \right)}\| d_{n,r}(u) \|^q_{\hat u_{n,r}} +
\lambda_2 \cdot \sum_{n,i} \dist(\tilde u_{n-e_i,0},\tilde u_{n,0})^q.
\end{equation}
Here, 
$\| \cdot \|_{\hat u_{n,r}}$ denotes the norm
induced by the Riemannian scalar product in the points 
$\hat u_{n,r} = \mathrm S \tilde u_{n,r-1},$ 
and we recall that the symbol $s$ denotes the spatial dimension of the domain of $u.$ 
Concerning the second term which addresses the coarsest level, we let $e_i=(\delta_{ij})_{1 \leq j \leq s},$ $i \in {1,\ldots,s}$, be the $i$th unit vector. The term measures the $q$th power of the distance between neighboring data items on the coarsest scale w.r.t.\ any of the coordinate directions.   
We note that often three parameters $\alpha,q,q'$ are considered in connection with the linear space analogue of the regularizer in \eqref{eq:DefWavRegMult}, cf. \cite{chambolle1998nonlinear}. In this respect, we here consider the particular case of $q=q'.$
As already pointed out we are particularly interested in the case $q=1$ in \eqref{eq:DefWavRegMult}, 
which corresponds to the manifold analogue of $\ell^1$-sparse regularization. 

The manifold analogue of $\ell^0$-sparse regularization is given by using the regularizer 
\begin{equation}\label{eq:Defl0SparseMult}
\mathcal W_\lambda^0(u) = 	\lambda_1  \ \# \  \{(n,r) \ : \ d_{n,r}(u) \neq 0\}  \ + \
\lambda_2 \ \# \  \{(n,i) \ : \ \tilde u_{n-e_i,0} \neq \tilde u_{n,0}  \} 
\end{equation}
which incorporates the number of non-zero detail coefficients $d_{n,r}.$

\paragraph{Handling nonunique means.}

For the definition of the details in \eqref{eq:DefDetailsIntMult} we use geometric subdivision schemes 
and we have defined geometric subdivision scheme using intrinsic means in \eqref{eqDefSubdiv}.
As explained above, intrinsic means are locally unique.
(We note that it is often the case in differential geometry that the considered objects are only locally unique, e.g., geodesics.)
To avoid complications, we add an additional constraint originating from our variational formulation.
In case of non-uniqueness, we denote the set of all minimizers in \eqref{eq:DefIntrinsicMeanAnalogue} by 
$\mathrm Sp_n$ and single out those which minimize the corresponding expression used to define $\| d_{n,r}(u) \|.$
More precisely, in such cases, we overload the previously given definition in \eqref{eq:DefDetailsIntMult} by 
\begin{equation}\label{eq:DefDnr}
   d_{n,r}(u) = \left\{ \hat d_{n,r}:
    \hat d_{n,r} = 2^{-sr/2}\left(\tilde u_{n,r} \ominus \hat u_{n,r} \right), 
   \hat u_{n,r} \in \mathrm S u_{n,r-1}, \hat d_{n,r} \in \argmin_{d'_{n,r} \in \mathrm S u_{n,r-1}} \|d'_{n,r}\|   \right\}.
\end{equation}
This means that we choose the details of smallest size (in the sense of the Riemannian metric).

\paragraph{Instances of the imaging operator $\mathcal A$.}

We here discuss various instance of possible imaging operators to give an impression to the reader.
The discussion is by no means complete. 
At first, letting $\mathcal A$ be the identity in the manifold $\mathcal M^s$ 
corresponds to pure denoising. Inpainting situations can be modeled by removing the missing summands.
Let us be more precise.  
To this end, let us consider a matrix $A$ with real-valued entries $a_{ij}$ given as
\begin{equation}
A = 
\begin{pmatrix}
a_{11}  & \cdots & a_{1N}  \\
\vdots  &        & \vdots  \\
a_{K1}  & \cdots & a_{KN}.         
\end{pmatrix}	
\end{equation}
$i$th positive row sums, i.e., $\sum_j a_{ij} > 0$ for all $i=1,\ldots,K.$
(Note that we do not require the particular items $a_{ij}$ to be nonnegative.)  
Here $i,j,N,K$ can be read as multiindices.
Using the matrix $A$ as a kind of kernel, we may define the $i$th component of $\mathcal A(u)$ by
\begin{align}\label{eq:DefMatVecAnaManiIth}
\mathcal A(u)_i = \mean(a_{i,\cdot},u) = \argmin_{v \in \mathcal{M}} \ \sum_j  a_{i,j} \dist(v,u_j)^2.
\end{align}
Then, the denoising situation corresponds to $A$ being the identity and the inpainting situation corresponds to 
removing those rows of the identity matrix which correspond to missing data.
But these are only special cases; actually, we can use any matrix $A$ with positive row sums; see \cite{storath2018variational}.
In particular, by this construction, we can realize the geometric analogue of any convolution operator which reproduces constants. Explicitly, in the bivariate situation (without multiindices), this reads  
\begin{align}\label{eq:ConvMani2D}
\mathcal A(u)_{rs} = \argmin_{v \in \mathcal{M}} \ \sum_{k,l}  a'_{r-k, s-l}  \ \dist(v,u_{kl})^2,
\end{align}
where $a'$ denotes the bivariate convolution kernel. 
We further notice, that in the manifold case, where each data item can be more complex than a real number,
also $\mathcal A$ may denote an item/pixel-wise construction or reconstruction process. 
An example of such a pixel-wise construction process is the pixel-wise generation of diffusion tensor from
DWI measurements in diffusion tensor imaging \cite{baust2016combined}. 
Another example in connection with shape spaces is \cite{stefanoiu2016joint}.

\paragraph{Handling the boundary of the image or volume.}

In the context of wavelets, classical means of dealing with data on a bounded domain consisting of a Cartesian product of intervals is to extend the data beyond the domain by either extending the data by $0$ (zero-padding), by periodizing, or by extending the data by reflection. In the setup of interpolatory manifold valued data there is no distinguished zero-element such that the notion of zero padding does not generalize immediately. 
However, periodizing and extending the data by reflection are notions which directly carry over to the manifold and interpolatory setup.
In the context of orthogonal and biorthogonal wavelets on the interval 
more sophisticated methods using tailored boundary rules have been developed.   
For details we refer to \cite{cohen1993wavelets} and the references therein.
Inspired by these approaches for orthogonal and biorthogonal wavelets in the linear case, 
we use specially tailored boundary rules in experiments in this work. In particular, for a $k$th order linear univariate Deslaurier-Dubuc scheme, a natural boundary treatment consists of fitting a $k$th order polynomial to the $k+1$ leftmost or rightmost data points, and evaluating these polynomials at the corresponding half-integers.

\subsection{Existence of minimizers}\label{sec:well-posedness}

We here derive results on the existence of minimizers for the variational problem 
\eqref{eq:VarProblemModelSec} of wavelet sparse regularization using the regularizers 
$\mathcal W_\lambda^{\alpha,q}(u)$ of \eqref{eq:DefWavRegMult}
and 
$\mathcal W_\lambda^0(u)$ of \eqref{eq:Defl0SparseMult}.
For the $\mathcal W_\lambda^{\alpha,q}(u)$ regularizer with $\lambda_2 > 0$
we get the existence of minimizers without additional constraints 
on the measurement operator $\mathcal A$
and the manifold $\mathcal M$. This result is formulated as 
Theorem~\ref{thm:ExistenceStandardProbLambda2neq0}. 
For the $\mathcal W_\lambda^{\alpha,q}(u)$ regularizer with $\lambda_2 = 0$
we need some additional constraints on $\mathcal A.$
This result is formulated as
Theorem~\ref{thm:ExistenceStandardProbLambda2equals0}.
In particular, existence is guaranteed for the denoising situation with $\mathcal A$ being the identity.
Finally, we give a corresponding existence result for  wavelet sparse regularization 
using $\ell^0$ type regularizing terms in Theorem~\ref{thm:ExistenceL0Sparse}.
We note that, for compact manifolds -- which include the spheres in euclidean space, the rotation groups, as well as the Grassmannian manifolds -- we get the existence of minimizers 
of the wavelet regularizers 
$\mathcal W_\lambda^{\alpha,q}(u)$ regularizer with $\lambda_2 = 0$
and of the  wavelet sparse regularizers $\mathcal W_\lambda^{0}(u)$  
using $\ell^0$ type regularizing terms
without additional constraints on the measurement operator $\mathcal A.$

We start with some preparation.
\begin{lemma} \label{lem:WalphaQTermLSC}
	The regularizing term $\mathcal W_\lambda^{\alpha,q}(u)$ 
	of \eqref{eq:DefWavRegMult} is lower semicontinuous.
\end{lemma}

\begin{proof}	
	In order to show that the sum in \eqref{eq:DefWavRegMult} is a lower semicontinuous function of $u,$
	it is enough to show that each member of the sum is a lower semicontinuous function of $u.$
	The members of the form $u \mapsto \dist(\tilde u_{n-e_i,0},\tilde u_{n,0})^q$ are continuous by the continuity of the distance function $\dist$.
	We show that the mappings $u \mapsto \| d_{n,r}(u) \|^q_{\hat u_{n,r}}$ 
	are lower semicontinuous functions on $\mathcal M^N.$ 
	To this end, we consider a sequence $(u^{(l)})_{l \in \mathbb N},$ each $u^{(l)} \in \mathcal M^N,$
	such that $u^{(l)} \to u$ in $ \mathcal M^N,$ as $l \to \infty.$ Since the power function is increasing, it is sufficient to show that 	
	\begin{equation}\label{eq:ToShowLSC}
	\| d_{n,r}(u) \|_{\hat u_{n,r}}  \leq   \liminf_{l} \| d_{n,r}(u^{(l)}) \|_{\hat u^{(l)}_{n,r}}, 
	\quad \text{ for all $n,r.$}
	\end{equation}
	Here, in case of non-unique details, we choose a realization $\hat d_{n,r}(u^{(l)}) \in d_{n,r}(u^{(l)})$ 
	and a corresponding prediction $\hat u^{(l)}_{n,r} \in \mathrm S u^{(l)}_{n,r-1}$ according to 
	\eqref{eq:DefDnr}. We note that, by \eqref{eq:DefDnr}, 
	the value $\| d_{n,r}(u^{(l)}) \|$ does not depend on the choice of  $\hat d_{n,r}(u^{(l)}).$
	
	Since, by assumption, $u^{(l)} \to u$ in $\mathcal M^N,$ we find a point $x \in \mathcal M$ and a positive number $r$ such 
	that all members $u^{(l)}_j$ together with all $u_j,$ 
	$j=1,\ldots,N,$ and 
	$l \in \mathbb N$ are contained in a common ball $B(x,R')$ around $x$ with radius $R',$
	or in other words, all members of all sequences are contained in a common bounded set.
	In particular, $u^{(l)}_{n,r-1} \in B(x,R'),$ cf. \eqref{eq:DefSampling}.
	As explained in Section~\ref{sec:Preliminaries} a subdivision scheme encodes 
	$2^s$ convolution operators with kernels given by $a^{(j)}=(a_{j + 2i})_{i},$ $j \in \{0,1\}^s,$ 
	where $a$ denotes the mask of the scheme. 
	Each convolution operator is defined via the intrinsic mean; cf. 
	\eqref{eqDefSubdiv} and \eqref{eq:DefIntrinsicMeanAnalogue}.
	Hence, for each convolution operator defined via $a^{(j)},$ we may apply
	\cite[Lemma 2]{storath2018variational} to obtain a positive number $R_j$
	such that all means 	
	$\mean(a^{(j)}_{i,\cdot},u^{(l)}_{n,r-1})$ are contained in a common ball $B(x,R_j).$ 		
	Taking the maximum of this radii $R = \max_{j} R_j,$ w.r.t.\ the $2^s-1$ averaging operators, where
	$s$ denotes the dimension of the domain, we have that all
	$\hat u^{(l)}_{n,r} \in \mathrm S u^{(l)}_{n,r-1}$ are contained in the common ball $B(x,R)$
	for all $l \in \mathbb N.$
	Hence, the $\hat u^{(l)}_{n,r}$ form a bounded sequence. 
	Further, the sequence $\tilde u^{(l)}_{n,r}$ is bounded as well as a convergent sequence.
    As a next step, we choose a subsequence indexed by $l_k$ such that 
	\begin{equation}\label{eq:ToShowLSC2}
	\lim_{k} \| d_{n,r}(u^{(l_k)}) \|_{\hat u^{(l_k)}_{n,r}}  
	=  	2^{-sr/2} \lim_{k} \|\tilde u^{(l_k)}_{n,r} \ominus \hat u^{(l_k)}_{n,r}\|_{\hat u^{(l_k)}_{n,r}}
	= \liminf_{l} \| d_{n,r}(u^{(l)}) \|_{\hat u^{(l)}_{n,r}}, 	
	\end{equation}
	and now invoke the boundedness of the sequences. 
	By the Hopf-Rinow theorem, since $\mathcal M$ is assumed to be geodesically complete,   
	we may extract convergent subsequences $\hat u^{(l_k)}_{n,r}, \tilde u^{(l_k)}_{n,r}$ 
	(which we, abusing notation for better readability, also denote by the indexation $n_k$) such that the limits 
	\begin{equation}\label{eq:DefvPrime}
	v_{n,r} := \lim_{k \to \infty} \hat u^{(l_k)}_{n,r}
	\end{equation}
	exist, 
	and such that $\tilde u^{(l_k)}_{n,r}$ converges to $u_{n,r}$ since $u^{(l)}$ is assumed to converge to $u.$
	Since 
	$
	\|\tilde u^{(l_k)}_{n,r} \ominus \hat u^{(l_k)}_{n,r}\|_{\hat u^{(l_k)}_{n,r}}
	=  \dist(\tilde u^{(l_k)}_{n,r},\hat u^{(l_k)}_{n,r}),
	$
	and the distance function is continuous, we have 
	\begin{align}\label{eq:LscFirstPart}
	     \|\tilde u_{n,r} \ominus v_{n,r}\|_{v_{n,r}}
		 = \lim_{k} \|\tilde u^{(l_k)}_{n,r} \ominus \hat u^{(l_k)}_{n,r}\|_{\hat u^{(l_k)}_{n,r}}
		 =  2^{sr/2} \liminf_{l} \| d_{n,r}(u^{(l)}) \|_{\hat u^{(l)}_{n,r}}, 	
	\end{align}
	using \eqref{eq:ToShowLSC2} for the last identity.	
	We next show that $v_{n,r} \in \mathrm S \tilde u_{n,r-1}.$
	To see this, we assume towards a contradiction that $v_{n,r}$ is not in $\mathrm S \tilde u_{n,r-1}= \mean(a_{n- 2\cdot}, \tilde u_{\cdot,r-1}).$
	Then, 
	there is an element $y \in \mean(a_{n- 2\cdot}, \tilde u_{\cdot,r-1})$ 
	such that,
	by the definition in \eqref{eq:DefIntrinsicMeanAnalogue},  
	$ \sum_j  a_{n-2j} \ \dist(y,\tilde u_{j,r-1})^2 < \sum_j  a_{n-2j}
	 \ \dist(v_{n,r},\tilde u_{j,r-1})^2.$
	Since 
	$\hat u^{(l_k)}_{n,r} \to v_{n,r}$ as $k \to \infty$
	by \eqref{eq:DefvPrime}, there is an index $k_0$
	such that 
	$ \sum_j  a_{n-2j} \ \dist(y,u^{(l_{k_0})}_{j,r-1})^2 < \sum_j  a_{n-2j} \ \dist(\hat u^{(l_{k_0})}_{n,r},u^{(l_{k_0})}_{j,r-1})^2$
	which contradicts $u^{(l_{k_0})}_{n,r}$ being a minimizer of the corresponding sum.	
	Hence $v_{n,r} \in \mathrm S u_{n,r-1}.$ By the definition in \eqref{eq:DefDnr}
	this implies
	$ \|\tilde u_{n,r} \ominus \hat u_{n,r}\|_{\hat u_{n,r}}  \leq  \|\tilde u_{n,r} \ominus v_{n,r}\|_{v_{n,r}},$
	for any $\hat u_{n,r}$ as in \eqref{eq:DefDnr}.
	Hence,  $\| d_{n,r}(u) \|_{\hat u_{n,r}} \leq 2^{-sr/2} \|\tilde u_{n,r} \ominus v_{n,r}\|_{v_{n,r}}.$
    Together with \eqref{eq:LscFirstPart}, this implies 
    \begin{align}\label{eq:LscFirstAndSecondPart}
    	\| d_{n,r}(u) \|_{\hat u_{n,r}}  \leq 2^{-sr/2} \|\tilde u_{n,r} \ominus v_{n,r}\|_{v_{n,r}}
    	=  \liminf_{l} \| d_{n,r}(u^{(l)}) \|_{\hat u^{(l)}_{n,r}}, 	
    \end{align}  
	which means that 
	the mappings $u \mapsto \| d_{n,r}(u) \|^q_{\hat u_{n,r}}$ 
	are lower semicontinuous and, in turn,
	yields the assertion of the lemma.
\end{proof}

\begin{lemma} \label{lem:WnullQTermLSC}
	The $\ell^0$ type regularizing term 
	$\mathcal W_\lambda^0(u)$ of \eqref{eq:Defl0SparseMult}
	is lower semicontinuous.	
\end{lemma}

\begin{proof}		
		We show that the mappings 
	    \begin{equation}
	    f^1_{n,r}: u \mapsto 
	    \begin{cases}
	    1,  & \text{if } d_{n,r}(u) \neq  0,    \\
	    0,  & \text{else,}    \\
	    \end{cases} \quad   
	    \quad
	    f^2_{n,i}: u \mapsto 
	    \begin{cases}
	    1,  & \text{if } \tilde u_{n-e_i,0} \neq \tilde u_{n,0},    \\
	    0,  & \text{else.}    \\
	    \end{cases}	    
	    \end{equation}
	    are lower semicontinuous functions of $u.$
	    This implies the assertion of the lemma since 
	    $\mathcal W_\lambda^0(u) = \lambda_1 \sum_{n,r}f^1_{n,r}(u) + \lambda_2 \sum_{n,i}f^2_{n,i}(u).$   
		To see the lower semicontinuity of $f^2_{n,i},$
		let $u^{(l)} \to u$ in $ \mathcal M^N,$ as $l \to \infty.$
		If $f^2_{n,i}(u)=0$ there is nothing to show, so we assume that 
		$f^2_{n,i}(u)=1.$ Then, $u_{n-e_i,0} \neq \tilde u_{n,0},$ and since $u^{(l)} \to u,$ this implies 
		 $u^{(l)}_{n-e_i,0} \neq \tilde u^{(l)}_{n,0},$ for sufficiently large $l.$
		Hence, for suffiently large $l$,  $f^2_{n,i}(u^{(l)})=1,$ and therefore, 
		$\liminf_l f^2_{n,i}(u^{(l)}) \geq f^2_{n,i}(u).$
		We next study the lower semicontinuity of $f^1_{n,r},$ and again let $u^{(l)} \to u.$
		As above, if $f^1_{n,r}(u)=0,$ there is nothing to show.
		So we may assume that $d_{n,r}(u) \neq 0,$ which means that  
		$\tilde u_{n,r} \not\in \mathrm S u_{n,r-1}$ by \eqref{eq:DefDnr}.
		We show that $\tilde u^{(l)}_{n,r} \not\in \mathrm S u^{(l)}_{n,r-1}$ for sufficently large $l$
		which then implies that $f^1_{n,r}(u^{(l)})=1.$
		Assume to the contrary, that there is a subsequence $u^{(l_k)}$ such that
		$\tilde u^{(l_k)}_{n,r} \in \mathrm S u^{(l_k)}_{n,r-1}.$
		Then, $\lim_k u^{(l_k)}_{n,r}$ (which exists by our assumption) is a mean for data  $\lim_k u^{(l_k)}_{n,r-1}.$ But this means $\tilde u_{n,r} \not\in \mathrm S u_{n,r-1}$ 
		which is a contradiction.		
		Hence, $f^1_{n,r}(u^{(l)})=1$ for sufficiently large $l,$ and therefore 
		$\liminf_l f^1_{n,r}(u^{(l)}) \geq f^1_{n,r}(u).$ 
		In summary,  $f^1_{n,r}$ and  $f^2_{n,i}$ are
		lower semicontinuous which implies that 
		$\mathcal W_\lambda^0$ is lower  semicontinuous.	
\end{proof}

\begin{theorem}\label{thm:ExistenceStandardProbLambda2neq0}
The variational problem \eqref{eq:VarProblemModelSec} of wavelet sparse regularization using the regularizers 
$\mathcal W_\lambda^{\alpha,q}$ of \eqref{eq:DefWavRegMult} with $\lambda_2 \neq 0$ has a minimizer.
\end{theorem}

\begin{proof}	
	By Lemma~\ref{lem:WalphaQTermLSC}, $\mathcal W_\lambda^{\alpha,q}$ is lower semicontinuous.
	We consider a sequence of signals $u^{(k)},$ and use the notation $\diam(u^{(k)})$ to denote the diameter of $u^{(k)}$.
	We show that $\mathcal W_\lambda^{\alpha,q}(u^{(k)}) \to \infty,$ as $\diam(u^{(k)}) \to \infty.$ 		
	To this end we consider the sums $\sum_{n,i} \dist(\tilde u^{(k)}_{n-e_i,0},\tilde u^{(k)}_{n,0})^q$ in  \eqref{eq:DefWavRegMult}, an note that 
	$\sum_{n,i} \dist(\tilde u^{(k)}_{n-e_i,0},\tilde u^{(k)}_{n,0})^q \geq 
	\left(\max _{n,i} \dist(\tilde u^{(k)}_{n-e_i,0},\tilde u^{(k)}_{n,0}) \right)^q.$
	Further, $\diam(u^{(k)}) \leq  C \max _{n,i} \dist(\tilde u^{(k)}_{n-e_i,0},\tilde u^{(k)}_{n,0})$
	where the constant $C$ is smaller than $|N'|=|N'_1|+\ldots+|N'_s|$ where $N'$ denotes the multiindex containing the dimensions of $u$ on the coarsest scale; cf. the description near \eqref{eq:DefSampling}.
	Hence, $\diam(u^{(k)})$ $\leq  C$ $\sum_{n,i} \dist(\tilde u^{(k)}_{n-e_i,0},\tilde u^{(k)}_{n,0})^q,$ with $C$ independent of $k.$ Therefore, $\diam(u^{(k)})\leq  C \mathcal W_\lambda^{\alpha,q},$
	and so $\mathcal W_\lambda^{\alpha,q}(u^{(n)}) \to \infty,$ as $\diam(u^{(n)}) \to \infty.$
	Together with the lower semicontinuity of $\mathcal W_\lambda^{\alpha,q}$ all requirements of 
	\cite[Theorem 1]{storath2018variational} are fulfilled and we may apply this theorem to conclude the assertion.		
\end{proof}

We note that a statement analogous to Theorem~\ref{thm:ExistenceStandardProbLambda2neq0} does not hold for the
regularizer $\mathcal W_\lambda^{\alpha,q}$ of \eqref{eq:DefWavRegMult} with $\lambda_2 = 0,$
since then coercivity cannot be ensured in general.
To account for that, we give a variant of Theorem~\ref{thm:ExistenceStandardProbLambda2neq0} which imposes some additional constraints to $\mathcal A,$ but then also applies to the situation $\lambda_2 = 0.$

\begin{theorem}\label{thm:ExistenceStandardProbLambda2equals0}
	The variational problem \eqref{eq:VarProblemModelSec} of wavelet regularization using the regularizers 
	$\mathcal W_\lambda^{\alpha,q}$ of \eqref{eq:DefWavRegMult} with $\lambda_2 = 0$ has a minimizer
	provided 
	$\mathcal A$ is an operator such that there is a constant $C>0$ such that, for any signal $u$, the coarsest level of $u$ may be estimated by 
	\begin{equation}\label{eq:AddContraint4WavSp}
		\dist(u_{n-e_i,0},u_{n,0}) \leq C \max (\diam (\mathcal A (u)),R(u)),		
	\end{equation}
	for all $n,i.$ 
	In particular, the problem \eqref{eq:VarProblemModelSec} has a minimizer for $\mathcal A$ being the identity.
	Furthermore, the problem always has a solution, when the manifold $\mathcal M$ is compact.
\end{theorem}

For the proof of Theorem~\ref{thm:ExistenceStandardProbLambda2equals0}
we employ \cite[Theorem 7]{storath2018variational} which we here state for the reader's convenience. 
\begin{theorem}\label{thm:ExistenceCondR4SecondOrd}
	Let $(l_0,r_0),$ $\ldots,$ $(l_S,r_S)$ be $S$ pairs of (a priori fixed) indices.
	We assume that $R$ is lower semicontinuous.
	We further assume that $R$ is a regularizing term such that, 
	for any sequences of signals $u^{(n)},$  
	the conditions $\diam(u^{(n)}) \to \infty$ and $\dist(u^{(n)}_{l_s},u^{(n)}_{r_s}) \leq C,$ for some $C>0$ and for all $n \in \mathbb N$ and all $s \in \{0,S\},$
	imply that $R(u^{(n)}) \to \infty.$  
	If $\mathcal A$ is an imaging operator such that there is a constant $C'>0$ such that, for any signal $u$, $dist(u_{l_s},u_{r_s}) \leq C' \max (\diam (\mathcal A u),R(u)),$ for all $s \in \{0,S\}$,
	then the variational problem 
	\begin{align}\label{eq:TichManiGeneral}
	\argmin_{u \in \mathcal{M}^N} \dist(\mathcal A(u),f)^p + \lambda \  R(u), \qquad \lambda>0,
	\end{align}
	 has a minimizer.
\end{theorem} 
The proof of Theorem~\ref{thm:ExistenceCondR4SecondOrd} may be found in \cite{storath2018variational}.

\begin{proof}[Proof of Theorem~\ref{thm:ExistenceStandardProbLambda2equals0}]
		We apply Theorem~\ref{thm:ExistenceCondR4SecondOrd}. 
		The lower semicontinuity of the $\mathcal W_\lambda^{\alpha,q}$ regularizer is shown in Lemma~\ref{lem:WalphaQTermLSC}. 	
		Towards the other condition of Theorem~\ref{thm:ExistenceCondR4SecondOrd}, 
		let $u^{(k)}$ be  a sequence such that  $\diam(u^{(k)}) \to \infty,$ 
		and such that 
		$\dist(u^{(k)}_{n-e_i,0},u^{(k)}_{n,0}) \leq C' $ for some $C'>0,$ 
		all $n,i$ and all $k \in \mathbb N.$ 
		We have to show that  $\mathcal W_\lambda^{\alpha,q}(u^{(k)})\to \infty.$	
		Towards a contradiction, assume that there is a subsequence $u^{(k_l)}$ of $u^{(k)}$ and $C''>0$ such that $\mathcal W_\lambda^{\alpha,q}(u^{(k_l)}) \leq C'',$ 
		for all $l \in \mathbb N.$ 
		We show that there is a 
		constant $C'''>0$ such that $\dist(u^{(k_l)}_{n-e_i},u^{(k_l)}_{n-e_i}) \leq C''',$
		for all $n$ and $i.$
		Since $\mathcal W_\lambda^{\alpha,q}(u^{(k_l)}) \leq C'',$ and since 
		$\dist(u^{(k_l)}_{n-e_i,0},u^{(k_l)}_{n,0}) \leq C',$ by the definition of the details, there is a constant $C_1$ such that $\dist(u^{(k_l)}_{n-e_i,1},u^{(k_l)}_{n,1}) \leq C_1,$ for all $n,i$ and all $l \in \mathbb N.$ 
		Applying induction on the lever $r$ there is a constant 
		$C_R$ such that $\dist(u^{(k_l)}_{n-e_i,R},u^{(k_l)}_{n,R}) \leq C_R,$ 
		for all $n,i$ and all $l \in \mathbb N.$ Letting $C'''=C_R$ and noting that the finest level signal $u^{(k_l)}_{\cdot,R}$ equals $u^{(k_l)}$ yields that 
		$\dist(u^{(k_l)}_{n-e_i},u^{(k_l)}_{n-e_i}) \leq C''',$
		for all $n$ and $i.$
		This implies that $\diam(u^{(k)})$ is bounded which contradicts 	
		$\diam(u^{(k)}) \to \infty.$
		Hence, $\mathcal W_\lambda^{\alpha,q}(u^{(k)})\to \infty$  
		which allows us to apply Theorem~\ref{thm:ExistenceCondR4SecondOrd}
		and, in turn, to conclude the existence of minimizers.
		
		Concerning the case when $\mathcal A$ is the identity,
		we notice that $\dist(u_{n-e_i,0},u_{n,0}) \leq \diam (u)$
		which implies  the condition \eqref{eq:AddContraint4WavSp}.
		Hence, we may apply the already proved assertion to conclude the existence of minimizers 
		in case $\mathcal A$ is the identity. 
			
	    Concerning the case when the manifold $\mathcal M$ is compact
		we note that the condition 
		on $R$ in Theorem~\ref{thm:ExistenceCondR4SecondOrd}
	    is trivially fulfilled since the boundedness of $\mathcal M$ implies that no sequence in
	    $\mathcal M^N$ can have diameters tending to $\infty.$
	    Hence, we may apply Theorem~\ref{thm:ExistenceCondR4SecondOrd} to conclude the existence of minimizers for compact manifolds $\mathcal M.$
	    \end{proof}

Finally, we give an existence result for the 
variational problem \eqref{eq:VarProblemModelSec} of wavelet sparse regularization 
using $\ell^0$ type regularizing terms $\mathcal W_\lambda^0(u)$.

\begin{theorem}\label{thm:ExistenceL0Sparse}
	The variational problem \eqref{eq:VarProblemModelSec} of wavelet sparse regularization using 
    the $\ell^0$ type regularizing terms $\mathcal W_\lambda^0(u)$ of \eqref{eq:Defl0SparseMult} has a minimizer provided 
    $\mathcal A$ is an imaging operator such that there is a constant $C>0$ such that
    \begin{equation}\label{eq:AddContraint4WavSp0}
    \diam(u) \leq C  \ \diam (\mathcal A (u)).		
    \end{equation}
	In particular, the problem has a minimizer for $\mathcal A$ being the identity.
	Furthermore, the problem always has a solution, when the manifold $\mathcal M$ is compact.
\end{theorem}

\begin{proof}
	We consider the functional 
	$\mathcal F (u) := \dist(\mathcal A(u),f)^p +  \mathcal W_\lambda^0(u).$	 
	The regularizing term $\mathcal W_\lambda^0(u)$ is lower semicontinuous by Lemma~\ref{lem:WnullQTermLSC} 
	 and the data term is lower semicontinuous
	by \cite[Lemma 3]{storath2018variational}. 
	Together, the energy $\mathcal F$ is lower semicontinuous.
	Assuming the condition \eqref{eq:AddContraint4WavSp0},
	we show the coercivity of $\mathcal F,$
	 i.e.,
	 we show that, for $\sigma \in \mathcal M^N$ and any sequence $u^{(k)}$ in $\mathcal M^N,$
	 $\dist(u^{(k)},\sigma) \to \infty$ as  $k \to \infty$ implies $\mathcal F(u^{(k)}) \to \infty$ 
	 as $k \to \infty$.
	 	Towards a contradiction suppose that $\mathcal F$ is not coercive. 
	 	Then there is $\sigma \in \mathcal M^N$ and a sequence $u^{(k)}$ in
	 	$\mathcal M^N,$ such that  
	 	$\dist(u^{(k)},\sigma) \to \infty$
	 	and such that $\dist(\mathcal A(u^{(k)}),f)^p$ is bounded
	 	(by passing to a subsequence if necessary.)
	 Then $\diam (\mathcal A (u^{(k)}))$ is bounded.	 
	 By assumption \eqref{eq:AddContraint4WavSp0}, $\diam (u^{(k)})$ is then bounded as well,
	 i.e., there is a constant  $C'>0$ such that $\diam (u^{(k)})< C'$ for all $k \in \mathbb N.$

	 Since $\dist(u^{(k)},\sigma) \to \infty$ and $\diam (u^{(k)})$ is bounded we find for any given (ball radius) $C''>C'$ a subsequence (indexed by $k_l$) of closed, disjoint balls $B_l$ such that  
	 all members of $u^{(k_l)}$ are contained in $B_l,$ i.e.,
	 \begin{align}\label{eq:DisjointBall4L0}
		 \left\{u^{(k_l)}_i : i = 1,\ldots,N \right\} \subset B_l,
	 	 \quad B_l \cap B_{l'} = \emptyset \quad \text {for all } l,l' \in \mathbb N. 
	 \end{align}
	 We next define a suitable (ball radius) $C''.$ By \cite[Lemma 2]{storath2018variational},
     since $u^{(k_l)}$ is contained in a ball of radius $C',$ then there is a constant $L\geq 1$(which depends on the weights in $\mathcal A$) such that $\mathcal A (u^{(k)})_j$ is contained in  
     in a ball of radius $L C'.$ Accordingly, we choose $C'' := L \ C'$ and consider  
     the balls $B_l$ of \eqref{eq:DisjointBall4L0} of radius $C''.$
     Then,
     \begin{align}\label{eq:DisjointBall4L0}
     \left\{\mathcal A(u^{(k_l)}_j) : j = 1,\ldots,K \right\} \subset B_l,
      \quad B_l \cap B_{l'} = \emptyset \quad \text {for all } l,l' \in \mathbb N. 
     \end{align}  
     As a result, $\dist(\mathcal A(u^{(k_l)}),f)^p \to \infty$ which contradicts the assumption.
	Hence, the energy is coercive which, together with its lower semicontinuity, yields the existence of minimizers.
	
	If $\mathcal A$ is the identity,
	the condition \eqref{eq:AddContraint4WavSp0} is obviously fulfilled.
	Hence, we may apply the already proved assertion to conclude the existence of minimizers. 
	If the manifold $\mathcal M$ is compact, the functional is automatically coercive which together 
	with its lower semicontinuity already observed above, yields the existence of minimizers
	for compact manifolds.	
\end{proof}

\section{Algorithms}\label{sec:algorithm}

In the following we derive algorithms for the manifold valued variational problem 
\eqref{eq:VarProblemModelSec} of wavelet regularization using the regularizers 
$\mathcal W_\lambda^{\alpha,q}(u)$ of \eqref{eq:DefWavRegMult}
and 
$\mathcal W_\lambda^0(u)$ of \eqref{eq:Defl0SparseMult}.
We consider generalized forward-backward as well as cyclic proximal point algorithms. 
These basic algorithmic structures are presented and applied to the considered variational problems in Section \ref{sec:AlgoStr}. 
The differential geometric details to implement these structures for the present problems are given in 
Section~\ref{sec:DiffGeoDerivations}.
In particular, we there explain how to compute the derivatives and proximal mappings of the atoms of the involved wavelet regularizers.

\subsection{Algorithmic Structures}\label{sec:AlgoStr}

We denote the functional in \eqref{eq:VarProblemModelSec} by $\mathcal F$ and decompose both the data term $\mathcal D$ and the regularizer $\mathcal W_\lambda(u)=\mathcal W_\lambda^{\alpha,q}(u)$ or
$\mathcal W_\lambda(u)=\mathcal W_\lambda^0(u)$ into atoms  $D_i$ and  $W_{n,r},$ i.e., we let
\begin{align}
\mathcal F(u) = \mathcal D (u) + \ \mathcal W_\lambda(u)  =  \sum\nolimits_{i=1}^K  D_i (u) +  \lambda \sum\nolimits_{r=0}^R \sum\nolimits_{n}  W_{n,r}(u)
\end{align}
with 
\begin{align}\label{eq:DataTermDi}
 D_i (u):=   \dist(\mathcal A(u)_i,f_i)^p, 	
\end{align}
and, for level $r \geq 1$, 
\begin{align}\label{eq:SplittingGeneralLevel}
 W_{n,r}(u) = 
 \begin{cases}
 \lambda_1 \cdot 2^{r q \ \left(\alpha+\tfrac{s}{2}-\tfrac{s}{q} \right)}\| d_{n,r}(u) \|^q_{\hat u_{n,r}}, &  
 \text{ if } q \geq 1, \\
 \lambda_1  \ \# \  \{d_{n,r}(u) \neq 0\}, 
 &  \text{ if }  q = 0.
 \end{cases}
\end{align}
For level $r=0,$ we consider extended multiindices $\tilde n = (n,i)$ and use the notation
\begin{align}\label{eq:SplittingZeroLevel}
W_{\tilde n,0}(u) = W_{(n,i),0}(u) = 
\begin{cases}
\lambda_2 \cdot  \dist(\tilde u_{n-e_i,0},\tilde u_{n,0})^q 
&  \text{ if } q \geq 1 \\
\lambda_2 \ \# \  \{\tilde u_{n-e_i,0} \neq \tilde u_{n,0}  \} 
&  \text{ if }  q = 0
\end{cases}
\end{align}
We note that the meaning of the employed symbols is explained near \eqref{eq:DefWavRegMult} and
\eqref{eq:Defl0SparseMult}.
We will omit the tilde notation of \eqref{eq:SplittingZeroLevel} in the following keeping in mind that the index $n$ in $W_{n,0}$ encodes a zeroth level index together with a direction whereas the $n$ in $W_{n,r},$ $r \geq 1,$ encodes an $r$th level index.

In the following we will employ the above decomposition into atoms within the proposed algorithms. 
More precisely, for each atom we will either apply a gradient step (which is possible for differentiable terms) or its proximal mapping within an iterative scheme.
Here, the  proximal mappings \cite{moreau1962fonctions, ferreira2002proximal, azagra2005proximal} of a function $f$ on a manifold  $\mathcal{M^N}$ is given by
\begin{align} \label{eq:prox_mapping_abstract}
\prox_{\mu f} x = \argmin_y f(y) + \frac{1}{2 \mu} \dist(x,y)^2, \qquad \mu > 0.  
\end{align}
For general manifolds, the proximal mappings \eqref{eq:prox_mapping_abstract} are not globally defined, and the minimizers are not unique in general, at least for possibly far apart points; cf. \cite{ferreira2002proximal, azagra2005proximal}.
This is a general issue in the context of manifolds that are -- in a certain sense -- a local concept
involving objects that are locally well defined. In case of ambiguities, a possibility is to consider the above objects as set-valued quantities.

\paragraph{Generalized Forward Backward Scheme.} 

In \cite{baust2016combined}, we have proposed a generalized forward backward algorithm for DTI data with a voxel-wise indirect data term. 
In \cite{storath2018variational}, we have employed these schemes for the variational regularization of inverse problems for manifold-valued signals. To improve the method, we have proposed a variant based on a trajectory method together with a Gau\ss-Seidel type 
update strategy as well as a stochastic variant of the generalized forward-backward scheme there.
We explain the ideas and apply the schemes to our regularizers.

The basic approach using a generalized forward-backward scheme is to decompose the considered functional $\mathcal F$ into two summands  $\mathcal F(u) = \mathcal F^1(u) +  \mathcal F^2(u)$ 
where the one summand $\mathcal F^1$ is differentiable, and the  other summand $\mathcal F^2$
is further decomposed into atoms via $\mathcal F^2(u) = \sum_k \mathcal F_k^2(u).$ 
The purpose is to find a decomposition such that, for each of these atoms $\mathcal F_k^2$, 
one has rather simple means to compute the proximal mappings.
A generalized forward-backward scheme then performs an (explicit) gradient step for the  differentiable term $\mathcal F^1$, as well as an (implicit) proximal mapping step for each atom $\mathcal F_k^2$ of the term $\mathcal F^2.$
One instantiation in our situation of wavelet regularization is 
\begin{align}\label{eq:IntstGFB1}
  \mathcal F^1 = \mathcal D \quad \text{ (with $p>1$) }, \quad  
  \text{and} \quad 
  \mathcal F^2 = 
  \begin{cases}
  \mathcal W_\lambda^{\alpha,q}(u),   \ q\geq 1,  \ \text{ or }  \\
  \mathcal \mathcal W_\lambda^0(u),\\ 
  \end{cases}
  \text{ with atoms }	
  W_{n,r},
\end{align}
where the $ W_{n,r}$ are given by \eqref{eq:SplittingGeneralLevel} and \eqref{eq:SplittingZeroLevel}, respectively.
We note that, for $p>1$,  the data terms $\mathcal D$ are differentiable.
Another possible instantiation in our situation is 
\begin{align}\label{eq:IntstGFB2}
\mathcal F^1 = 
\mathcal W_\lambda^{\alpha,q}(u),   \quad \text{ (with $q>1$) }, \quad  
\text{and} \quad 
\mathcal F^2 = \mathcal D \quad \text{ (with $p\geq 1$) }, 
\text{ with atoms }	
D_i, 
\end{align}
where the $D_i$ are given by \eqref{eq:DataTermDi}.

If the differentiable term $\mathcal F^1$ also allows for an additive decomposition of the form 
$\mathcal F^1(u) = \sum_{k'} \mathcal F_{k'}^1(u)$ (which is the case in our examples above)
the computation of the gradient of $\mathcal F^1$ may be interpreted as 
computing the gradients of the atoms $\mathcal F_{k'}^1(u)$ and then to apply 
a Jacobi type update. Instead, one may use a Gau\ss-Seidel type 
update strategy instead of the Jacobi type update.
This means that the result of computing the gradient of $\mathcal F_{k'}^1$ is already used as an iterate when computing the gradient of $\mathcal F_{k'+1}^1.$
The advantage of the Gau\ss-Seidel type 
update strategy is that it is amenable for a trajectory method which is crucial for avoiding unreasonably small step sizes of the overall algorithm which limits its performance.
(For a detailed discussion of these step size issues and its resolution we refer 
to the authors' previous work \cite{ storath2018variational}.)
Instead of computing only one (potentially large) gradient step,
the trajectory method applied to an atom $\mathcal F_{k'}^1$ computes 
several smaller steps. More precisely, we 
define the trajectory operator $\traj_\mu \mathcal F_{k'}^1$ 
applied to the atom $\mathcal F_{k'}^1$ 
for input $x_0$ (which may be thought of as the result for $\mathcal F_{k'-1}^1$) 
as the output $x$ of the following algorithm, i.e.,
$x=\traj_\mu \mathcal F_{k'}^1(x_0)$ is given by  
\begin{align}\notag
\text{ Iterate} & \text{ w.r.t.\ $r$ until $\tau \geq 1$ :} \\
& x_r :=   x_{r-1} +   \tau_{r-1}  \mu \nabla \mathcal F_{k'}^1(x_{r-1}); \quad  \label{eq:TrajMeth} 
\tau :=  \sum\nolimits_{l=0}^{r-1} \tau_{l} 
\\
x=   x_{r}&_{-1}  +   \left(1 - \sum\nolimits_{l=0}^{r-2} \tau_{l}\right)  \mu \nabla \mathcal F_{k'}^1(x_{r-1}) \notag
\end{align}  
Here, $\tau_{r-1}$ is a predicted step size for the gradient step at $x_{r-1}$
(which may be realized by a line search yielding a parameter with minimal functional value along the line.)
Further $\mu>0$ is a parameter.
This scheme is inspired by solving initial value problems for ODEs.   
Instead of using a straight line we follow a polygonal path normalized by evaluating it at ``time'' 
$\tau=1$. 
Using this trajectory method instead of a gradient step, we obtain the following
generalized forward backward algorithm with Gau\ss-Seidel type update scheme.
\begin{align}\notag
\text{ Iterate} & \text{ w.r.t.\ $n$ :} \\
1. &\text{ Compute } u^{(n+k'/2K')} =  \traj_{\mu_n} \mathcal F_{k'}^1 \left(u^{(n + (k'-1)/2K')}\right) \qquad  \text{ for all $k' =1,\ldots,K'$ }; \label{eq:algFBwithGSupdateAndTraj} \\  
2.  &\text{ Compute } u^{(n+0.5+ k/2K'')} =  \prox_{\mu_n \lambda \mathcal F_{k}^2} u^{(n+0.5+ (k-1)/2K'')}, \qquad  \text{ for all $k =1,\ldots,K''.$} \qquad \qquad \notag 
\end{align}

As explained above $\traj_\mu \mathcal F_{k'}^1(\cdot)$ denotes the application of the trajectory method defined by \eqref{eq:TrajMeth}.
During the iteration, the positive parameter $\mu_n$ is decreased in a way such that $\sum_n \mu_n = \infty$ and  
such that $\sum_n \mu_n^2 < \infty.$ 

 We get the generalized forward backward algorithm for wavelet regularization (including wavelet sparse $\ell^0$ and $\ell^1$ regularization) for data terms with $p>1$ by the scheme \eqref{eq:algFBwithGSupdateAndTraj} with 
\begin{align}\label{eq:IntstGFB1developed}
\mathcal F_{k'}^1 = D_{k'} \quad \text{ (with $p>1$) }, \quad  
\text{and} \quad 
\mathcal F_{k}^2 =  \mathcal F_{k(n,r)}^2 = W_{n,r},
\end{align}
where the $D_{k'}$ are given by \eqref{eq:DataTermDi}, and
where the $ W_{n,r}$ are given by \eqref{eq:SplittingGeneralLevel} and \eqref{eq:SplittingZeroLevel}, respectively.
The index mapping $k=k(n,r)$ serializes the pyramid like data of the wavelet scheme.
Equation \eqref{eq:IntstGFB1developed} is an adaption of the instantiation \eqref{eq:IntstGFB1} above.
The corresponding adaption \eqref{eq:IntstGFB2} above is to apply the scheme \eqref{eq:algFBwithGSupdateAndTraj} with 
\begin{align}\label{eq:IntstGFB2developed}
\mathcal F_{k'}^1 = F_{k'(n,r)}^1 =  W_{n,r},   \quad \text{ (with $q>1$) }, \quad  
\text{and} \quad 
\mathcal F_{k}^2 = D_{k} \quad \text{ (with $p\geq 1$) }. 
\end{align}
As above, where the $D_k$ are given by \eqref{eq:DataTermDi}, 
the $ W_{n,r}$ are given by \eqref{eq:SplittingGeneralLevel} and \eqref{eq:SplittingZeroLevel},
and $k'=k'(n,r)$ is the index mapping which serializes the pyramid like data of the wavelet scheme.

\paragraph{The Cyclic Proximal Point Scheme.}

Further, we consider proximal point algorithms which may also be used in the case $p=1$ and $q=1.$ 
A reference for cyclic proximal point algorithms in vector spaces is \cite{Bertsekas2011in}.
In the context of Hadamard spaces, the concept of CPPAs was developed by \cite{bavcak2013computing}, where it is used to compute means and medians.
In the context of variational regularization methods for nonlinear, manifold-valued data,
they were first used in \cite{weinmann2014total} and then later in various variants in 
\cite{bergmann2014second,bavcak2016second,bredies2017total}. 
The basic idea of CPPAs is to compose the considered functional $\mathcal F$ into 
basic atoms and then to compute the proximal mappings of each of the atoms in a cyclic, iterative way.
In the above notation we have the algorithm
\begin{align}\notag
\text{ Iterate} & \text{ w.r.t.\ $n$ :} \\
1. &\text{ Compute } u^{(n+k'/2K')} =  \prox_{\mu_n \mathcal F_{k'}^1}u^{(n + (k'-1)/2K')} \qquad  \text{ for all $k' =1,\ldots,K'$ }; \label{eq:algCPPA} \\  
2.  &\text{ Compute } u^{(n+0.5+ k/2K'')} =  \prox_{\mu_n \lambda \mathcal F_{k}^2} u^{(n+0.5+ (k-1)/2K'')}, \qquad  \text{ for all $k =1,\ldots,K''.$} \qquad \qquad \notag 
\end{align} 

As above, the parameters $\mu_n$ are chosen such that $\sum_n \mu_n = \infty$ and such that $\sum_n \mu_n^2 < \infty.$ 

We implement the cyclic proximal point algorithm for wavelet regularization by using  \eqref{eq:IntstGFB1developed} in \eqref{eq:algCPPA}.
We note that we in particular use \eqref{eq:algCPPA} when both the exponent $p$ of the power of the distance in the data term $\mathcal D$ equals one and the exponent $q$ in $\mathcal W_\lambda^{\alpha,q}(u)$ of \eqref{eq:DefWavRegMult} equals one, or 
when we consider $p=1$ together with the regularizer
$\mathcal W_\lambda^0(u)$ of \eqref{eq:Defl0SparseMult},i.e.,
when we consider the analogues of wavelet $\ell^1$ or $\ell^0$ sparse  regularization 
with an $\ell^1$ type data term.

Often the considered functionals are not convex; hence the convergence to a globally optimal solution cannot be ensured. Nevertheless, as will be seen in the numerical experiments section, we experience a good convergence behavior in practice.
This was also observed in previous works such as  \cite{bergmann2014second,bavcak2016second} where the involved manifold valued functionals are not convex either.

\paragraph{Potential for Parallelization.}
We note that the algorithms given by the application of \eqref{eq:IntstGFB1developed} and \eqref{eq:IntstGFB2developed} in 
\eqref{eq:algFBwithGSupdateAndTraj} and \eqref{eq:algCPPA} are in many situations massively parallelizable with minor modifications.
Concerning the wavelet terms 
$\mathcal W_\lambda^{\alpha,q}(u)$ of \eqref{eq:DefWavRegMult}
and 
$\mathcal W_\lambda^0(u)$ of \eqref{eq:Defl0SparseMult}
which the $W_{n,r}(u)$ are based on, we notice that each item $u$ directly contributes to at most two scales $r,$ and $r+1.$ On the finer scale $r+1,$ it only contributes to one member $d_{n,r+1}.$ 
Further, on the coarser scale $r$ it directly contributes to at most $A$  members $d_{n,r+1},$
where $A$ is the number of elements in the support of the mask $a$ of the subdivision scheme $\mathrm{S}$
 given in \eqref{eq:defSubdivOperator}.
Hence, any two $W_{n,r},$ $W_{n',r'}$ with $|r-r'|>1$  may be computed in parallel.
Further,  $W_{n,r}, W_{n',r}$ (which corresponds to to the case $r'=r$) can be computed in parallel
whenever $|n-n'| \geq A'$ where $A'$ is the smallest integer such the the support of the mask  $a$ 
is contained in $\{-A',\ldots, A'\}^s,$ where $s$ is the dimension of the multivariate domain.
Finally, $W_{n,r}, W_{n',r+1}$ (which corresponds to to the case $r'=r+1$) can be computed in parallel
whenever $|n'-2n|\geq A',$ with the same $A'$ as in the previous line.
The computations for the data term $\mathcal D$ is parallelizable as well;
for details we refer to the corresponding discussion in \cite{storath2018variational}.  
The resulting minor modification in the algorithms 
resulting from the parallelization of the wavelet term and the data term 
then consists of another order of applying the operations 
$\traj_{\mu_n} \mathcal F_{k'}^1$
and
$\prox_{\mu_n \lambda \mathcal F_{k}^2}, $ 
or the operations 
$\prox_{\mu_n \mathcal F_{k'}^1}$
and
$\prox_{\mu_n \lambda \mathcal F_{k}^2}, $ 
within the cycle $n.$

\subsection{Differential geometric derivation of the relevant gradients and the proximal mappings.}
\label{sec:DiffGeoDerivations}

We here provide differential geometric expressions for the gradients and proximal mappings of the atoms
derived and employed in the algorithms of Section~\ref{sec:AlgoStr}.
In particular, we provide a derivation of the gradients and proximal mappings of the 
atoms of the wavelet regularizers 
$\mathcal W_\lambda^{\alpha,q}(u)$ of \eqref{eq:DefWavRegMult}
including the analogue of $\ell_1$ sparse wavelet regularization. 
Further, for the atoms of 
$\mathcal W_\lambda^0(u)$ of \eqref{eq:Defl0SparseMult},
which is the analogue of $\ell_0$ sparse wavelet regularization,
we provide means to compute the proximal mappings of the corresponding atoms.

\paragraph{Gradients of the atoms of the wavelet regularizers $\mathcal W_\lambda^{\alpha,q}(u)$.}

We here derive expressions for the gradients of the atoms $W_{n,r}$ of $\mathcal W_\lambda^{\alpha,q}$ of \eqref{eq:DefWavRegMult}
given in \eqref{eq:SplittingGeneralLevel} and \eqref{eq:SplittingZeroLevel}.

We first consider the zeroth level $n=0,$ and see how to compute the gradient of the corresponding mappings given by 
\begin{align}\label{eq:atomWavExplCalc}
u' \mapsto W_{\tilde n,0}(u') = W_{(n,i),0}(u') = 
\lambda_2 \cdot  \dist(\tilde u'_{n-e_i,0},\tilde u'_{n,0})^q.
\end{align} 
Hence, we have to compute the gradient of the $p$th power of the distance mapping 
$
d: y \mapsto \dist (y,f)^p,
$
$f \in \mathcal M.$
This gradient is given by (cf. \cite{afsari2011riemannian})  
$
\nabla d(y) =  \tfrac{1}{q}\|\exp^{-1}_y(f)\|^{q-2} \exp^{-1}_y(f).
$
Applied to our situation in \eqref{eq:atomWavExplCalc}, we have
\begin{align}\label{eq:tgv_manifoldParallelUnivAlgoProxDist}
\nabla_{W_{\tilde n,0}} (u)_{n',0}=  
\begin{cases}
\tfrac{1}{q}\|\exp^{-1}_{\tilde u_{n-e_i,0}}(\tilde u_{n,0})\|^{q-2} 
\exp^{-1}_{\tilde u_{n-e_i,0}}(\tilde u_{n,0}),    &\text{if } n' = n-e_i,  \\
\tfrac{1}{q}\|\exp^{-1}_{\tilde u_{n,0}}(\tilde u_{n-e_i,0})\|^{q-2} 
\exp^{-1}_{\tilde u_{n,0}}(\tilde u_{n-e_i,0}),    &\text{if } n'= n,    \\
u_{n'}  ,              &\text{else}.  
\end{cases}
\end{align}

Concerning a general level $r \neq 0,$ according to \eqref{eq:SplittingGeneralLevel}, we have to compute the gradients of the mappings  
\begin{align}\label{eq:proxWavGenLevFun}
u \mapsto W_{n,r}(u) = 
\lambda_1 \cdot 2^{r q \ \left(\alpha+\tfrac{s}{2}-\tfrac{s}{q} \right)} \ \| d_{n,r}(u) \|^q_{\hat u_{n,r}}.   
\end{align} 
$q \geq 1.$ 
By the definition of $d_{n,r}$ in \eqref{eq:DefDnr} and the properties of the exponential mapping we have that 
\begin{equation}
\| d_{n,r}(u) \|^q_{\hat u_{n,r}} =  2^{-s r/2} \ \dist(\mathrm{S} \tilde u_{n,r-1},\tilde u_{n,r})^q =
2^{-s r/2} \ \dist(\mean(a_{n-2\cdot}, \tilde u_{\cdot,r-1}),\tilde u_{n,r})^q.
\end{equation}
Hence, in order to obtain the gradient of the mapping in \eqref{eq:proxWavGenLevFun} w.r.t.\ the argument $u,$
we have to derive the gradients of the mappings 
\begin{align}
  u_l \mapsto  W_{n,r}(u) = \lambda_1 \cdot 2^{r q\ \left(\alpha+\tfrac{s}{2}-\tfrac{s}{q}\right)} 2^{-s r q/2}
  \dist(\mean(a_{n-2\cdot},\tilde u_{\cdot,r-1}),\tilde u_{n,r})^q,     	
\end{align}
for all indices $l$ of $u.$ For the index corresponding to $\tilde u_{n,r},$
i.e., the index $l_0$ for which $u_{l_0} = \tilde u_{n,r},$ we get by applying the above formula for the gradient of the distance function that
\begin{equation}\label{eq:derWnrfornr}
	\nabla_{W_{n,r}} (u)_{n,r} = 
	\tfrac{\lambda_1}{q} \cdot 2^{r q \ \left(\alpha-\tfrac{s}{q}\right)} 
	\|\exp^{-1}_{\tilde u_{n,r}}\mean(a_{n-2\cdot}, \tilde u_{\cdot,r-1}) \|^{q-2} 
	\exp^{-1}_{\tilde u_{n,r}} \mean(a_{n-2\cdot}, \tilde u_{\cdot,r-1})   ,    
	\end{equation}
Here, the lower index ${n,r}$ denotes the component of the gradient corresponding to the variable 
$\tilde u_{n,r}$ which corresponds to the  $l_0$th index of $u.$

As the major task, we have to compute the gradient of $W_{n,r}$ w.r.t.\ the variables $\tilde u_{\cdot,r-1}.$
To this end, we notice that the mapping  $ \tilde u_{k_0,r-1} \mapsto W_{n,r}(u)$ is the concatenation of the mean mapping
\begin{align}\label{eq:MapToDeriveGrad}
M: \tilde u_{k_0,r-1} \mapsto \mean(a_{n-2\cdot}, \tilde u_{\cdot,r-1}) = \argmin_{x \in \mathcal M}  \ \sum\nolimits_k  a_{n-2k} \ \dist(x,\tilde u_{k,r-1})^2 
\end{align}
and the $q$th power of the distance mapping. Hence, by the chain rule and the rules of transposition,
the gradient of $W_{n,r}$ w.r.t.\ $\tilde u_{k_0,r-1}$ is given by the application of the
adjoint of the differential of the mapping $M$ of \eqref{eq:MapToDeriveGrad}
to the gradient of the distance function already presented above. 
So, let us consider the adjoint of the differential of the mapping $M$ of \eqref{eq:MapToDeriveGrad}.
The point here is that the mapping $M$ may not be written down explicitly, since the function value itself is the minimizer of a corresponding minimization problem which cannot be explicitly solved. However, we have found a rather explicit representation of the differential of the mean mapping (and its adjoint) as a function of the points in \cite{storath2018variational}. We here provide some explanation and notation and recall the corresponding statement in the notation of the present work then.
We use the fact that every (weighted) mean $M(x_1,\ldots,x_K)$ of $K$ points $x_1,\ldots,x_K,$ is a zero of the mapping $\mathcal W,$ i.e.,
\begin{align}\label{eq:ZeroVecField}
\mathcal W (x_1,\ldots,x_K,M(u_1,\ldots,u_K)) = 0, 	
\end{align} 
where the mapping $\mathcal W: \mathcal M ^{K+1} \to \mathcal{TM}$ is defined as 
$
\mathcal W (x_1,\ldots,x_K,m) = \sum_{j=1}^K a_{i-2j}  exp^{-1}_m(x_j). 	
$
Here, $\mathcal{TM}$ denotes the tangent bundle of $\mathcal M.$
As a mapping of the argument $m,$ $\mathcal W$ is a tangent vector field.
(If there is a unique mean $M(x_1,\ldots,x_K)$ of $x_1,\ldots,x_K,$  \eqref{eq:MapToDeriveGrad} it is characterized by \eqref{eq:ZeroVecField}. In case of non-uniqueness, the characterization holds at least locally.)
In \cite{storath2018variational}, we use the notation $\mathcal W'$ for the function
\begin{align}\label{eq:ZeroVecFieldDerRightHs}
\mathcal W' : (x_1,\ldots,x_K) \mapsto \mathcal W (x_1,\ldots,x_K,M(x_1,\ldots,x_K)).
\end{align} 
and calculate the derivative of $\mathcal W',$ which equals zero in view of \eqref{eq:ZeroVecField} and so obtain a representation of the derivative of $M.$
To formulate this representation we need the linear mappings $R_{k_0}$ and $L,$ which are defined in terms of Jacobi fields. For defining $R_{k_0}$ we consider
the Jacobi fields along the geodesic 
$\gamma$ connecting $\gamma(0) =m= \mean(a_{n-2\cdot}, \tilde u_{\cdot,r-1}) $ and 
$\gamma(1)= \tilde u_{k_0,r-1}.$
As intermediate step, the mapping $r_{k_0}$ is given by the boundary to initial value mapping 
\begin{align}\label{eq:Littlerjay}
r_{k_0}: \quad    \mathcal{TM}_{\tilde u_{k_0,r-1}} \to \mathcal{TM}_{m}, \quad J(1) \mapsto \frac{D}{dt} J (0),     
\end{align}
where the $J$ are the Jacobi fields with $J(0)=0$ which parametrize $\mathcal{TM}_{\tilde u_{k_0,r-1}}$ via the point evaluation mapping $J \mapsto J(1).$  
We note that this mapping is well-defined for non conjugate points which is the case for close enough points.
We then let 
\begin{align}\label{eq:DefRk}
R_{k_0} w :=  a_{n-2{k_0}} 	r_{k_0}  w,    
\end{align}
for tangent vectors $w$ sitting at the point $\tilde u_{k_0,r-1}.$
Next, we define the mapping $L.$
To this end, let $\gamma_k$ be the geodesic connecting 
$\gamma_k (0) = m=\mean(a_{n-2\cdot}, \tilde u_{\cdot,r-1})$ 
and $\gamma_k(1)= \tilde u_{k,r-1} ,$ for the indices $k$ 
in the $r-1$th level for which the weight is nonzero.
For each geodesic $\gamma_k,$ we consider the Jacobi fields $J_k$ with $J_k(1) = 0$
and the mappings
\begin{align}\label{eq:littlelJay}
l_k:\    \mathcal{TM}_{m} \to \mathcal{TM}_{m}, \quad J_k(0) \mapsto \frac{D}{dt} J_k (0)
\end{align}
where the Jacobi fields $J_k$ (with $J_k(1)=0$) parametrize $\mathcal{TM}_{m}$ 
via the evaluation map $J_k \mapsto J_k(0).$  
Again, we note that this mapping is well-defined for non conjugate points which is the case for close enough points. Then, we let
\begin{align}\label{eq:DefL}
L \ v := \sum\nolimits_k a_{n-2k} \ l_k \ v,
\end{align}
for tangent vectors $v$ in the point $m=\mean(a_{n-2\cdot}, \tilde u_{\cdot,r-1}).$
We have now gathered all information to formulate \cite[Theorem 11]{storath2018variational} 
adapted to the notation of the present work.
\begin{theorem}\label{thm:ComputeDerivativeOfM}
	The derivative of the intrinsic mean mapping w.r.t. the input variable $\tilde u_{k_0,r-1}$ in direction $w$ at $\mathcal{TM}_{\tilde u_{k_0,r-1}}$ is given by  
	\begin{align}\label{eq:thmEq2}
	\partial_{\tilde u_{k_0,r-1}}M(u) \ w  =  - L^{-1} R_{k_0} w
	\end{align}	 
	where the linear mapping $L$ is given by \eqref{eq:DefL},
	and the linear mappings $R_{k_0},$ is given by \eqref{eq:DefRk}. 	 
\end{theorem}
Combining these results with \eqref{eq:derWnrfornr} we can descibe the gradients of the mappings $W_{n,r}$
rather explicitly. We have  
\begin{equation}\label{eq:derWnrfornr}
\nabla_{W_{n,r}} (u)_{n',r'} =
\begin{cases}
s \ \frac{ 
\exp^{-1}_{\tilde u_{n,r}} \mean(a_{n-2\cdot}, \tilde u_{\cdot,r-1})}
{\|\exp^{-1}_{\tilde u_{n,r}}\mean(a_{n-2\cdot}, \tilde u_{\cdot,r-1}) \|^{2-q}}   ,    
& \text{if } r' = r,n' = n, \\
s \ R_{n'}^\ast   L^{-1 \ast} \
\frac{
\exp^{-1}_{\mean(a_{n-2\cdot}, \tilde u_{\cdot,r-1})}{\tilde u_{n,r}}}
{\|\exp^{-1}_{\mean(a_{n-2\cdot}, \tilde u_{\cdot,r-1})} {\tilde u_{n,r} \|^{2-q}}}
,    
& \text{if }  r' = r-1, n-2n' \in \supp a, \\
0 & \text{else,}
\end{cases} 
\end{equation}
where 
\begin{equation}
	s = \tfrac{\lambda_1}{q} \cdot 2^{r \ \left(\alpha-\tfrac{s}{2}-\tfrac{s}{q}\right)} 2^{r/2}.
\end{equation}
Here, $R_{n'}^\ast$ denotes the adjoint of $R_{n'}$ given by \eqref{eq:DefRk},
$L^{-1\ast}$ denotes the adjoint of $L^{-1} $ given by \eqref{eq:DefL},
the  symbol $\supp$ is used to denote the support of the mask $a$ of the subdivision scheme,
and the condition $n-2n' \in \supp a$ means that $u_{r-1,n'}$ contributes to the computation of  
$\hat u_{n,r},$ i.e., the corresponding weight $a_{n-2n'}$ for averaging is nonzero.

If the manifold $\mathcal M$ is a Riemannian symmetric space, the Jacobi fields 
needed to compute the mappings $R_{j_0},$ 
and $L$ in the above theorem can be made more explicit.
This is as well discussed in \cite{storath2018variational},
and we here only state the corresponding results for completeness.
In particular, we observe that basically a series of low dimensional eigenvalue problems
(instead of ODEs) has to be solved. 
As a reference on symmetric spaces we refer to \cite{cheeger1975comparison}.
We consider the geodesic $\gamma$ and the Jacobi fields $J$ with $J(0)=0$ along $\gamma.$
Then the numerical task is to compute an orthonormal basis $(w^i)_i$ of eigenvectors together with corresponding eigenvalues $(\lambda_i)_i$ of the self-adjoint Jacobi operator 
$J \mapsto R(\tfrac{\gamma'(0)}{\|\gamma'(0) \|},J)\tfrac{\gamma'(0)}{\|\gamma'(0) \|},$ 
where $R$ denotes the Riemannian curvature tensor.
In terms of this eigen-decomposition the adjoint $R_{k}^\ast$ can then be written using the functional calculus
together with a spectral mapping theorem, i.e., $R_{k}^\ast$ is given by  
\begin{equation}\label{eq:DiffOfLogInSym}
w= \sum\nolimits_i \alpha_i  w^i \ \mapsto \ R_{k}^\ast w =  a_{n-2k}\sum\nolimits_i \alpha_i f_1(\lambda_i) \ 
\pt_{m,\tilde u_{k,r-1}} w^n,
\end{equation}
where 
$f_1(\lambda_i) = 1$ if $\lambda_i = 0,$
$f_1(\lambda_i) = \frac{\sqrt{\lambda_i} d}{\sin (\sqrt{\lambda_i} d)},$ 
if $\lambda_i > 0,$ and $d < \pi/\sqrt{\lambda_i},$
and where 
$f_1(\lambda_i) = 
\frac{\sqrt{-\lambda_i} d}{\sinh (\sqrt{-\lambda_i} d)},$  if $\lambda_i < 0,$
with $d= \dist(m,\tilde u_{k,r-1}).$ 
Here, $\pt_{m,\tilde u_{k,r-1}} w^i$ denotes the parallel transport of the basis vector $w^i$ 
from the point $m = \mean(a_{n-2\cdot}, \tilde u_{\cdot,r-1}) $ to the point $\tilde u_{k,r-1}.$ 
(For the parallel transport, there are typically closed form expressions in a symmetric space.)
Similarly, but even without parallel transport, we may compute 
the adjoints $l_k^\ast$ defined by \eqref{eq:littlelJay}
which yield the adjoint $L^\ast$ via
$L^\ast \ v := \sum_k a_{n-2k} \ l^\ast_k \ v,$
cf. \eqref{eq:DefL}.
We consider the geodesics $\gamma_k$ we used above for \eqref{eq:littlelJay} to connect $m=\gamma(0)$ and $\tilde u_{k,r-1}=\gamma(1),$
and the orthonormal basis $(w^n)_n$ of eigenvectors of the self-adjoint Jacobi operator $R$ w.r.t.\ the geodesic $\gamma_k$ and the above Jacobi fields $J$ with $J(1)=0.$
We get, for $l_{k}^\ast: \mathcal{TM}_{m} \to \mathcal{TM}_{m},$
\begin{align}\label{eq:ComputeljotIn}
l_k^\ast :  \ w= \sum\nolimits_i \alpha_i  w^i  \mapsto  l_k^\ast (w) = \sum\nolimits_i f_2(\lambda_i) \ \alpha_i \  w^i, 
\end{align}
where 
$f_2(\lambda_i) = -1,$ if $\quad \lambda_n = 0,$  
where 
$f_2(\lambda_i) = - d \sqrt{\lambda_n} \ \frac{  \cos( \sqrt{\lambda_n} d)}{\sin( \sqrt{\lambda_n} d)},$  
if $\lambda_n > 0,$ and $d < \pi/\sqrt{\lambda_n},$ 
and where
$f_2(\lambda_i) = - d \sqrt{-\lambda_n} \ \frac{  \cosh( \sqrt{-\lambda_n} d)}{\sinh( \sqrt{-\lambda_n} d)},$ 
if $\lambda_n < 0.$

\paragraph{Proximal mappings of the atoms of the wavelet regularizers $\mathcal W_\lambda^{\alpha,q}(u).$}
We here explain how to compute the proximal mappings of the atoms $W_{n,r}$ of the wavelet regularizer $\mathcal W_\lambda^{\alpha,q}$ of \eqref{eq:DefWavRegMult} given in \eqref{eq:SplittingGeneralLevel} and \eqref{eq:SplittingZeroLevel}.

For the zeroth level $n=0,$ we notice that, by \eqref{eq:SplittingZeroLevel}, we have to compute the proximal mapping of \eqref{eq:atomWavExplCalc}.
This essentially corresponds to computing the proximal mapping of the  $q$th power of the distance function as a joint function of both arguments. The proximal mapping can be explicitly computed in terms of geodesics as shown in \cite{weinmann2014total}. It is given by 
\begin{align}\label{eq:tgv_manifoldParallelUnivAlgoProxDist}
\prox_{W_{\tilde n,0}} (u)_{n',0}=  
\begin{cases}
[\tilde u_{n-e_i,0},\tilde u_{n,0}]_t,    &\text{if } n' = n-e_i,  \\
[\tilde u_{n,0},\tilde u_{n-e_i,0}]_t,    &\text{if } n'= n,    \\
u_{n'}  ,              &\text{else}.  
\end{cases}
\end{align}
Here, $[\tilde u_{n-e_i,0},\tilde u_{n,0}]_t$ denotes the unit speed geodesic connecting   
$\tilde u_{n-e_i,0}$ and $\tilde u_{n,0}$ evaluated at time parameter $t,$ with $t$ dependent on $q.$
For $q=1,2$ we can explicitly calculate $t$ as
\begin{align}
	t = 
	\begin{cases}
	\tfrac{\lambda_2}{(2+ 2 \lambda_2)} \ \dist(\tilde u'_{n-e_i,0},\tilde u'_{n,0})
	   & \text{ for } q= 2,\\      
	\min\left(\lambda_2,\tfrac{1}{2} \dist(\tilde u'_{n-e_i,0},\tilde u'_{n,0})\right) &  
		\text{ for } q= 1.            
	\end{cases}           
\end{align} 

Concerning a general level $r \neq 0,$ according to \eqref{eq:SplittingGeneralLevel}, we have to compute the proximal mappings of \eqref{eq:proxWavGenLevFun}. This means, for given iterate $u \in \mathcal M^N$, we have to find a minimizer of the functional
\begin{align}\label{eq:proxWavGenLevFunExplDown}
u' \mapsto \tfrac{1}{2\mu} \dist(u,u')^2 +   W_{n,r}(u')  =
\tfrac{1}{2\mu} \dist(u,u')^2 +
\lambda_1 \cdot 2^{r q \ \left(\alpha+\tfrac{s}{2}-\tfrac{s}{q} \right)}\| d_{n,r}(u') \|^q_{\hat u_{n,r}}.   
\end{align} 
$q \geq 1.$
This is particularly interesting if the exponent $q$ equals $one$ since then  
the distance mapping is not differentiable on the diagonal. 
Since for a general level $r,$ no closed form expression of the proximal mapping of \eqref{eq:proxWavGenLevFun}
seems available,
we use a gradient descent scheme for $q>1$ , and a subgradient descent for $q=1$. 
Subgradient descent has already been used to compute the proximal mappings in the context of higher order 
total variation type regularization \cite{bavcak2016second,bredies2017total}.
Let us explain the (sub)gradient descent scheme to compute the proximal mapping of \eqref{eq:proxWavGenLevFun}. 
We first recall that the gradients 
of the mapping $u_k' \mapsto \dist(u_k,u_k')^2$ are given by $-\log_{u'_{k}}u_{k}.$ 
The gradient of the summand $W_{n,r}$ in \eqref{eq:proxWavGenLevFunExplDown} is given by 
\eqref{eq:derWnrfornr}. Using these gradients we compute the gradient of the mapping
\eqref{eq:proxWavGenLevFunExplDown} and then apply a standard gradient descent scheme for $q>1.$
If $q=1,$ we apply subgradient descend with an iteration dependent damping factor, i.e., for the $l$th iteration the subgradient of the  $(l-1)$th iterate is scaled by a factor $\tau_l.$ The sequence $(\tau_l)_l$ is chosen to be square-summable but not summable for convergence reasons. For the computation of the subgradient     
of the summands $W_{n,r}$ for $q=1,$ we notice that the mapping is differentiable whenever 
$\mean(a_{n-2\cdot}, \tilde u_{\cdot,r-1}) \neq \tilde u_{n,r}.$ For the degenerate case of point constellations where $W_{n,r}$ is not differentiable we notice that then at least $0$ is contained in the subdifferential of $W_{n,r}.$

\paragraph{Proximal mappings for the  analogue of $\ell_0$ sparse wavelet regularization.}
We here consider the atoms of $\mathcal W_\lambda^0(u)$ of \eqref{eq:Defl0SparseMult}
which we also denote by $W_{n,r}$ as in \eqref{eq:SplittingGeneralLevel}
and \eqref{eq:SplittingZeroLevel}.

For the zeroth level $n=0,$ we notice that, by \eqref{eq:SplittingZeroLevel}, we have to compute the proximal mapping of 
\begin{align}
u' \mapsto W_{\tilde n,0}(u') = W_{(n,i),0}(u') = 
\lambda_2 \ \#   \{\ \tilde u'_{n-e_i,0} \neq \tilde u'_{n,0} \}.
\end{align} 
We first notice that both $\tilde u'_{n-e_i,0}, \tilde u'_{n,0}$ correspond to data items of $u'$ whose indices we denote by $\iota(0),\iota(1),$ i.e., 
$\tilde u'_{n-e_i,0} = u'_{\iota(0)},$  and 
$\tilde u'_{n,0} = u'_{\iota(1)}.$ 
Then the computation of the proximal mapping of $W_{\tilde n,0}(u)$ amounts to minimizing, for given $u$, the functional 
\begin{align}\label{eq:ProblemZerothLevel}
u' \mapsto \tfrac{1}{2\mu} \dist(u,u')^2 + \lambda_2 \ \#  \{\ \tilde u'_{n-e_i,0} \neq \tilde u'_{n,0} \}.
\end{align} 
A short application of calculus shows that, if $\dist(u_{\iota(0)},u_{\iota(1)})> 2 \sqrt{\mu\lambda_2},$ 
then $u^\ast := u$ minimizes \eqref{eq:ProblemZerothLevel}. 
If $\dist(u_{\iota(0)},u_{\iota(1)})> 2 \sqrt{\mu\lambda_2},$ it follows that choosing both components $\iota(0)$ and $\iota(1)$ of  $u^\ast$ by the geodesic midpoint of $u_{\iota(0)}$ and $u_{\iota(1)}$ and setting the other components to the corresponding values of $u$ yields a minimizer of 
\eqref{eq:ProblemZerothLevel}, i.e., a minimizer $u^\ast$ is given by its components
\begin{align}
	(u^\ast)_l =
	\begin{cases} 
	     \mean\left((0.5,0.5), (u_{\iota(0)},u_{\iota(1)})\right),\quad  & 
	     \text{if} \ l \in \{\iota(0),\iota(1)\},   \\
	     u_l, & \text{else}.
	\end{cases}  
\end{align}
For a general level $r \neq 0,$ we have to compute the proximal mappings of 
\begin{align}
u' \mapsto W_{n,r}(u') = \lambda_1  \ \# \ \{d_{n,r}(u) \neq 0\}
\end{align} 
where the right hand side equals zero if and only if $\tilde u'_{n,r}  \in \mathrm{S} \tilde u'_{n',r-1},$  
i.e., $\tilde u'_{n,r}$ is a mean of the $\tilde u_{n',r-1}$ with the weights given by the subdivision scheme;
see \eqref{eq:DefDnr} and \eqref{eqDefSubdiv}.
Else, the right hand side equals $\lambda_1.$
The computation of the proximal mapping of $W_{n,r}(u)$ amounts to minimizing, for given $u$, the functional 
\begin{align}\label{eq:ProblemZerothLevel}
u' \mapsto W'(u')= \tfrac{1}{2\mu} \dist(u,u')^2 +   \lambda_1  \ \# \ \{d_{n,r}(u') \neq 0\}.
\end{align} 
We distinguish two cases depending on whether $\#  \{d_{n,r}(u') \neq 0\}$ equals one or zero.
Given that $\#  \{d_{n,r}(u') \neq 0\} = 1,$ the optimal solution of \eqref{eq:ProblemZerothLevel} subject to this constraint is given by
\begin{align}
	u^\#  =  u, \qquad  W'(u^\#) = \lambda_1.                 
\end{align}
Given that $\#  \{d_{n,r}(u') \neq 0\} = 0,$ we have by the discussion preceding \eqref{eq:ProblemZerothLevel}
that
\begin{align}\label{eq:RewriteMin4W0dnr}
u^{\#\#}  
=  \argmin_{u': d_{n,r}(u') = 0} W'(u')
= \argmin_{u': \tilde u'_{n,r}  \in \mathrm{S} \tilde u'_{n',r-1}} W'(u')
\end{align}
For minimizing the right hand side of \eqref{eq:RewriteMin4W0dnr}, we set all the components of $u^{\#\#}$ to 
the corresponding components of $u$ which do not correspond to $\tilde u^{\#\#}_{n,r}$ or to $\tilde u^{\#\#}_{n',r-1}$ 
with $n-2n' \in \supp a.$ The components of $u^{\#\#}$  corresponding to $\tilde u^{\#\#}_{n',r-1}$
are given by minimizing the functional
\begin{align}\label{eq:ProblemRedGenL0}
   u' \mapsto V(u') :=
   \tfrac{1}{2\mu} 
   \left(
   \sum_{n': n-2n' \in \supp a} \dist(\tilde u'_{n',r-1},\tilde u_{n',r-1})^2 
   +  
   \dist(\mean(a_{n-2\cdot}, \tilde u'_{\cdot,r-1}),\tilde u_{n,r})^2
   \right).
\end{align} 
Then, the component $\tilde u^{\#\#}_{n,r}$ of $u^{\#\#}$ is given by the point   
$\mean(a_{n-2\cdot}, \tilde u^{\#\#}_{\cdot,r-1})$ of \eqref{eq:ProblemRedGenL0} for the minimizing 
arguments $\tilde u^{\#\#}_{n',r-1}$ of \eqref{eq:ProblemRedGenL0}. 
For minimizing \eqref{eq:ProblemRedGenL0}, we observe that \eqref{eq:ProblemRedGenL0} just states the minimization problem for the proximal mapping of the mean mapping already discussed before.
We note that the gradient of the first summands are just the gradients of the distance map as discussed above.
Further, we notice that the last summand of the functional $V$ equals the second line of \eqref{eq:derWnrfornr}.
Then we use gradient descent for computing the $\tilde u^{\#\#}_{n',r-1}.$
 Summing up, we have 
\begin{equation}
\prox_{\mu W_{n,r}}(u) =
\begin{cases}
  u^{\#} & \text{if} \ V(u^{\#\#}) > \lambda_1, \\
  u^{\#\#}&  \text{if} \ V(u^{\#\#}) \leq  \lambda_1,
\end{cases}
\end{equation}
where $V$ is given by \eqref{eq:ProblemRedGenL0}.

\paragraph{Gradients and proximal mappings for the atoms of the data term.}
Methods for deriving the gradients and proximal mappings of the decomposition of the data term $\mathcal D$ into its atoms $D_i$ according to \eqref{eq:DataTermDi} have been derived in \cite{storath2018variational}.
For the data term, we employ the schemes derived there and refer to this reference for details.
(There, similar as above, the gradients of the intrinsic mean mapping are employed for the computations.)

\begin{figure}[!tp]
\def\figfolderA{experiments/exp_denoising_1D/}
\def\figfolderB{experiments/exp_denoising_1D_S2/}
\def\figfolderC{experiments/exp_denoising_1D_Pos3/}
\def\hs{\hfill}
\def\vs{\vspace{0.03\textwidth}}
\def\figurewidth{0.3\textwidth}
\def\figurewidthB{0.15\textwidth}
\centering
{
\footnotesize
\begin{tabular}{ccc}
Data & 1st order interpol. wavelet & 3rd order DD wavelet \\[1ex] 
\includegraphics[width=\figurewidth]{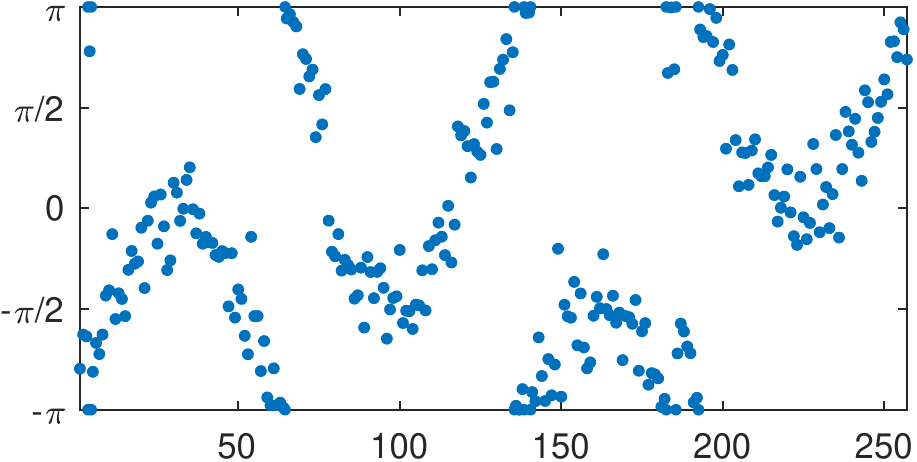}   &
\includegraphics[width=\figurewidth]{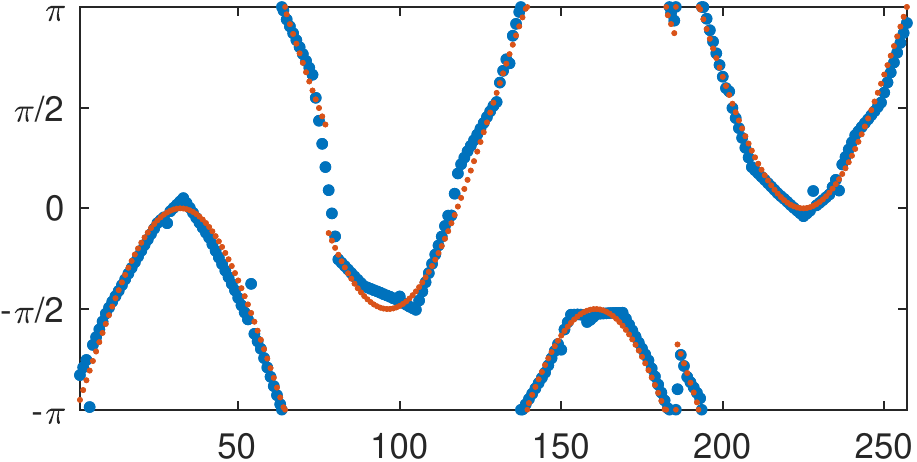} &
\includegraphics[width=\figurewidth]{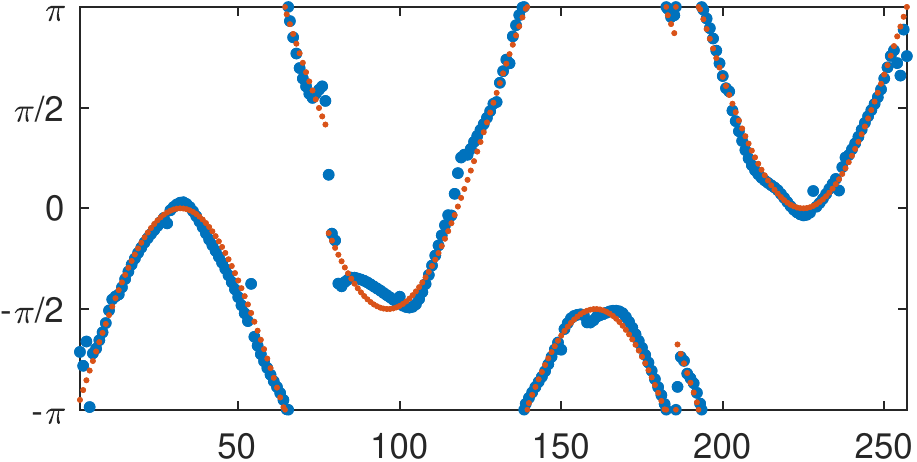}  \\
 &
$\deltaSNR$: $\input{\figfolderA exp_denoising_1D_Haar_deltaSNR_1.txt}$ dB &
$\deltaSNR$: $\input{\figfolderA exp_denoising_1D_Spline_deltaSNR_1.txt}$ dB \\[2ex]
\includegraphics[width=\figurewidth]{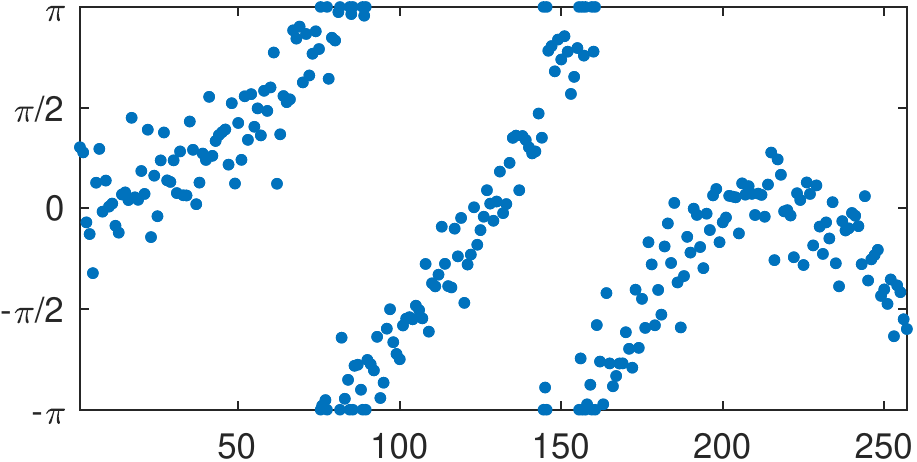}  &
\includegraphics[width=\figurewidth]{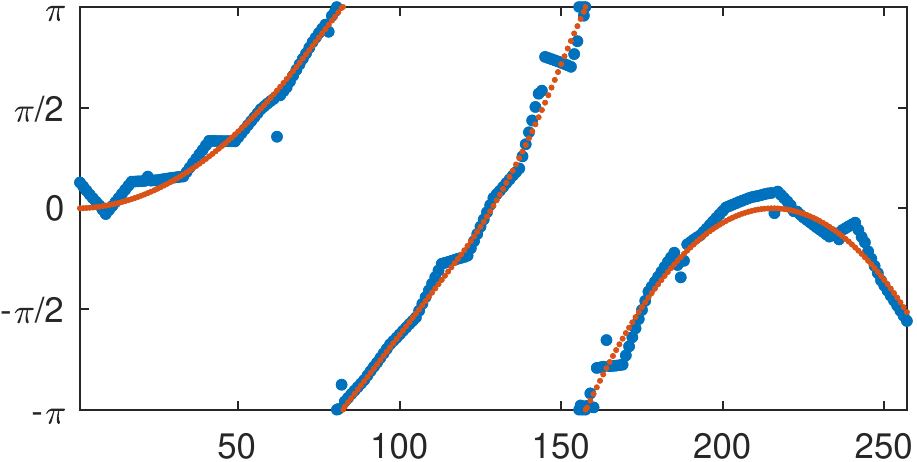} &
\includegraphics[width=\figurewidth]{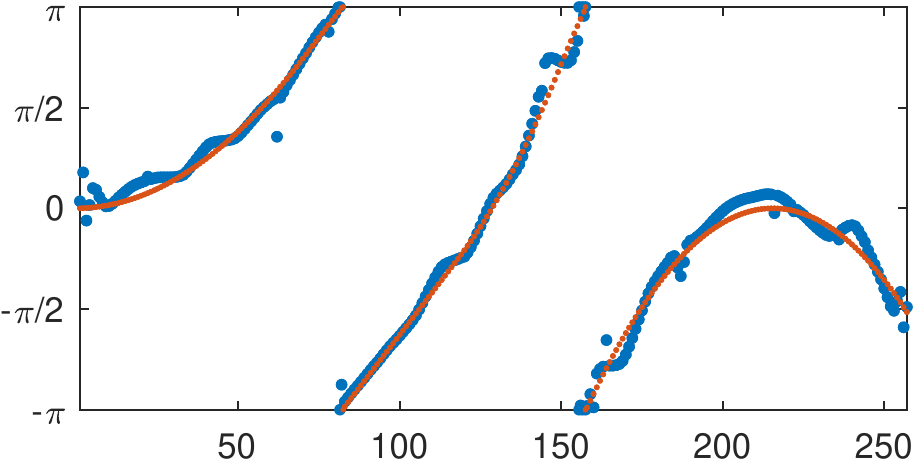} \\
&
$\deltaSNR$: $\input{\figfolderA exp_denoising_1D_Haar_deltaSNR_2.txt}$ dB &
$\deltaSNR$: $\input{\figfolderA exp_denoising_1D_Spline_deltaSNR_2.txt}$ dB \\[2ex]
\includegraphics[width=\figurewidth]{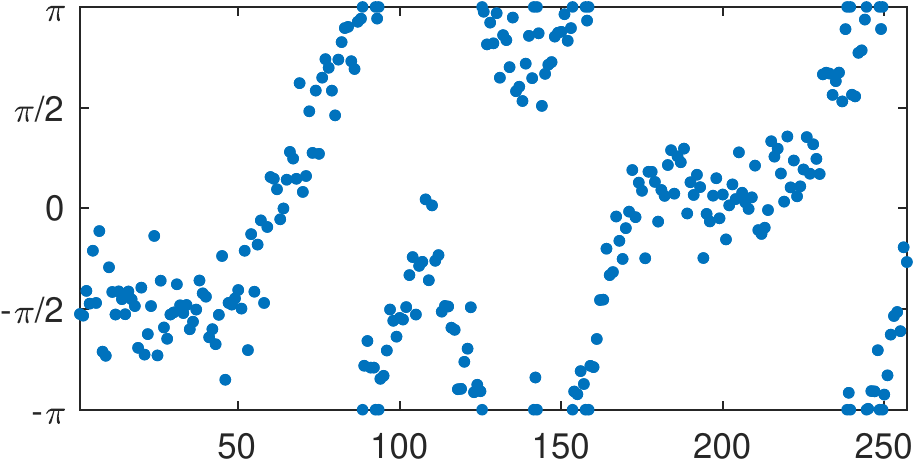}  &
\includegraphics[width=\figurewidth]{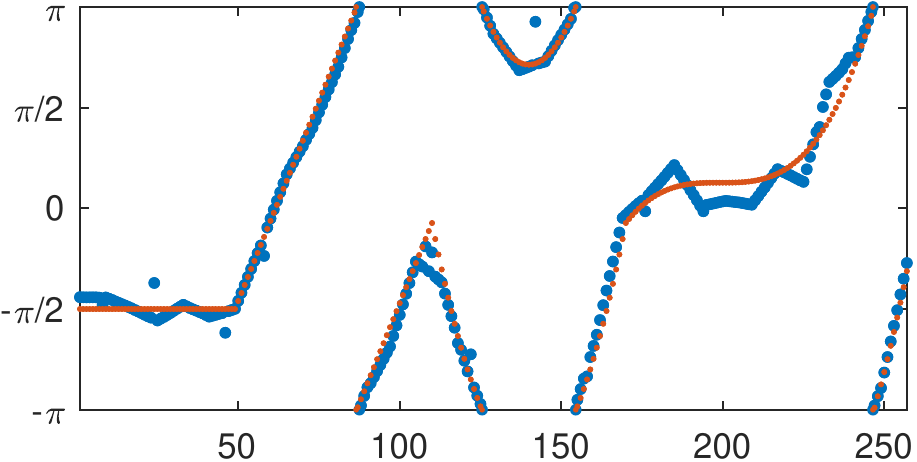} &
\includegraphics[width=\figurewidth]{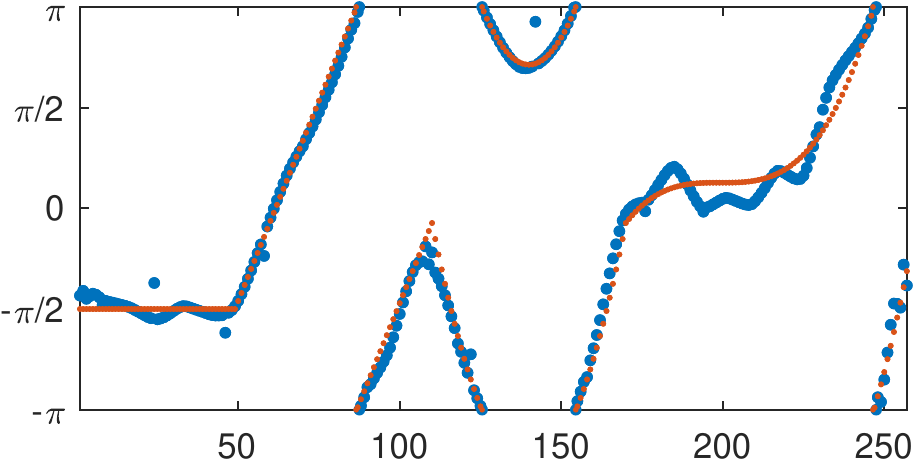} \\
&
$\deltaSNR$: $\input{\figfolderA exp_denoising_1D_Haar_deltaSNR_3.txt}$ dB &
$\deltaSNR$: $\input{\figfolderA exp_denoising_1D_Spline_deltaSNR_3.txt}$ dB \\[4ex]
\includegraphics[width=\figurewidthB]{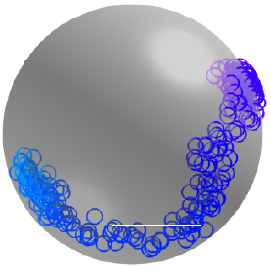}   &
\includegraphics[width=\figurewidthB]{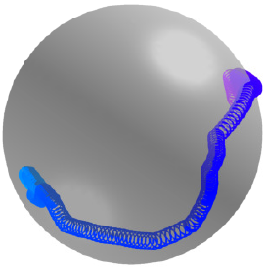} &
\includegraphics[width=\figurewidthB]{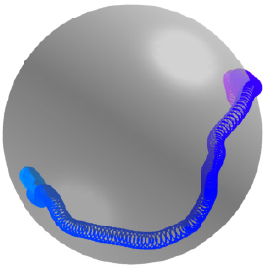}  
 \\
 &
$\deltaSNR$: $\input{\figfolderB exp_denoising_1D_S2_Haar_deltaSNR_1.txt}$ dB &
$\deltaSNR$: $\input{\figfolderB exp_denoising_1D_S2_Spline_deltaSNR_1.txt}$ dB \\[4ex]
\includegraphics[width=\figurewidthB]{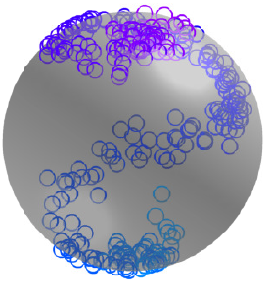}   &
\includegraphics[width=\figurewidthB]{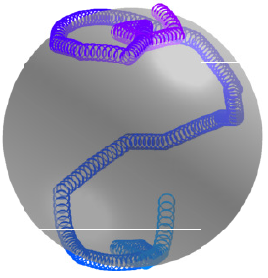} &
\includegraphics[width=\figurewidthB]{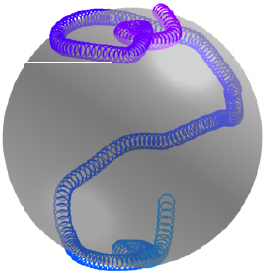}  
 \\
 &
$\deltaSNR$: $\input{\figfolderB exp_denoising_1D_S2_Haar_deltaSNR_2.txt}$ dB &
$\deltaSNR$: $\input{\figfolderB exp_denoising_1D_S2_Spline_deltaSNR_2.txt}$ dB \\[4ex]
\includegraphics[width=\figurewidth]{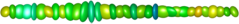}   &
\includegraphics[width=\figurewidth]{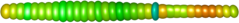} &
\includegraphics[width=\figurewidth]{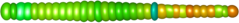}  
 \\
 &
$\deltaSNR$: $\input{\figfolderC exp_denoising_1D_Pos3_Haar_deltaSNR_1.txt}$ dB &
$\deltaSNR$: $\input{\figfolderC exp_denoising_1D_Pos3_Spline_deltaSNR_1.txt}$ dB \\[4ex]
\includegraphics[width=\figurewidth]{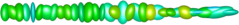}   &
\includegraphics[width=\figurewidth]{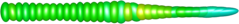} &
\includegraphics[width=\figurewidth]{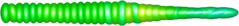}  
 \\
 &
$\deltaSNR$: $\input{\figfolderC exp_denoising_1D_Pos3_Haar_deltaSNR_2.txt}$ dB &
$\deltaSNR$: $\input{\figfolderC exp_denoising_1D_Pos3_Spline_deltaSNR_2.txt}$ dB 
\end{tabular}
}
\caption{
Results of the proposed $\ell^1$ wavelet regularization 
for manifold-valued signals. The given noisy data is shown on the left, 
the result using the manifold analogue of the first order interpolatory wavelet and the third order Deslaurier-Dubuc (DD) wavelet
along with the corresponding signal-to-noise-ratio improvement are shown in the center and on the right, respectively.\\
}
\label{fig:l1}
\end{figure}

\section{Experimental Results}\label{sec:Experiments}

\begin{figure}[!tp]
\def\figfolderA{experiments/exp_coeff_1D/}
\def\hs{\hfill}
\def\vs{\vspace{0.03\textwidth}}
\def\figurewidth{0.29\textwidth}
\def\figurewidthB{0.15\textwidth}
\def\levelA{1}
\def\levelB{2}
\def\levelC{3}
\centering
{
\footnotesize
\begin{tabular}{m{0.03\textwidth}ccc}
& Level 3 & Level 2 & Level 1 \\[1ex]
\tabrotate{\phantom{x}1st ord. (start)} & \includegraphics[width=\figurewidth]{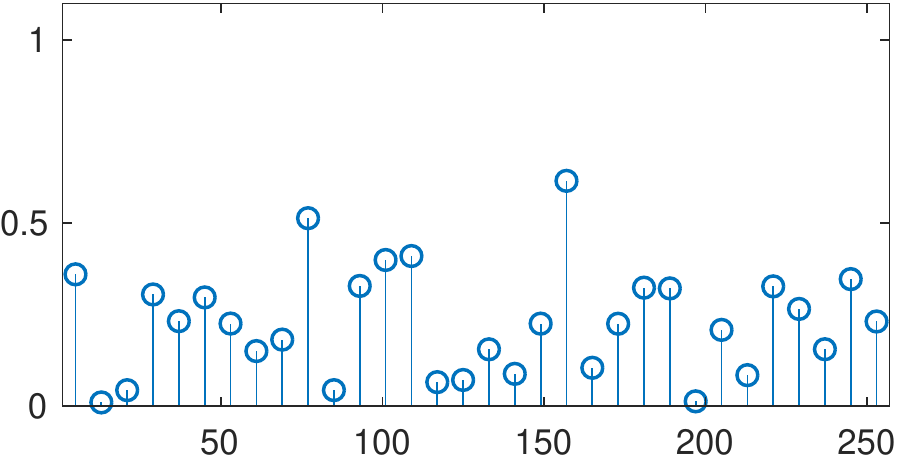} &
\includegraphics[width=\figurewidth]{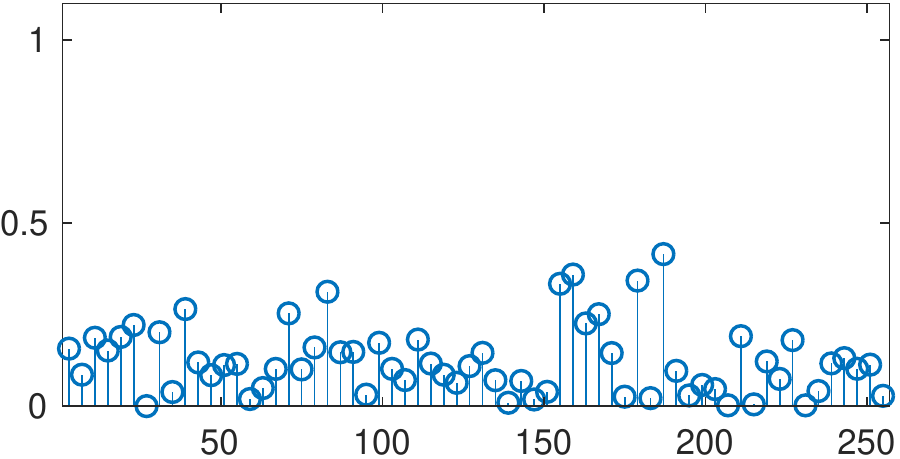} &
\includegraphics[width=\figurewidth]{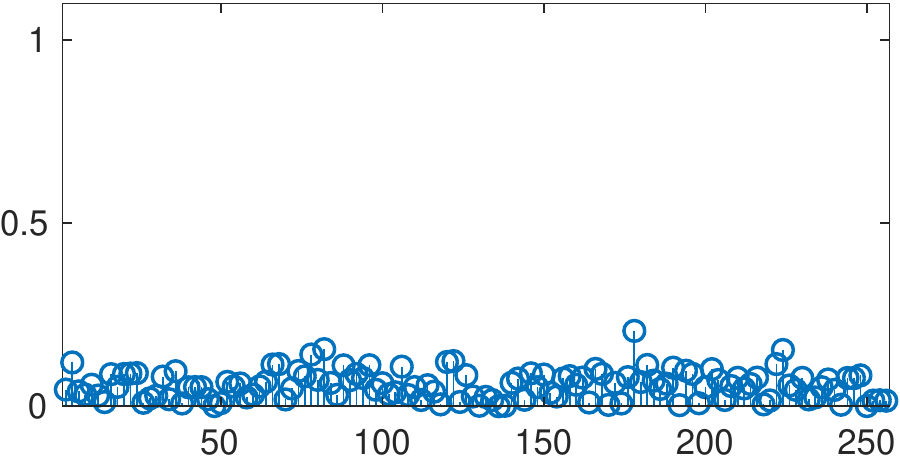} \\
\tabrotate{\phantom{xx}1st ord. (end)} & \includegraphics[width=\figurewidth]{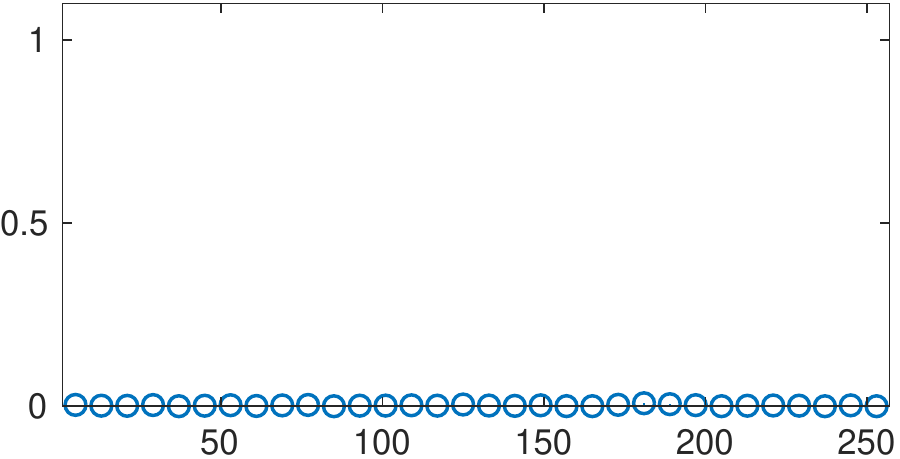} &
\includegraphics[width=\figurewidth]{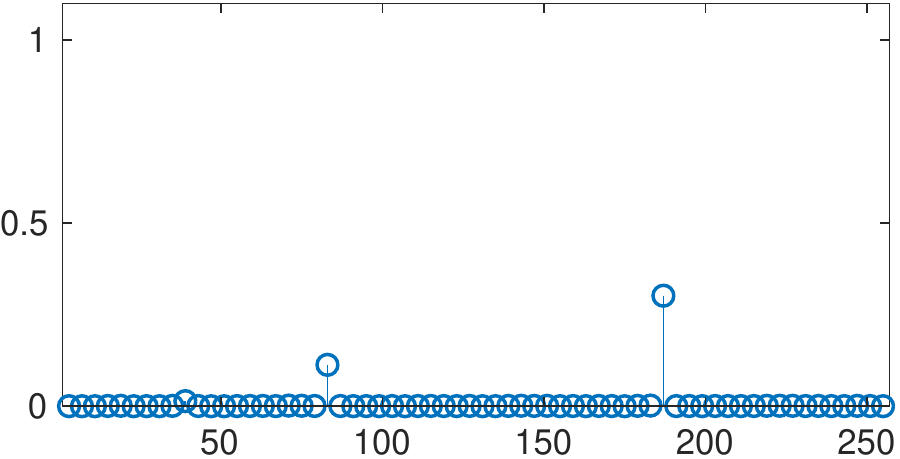} &
\includegraphics[width=\figurewidth]{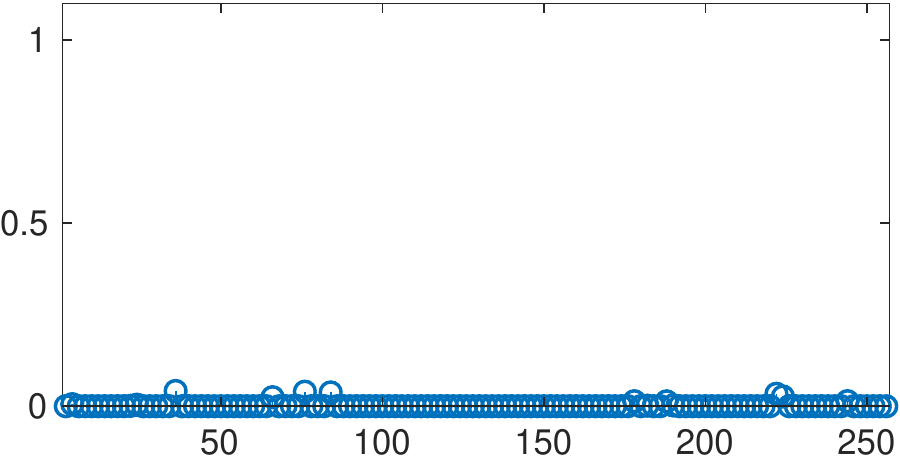} \\[2ex]
\tabrotate{\phantom{xxxx}DD (start)} &  \includegraphics[width=\figurewidth]{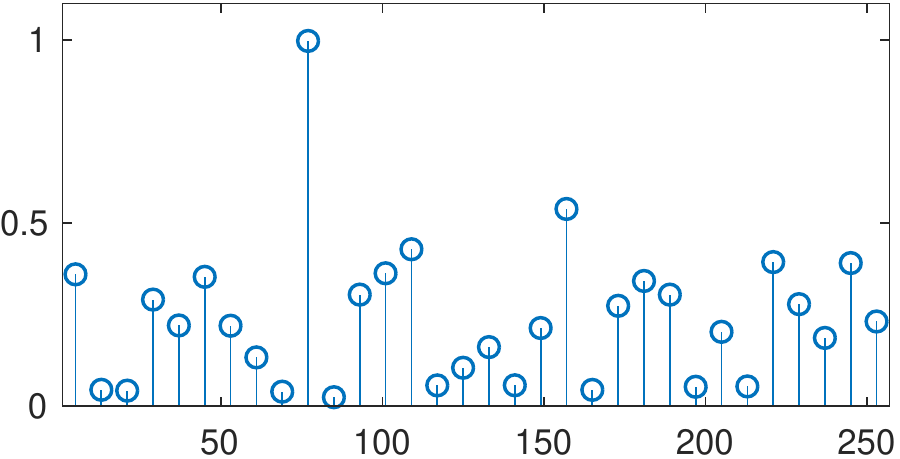} &
\includegraphics[width=\figurewidth]{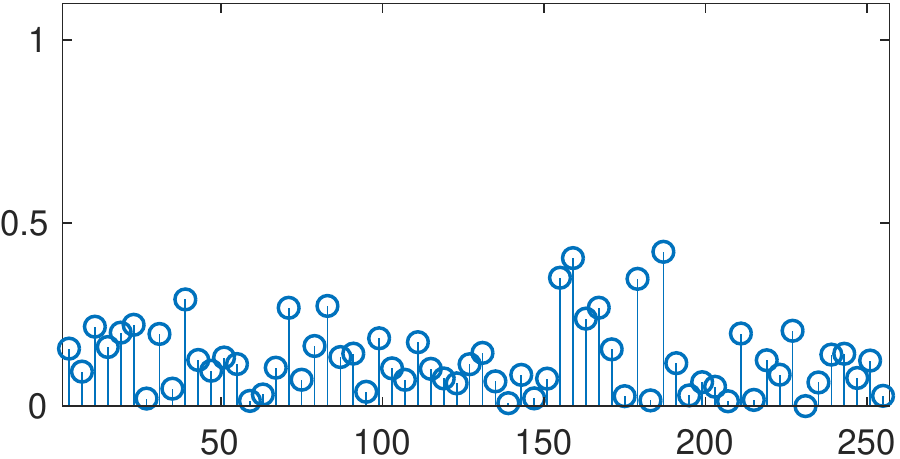} &
\includegraphics[width=\figurewidth]{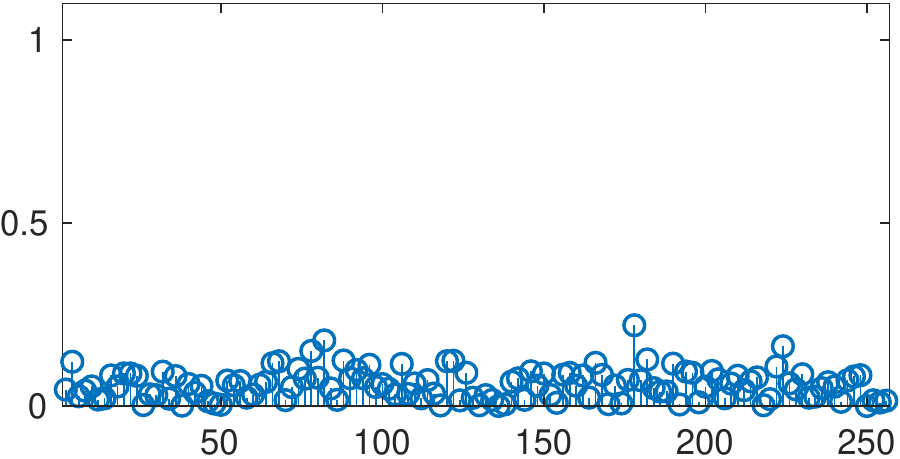} \\
\tabrotate{\phantom{xxxx}DD (end)} & \includegraphics[width=\figurewidth]{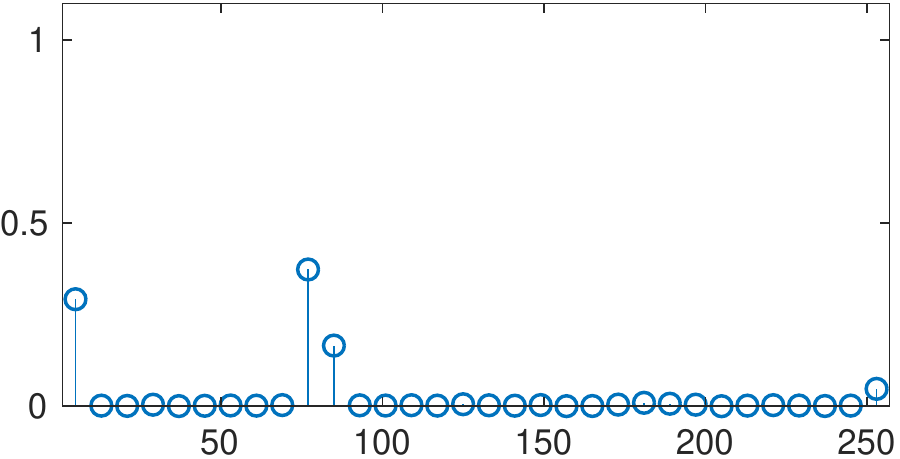} &
\includegraphics[width=\figurewidth]{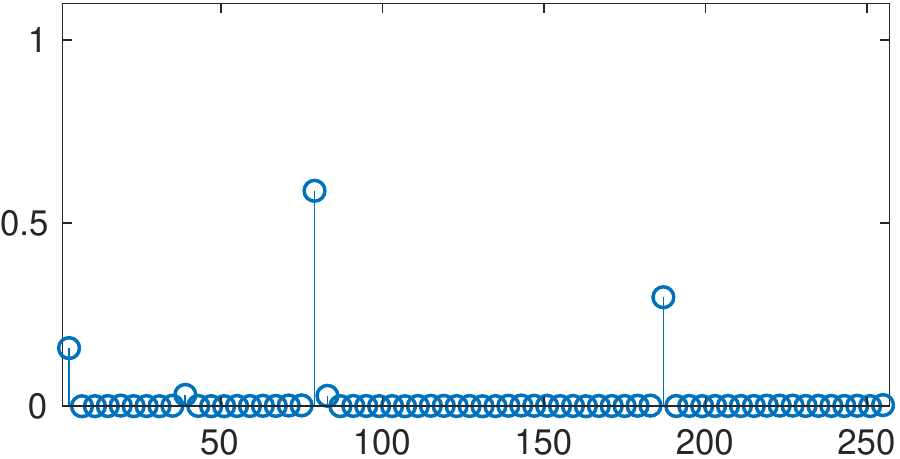} &
\includegraphics[width=\figurewidth]{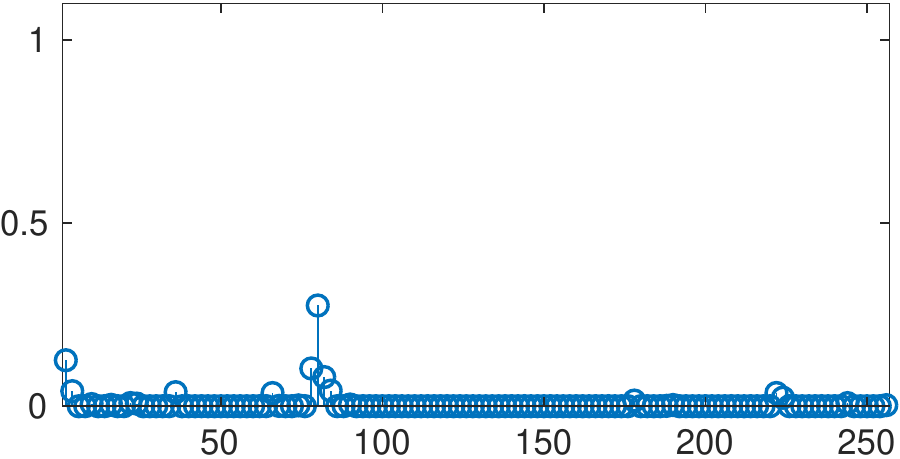} \\
\end{tabular}
}
\caption{
Wavelet coefficients of the first signal in Figure~\ref{fig:l1}
at the start (noisy data) and the end (reconstruction) of the iteration
using the  manifold analogue of the first order interpolatory wavelet (first and second row)
and the third order Deslaurier-Dubuc (DD) wavelet (third and fourth row).
Only a few coefficients remain significantly larger than zero after the iterations.
}
\label{fig:coeff_l1}
\end{figure}

\begin{figure}[!tp]
\def\figfolderA{experiments/exp_denoising_l0/}
\def\figfolderB{experiments/exp_denoising_1D_S2_l0/}
\def\figfolderC{experiments/exp_denoising_1D_Pos3_l0/}
\def\hs{\hfill}
\def\vs{\vspace{0.03\textwidth}}
\def\figurewidth{0.3\textwidth}
\def\figurewidthB{0.15\textwidth}
\centering
{
\footnotesize
\begin{tabular}{ccc}
Data & 1st order interpol. wavelet & 3rd order DD wavelet \\[1ex]
\includegraphics[width=\figurewidth]{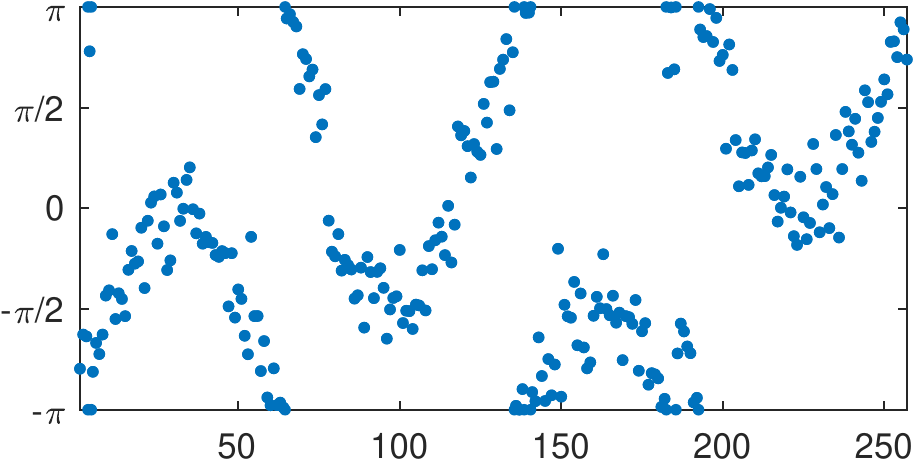}   &
\includegraphics[width=\figurewidth]{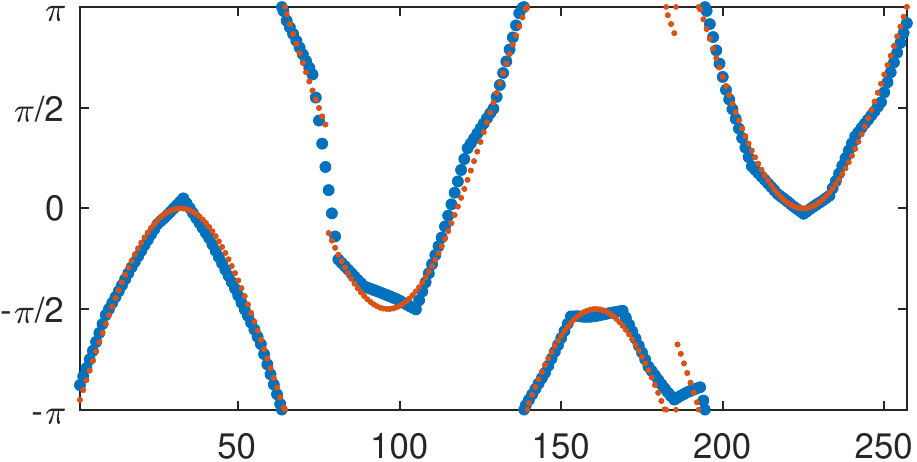} &
\includegraphics[width=\figurewidth]{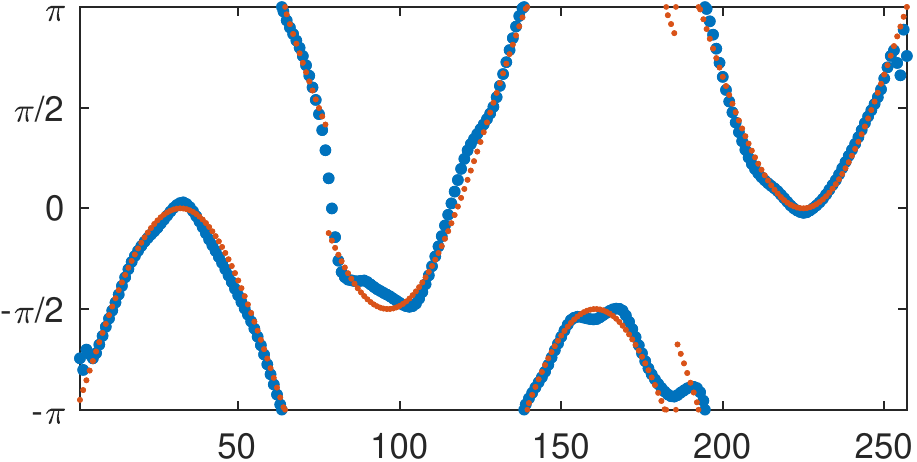}  \\
 &
$\deltaSNR$: $\input{\figfolderA exp_denoising_1D_l0_Haar_deltaSNR_1.txt}$ dB &
$\deltaSNR$: $\input{\figfolderA exp_denoising_1D_l0_Spline_deltaSNR_1.txt}$ dB \\[2ex]
\includegraphics[width=\figurewidth]{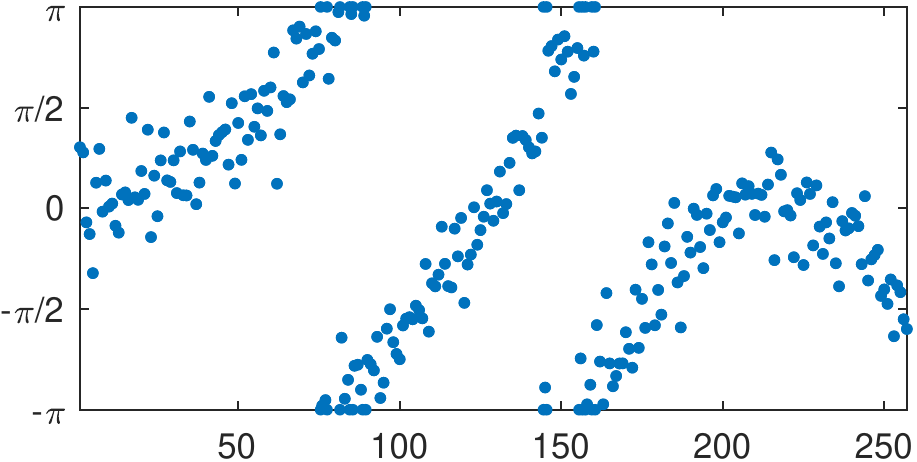}  &
\includegraphics[width=\figurewidth]{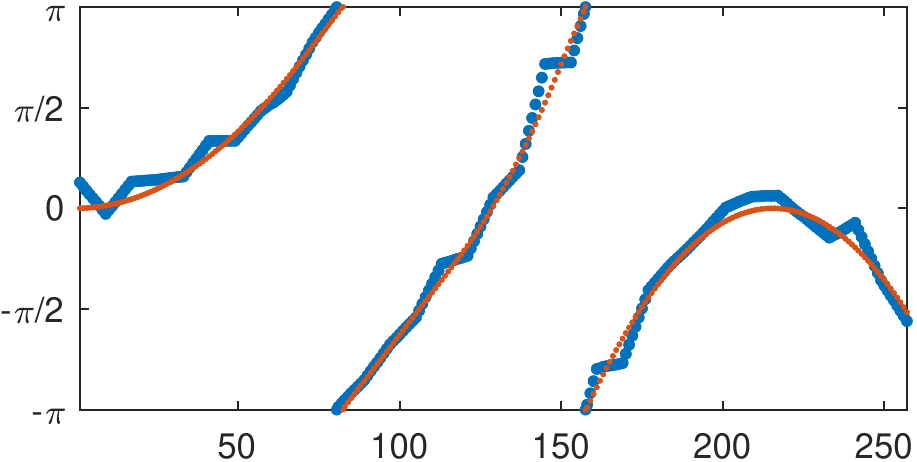} &
\includegraphics[width=\figurewidth]{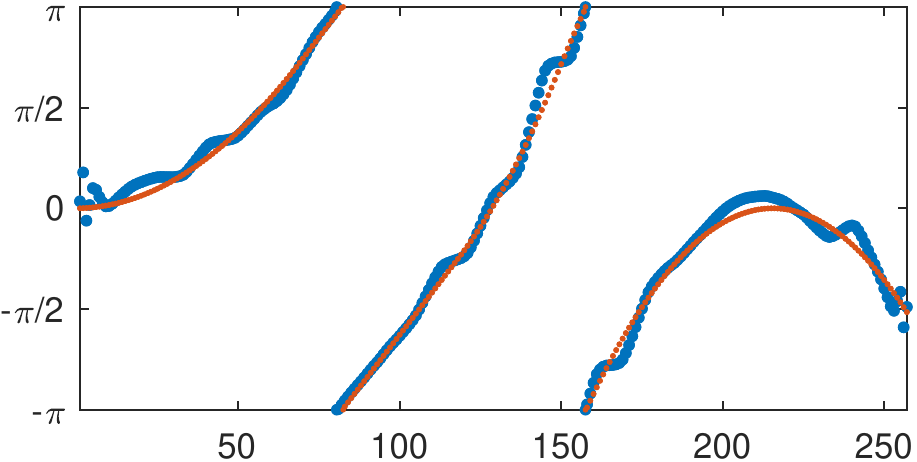} \\
&
$\deltaSNR$: $\input{\figfolderA exp_denoising_1D_l0_Haar_deltaSNR_2.txt}$ dB &
$\deltaSNR$: $\input{\figfolderA exp_denoising_1D_l0_Spline_deltaSNR_2.txt}$ dB \\[2ex]
\includegraphics[width=\figurewidth]{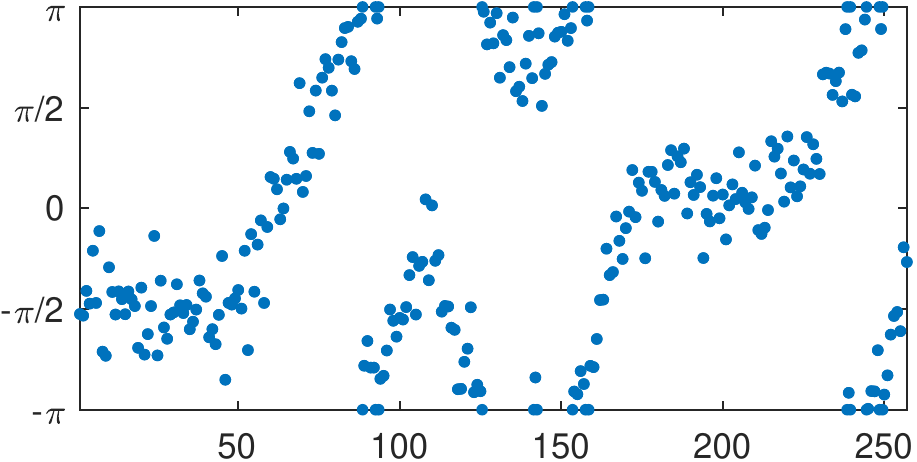}  &
\includegraphics[width=\figurewidth]{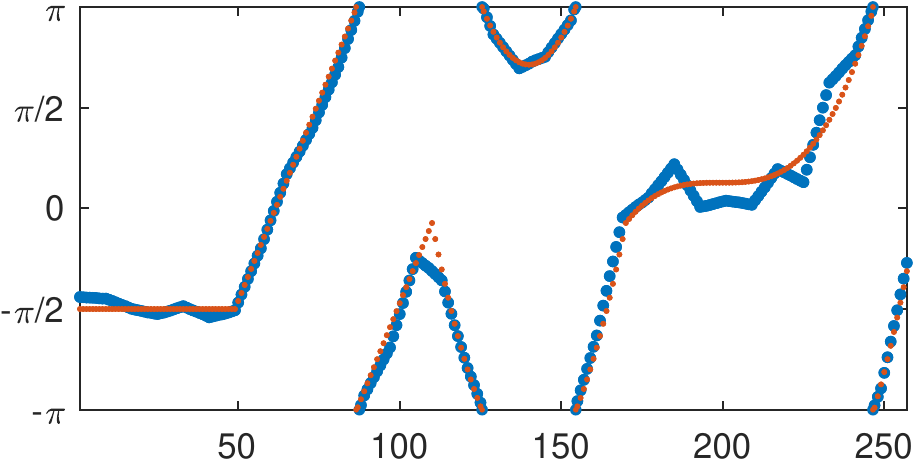} &
\includegraphics[width=\figurewidth]{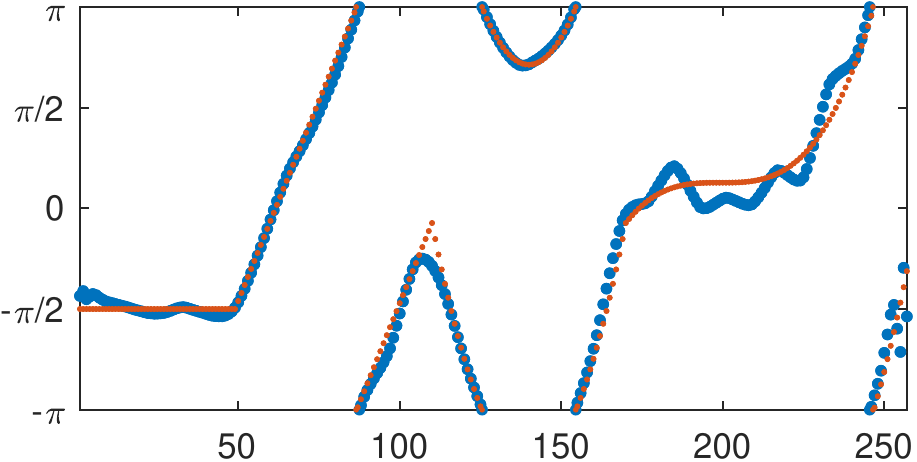} \\
&
$\deltaSNR$: $\input{\figfolderA exp_denoising_1D_l0_Haar_deltaSNR_3.txt}$ dB &
$\deltaSNR$: $\input{\figfolderA exp_denoising_1D_l0_Spline_deltaSNR_3.txt}$ dB \\[4ex]
\includegraphics[width=\figurewidthB]{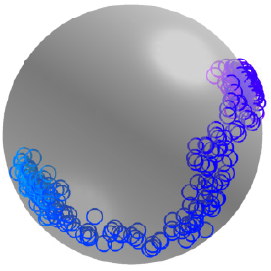}   &
\includegraphics[width=\figurewidthB]{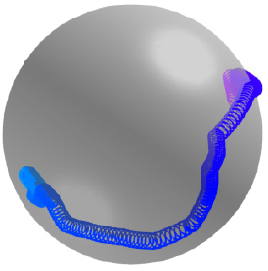} &
\includegraphics[width=\figurewidthB]{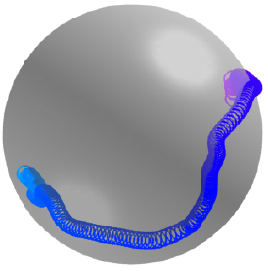}  
 \\
 &
$\deltaSNR$: $\input{\figfolderB exp_denoising_1D_S2_l0_Haar_deltaSNR_1.txt}$ dB &
$\deltaSNR$: $\input{\figfolderB exp_denoising_1D_S2_l0_Spline_deltaSNR_1.txt}$ dB \\[4ex]
\includegraphics[width=\figurewidthB]{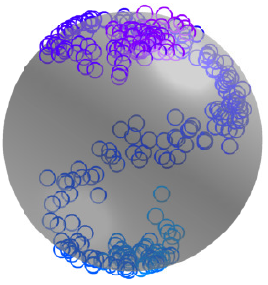}   &
\includegraphics[width=\figurewidthB]{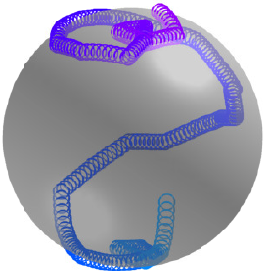} &
\includegraphics[width=\figurewidthB]{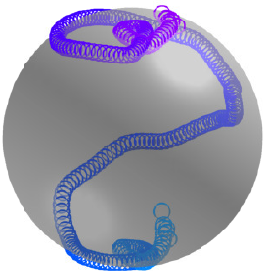}  
 \\
 &
$\deltaSNR$: $\input{\figfolderB exp_denoising_1D_S2_l0_Haar_deltaSNR_2.txt}$ dB &
$\deltaSNR$: $\input{\figfolderB exp_denoising_1D_S2_l0_Spline_deltaSNR_2.txt}$ dB \\[4ex]
\includegraphics[width=\figurewidth]{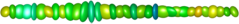}   &
\includegraphics[width=\figurewidth]{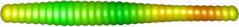} &
\includegraphics[width=\figurewidth]{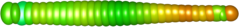}  
 \\
 &
$\deltaSNR$: $\input{\figfolderC exp_denoising_1D_Pos3_l0_Haar_deltaSNR_1.txt}$ dB &
$\deltaSNR$: $\input{\figfolderC exp_denoising_1D_Pos3_l0_Spline_deltaSNR_1.txt}$ dB \\[4ex]
\includegraphics[width=\figurewidth]{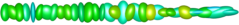}   &
\includegraphics[width=\figurewidth]{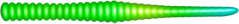} &
\includegraphics[width=\figurewidth]{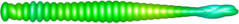}  
 \\
 &
$\deltaSNR$: $\input{\figfolderC exp_denoising_1D_Pos3_l0_Haar_deltaSNR_2.txt}$ dB &
$\deltaSNR$: $\input{\figfolderC exp_denoising_1D_Pos3_l0_Spline_deltaSNR_2.txt}$ dB 
\end{tabular}
}
\caption{Results of the  $\ell^0$-variant of the propose method for the same data as given in Figure~\ref{fig:l1}. The given noisy data is shown on the left, 
the result using the manifold analogue of the first order interpolatory wavelet and the third order Deslaurier-Dubuc (DD) wavelet
		along with the corresponding signal-to-noise-ratio improvement are shown in the center and on the right, respectively.
}
\label{fig:l0}
\end{figure}

We carry out experiments for data with values in the circle $S^1$, the sphere $S^2$ and 
 the manifold $\Pos_3$ of positive definite matrices equipped with the Fisher-Rao metric.
 $S^1$ valued data is visualized by the phase angle, and color-coded as hue value in the HSV color space when displaying image data.
 We visualize $S^2$-valued signal by points on the unit sphere.
The abscissa is represented  by a color varying continuously from the start of signal (blue) to the end of the signal (purple).
Data on the $\Pos_3$ manifold is visualized
by the isosurfaces of the corresponding quadratic forms. More precisely, the ellipse visualizing
the point $f_p$ at voxel $p$ are the points $x$ fulfilling $(x-p)^\top f^{-1}_p (x-p) = c,$ for some $c>0.$
The $S^1$-valued data are corrupted by Von Mises noise 
with concentration parameter $\kappa$ \cite{mardia2009directional}.
The $S^2$-valued data are created from the original image $g$ by
$f_{ij} = \exp_{g_{ij}} \eta_{ij}$ 
where  $\eta_{ij}$ is a tangent vector at $g_{ij}$
and both its  components are Gaussian distributed with standard deviation $\sigma.$
The $\Pos_3$ valued signal are corrupted by Rician noise  with noise parameter $\eta$ \cite{basu2006rician}.
To quantitatively measure the  quality of a reconstruction,
we use the manifold variant of the \emph{signal-to-noise ratio improvement} 
$
	\deltaSNR = 10 \log_{10} (  ( \sum_{ij} d(g_{ij}, f_{ij})^2) / (\sum_{ij} d(g_{ij}, u_{ij})^2 ) )  \dB,
$
see \cite{weinmann2014total}.
Here $f$ is the noisy data, $g$ is the ground truth, and $u$ is a regularized  reconstruction.
A higher $\deltaSNR$ value  means better reconstruction quality.
We have implemented the presented methods in Matlab 2017b.
We used functions of the toolboxes CircStat~\cite{berens2009circstat}, Manopt~\cite{manopt}, MVIRT~\cite{bavcak2016second}, 
and implementations from the authors' prior works \cite{weinmann2014total, bredies2017total}.
All examples were computed using $200$ iterations.
The cooling sequence $(\mu^k)_{k \in\N}$ used as stepsize in the gradient descent for computing the non-explicit proximal mappings was chosen as $\lambda^k = \mu^0 k^{-\tau}$
with $\tau = 1.$
For the spherical data we used
a stagewise cooling similar to the one in \cite{bredies2017total}, i.e., letting the sequence fixed to $\lambda_0$ for $50$ iterations in the first stage, use the moderate cooling $\tau = 0.35$ in the second stage until iteration $100$ and then the cooling $\tau = 0.55$ afterwards.
For the secondary regularization parameter, we used $\lambda_2 =0$ in all denoising experiments, 
and $\lambda_2 =0.0001$ in all deconvolution experiments.
The primary regularization parameter $\lambda_1$ was adjusted empirically
and is reported in the following.

Figure~\ref{fig:l1} shows the results of the proposed
 $\ell^1$ wavelet regularization for noise signal
 using the first order interpolatory wavelet and the third order Deslaurier-Dubuc (DD) wavelet for manifold valued data 
 given by the masks in \eqref{eq:TheBasicSchemesEmployed}.
 The noise parameters were chosen as $\kappa=5,$ $\sigma=$ and $\eta = 80,$ 
 for the $S^1$-valued, the $S^2$-valued and the $\Pos_3$-valued signals, respectively.
 The regularization parameters, $\lambda_1 = 4$ for the spherical data and  $\lambda_1 = 8$ for the $\Pos_3$ data, were chosen empirically.
 We observe that the noise is reduced significantly
 with a signal-to-noise-ratio improvement of up to around $8.9~\text{dB}$
by both the first order interpolatory wavelet and the third order Deslaurier-Dubuc (DD) wavelet. Both schemes yield comparable reconstruction results.

 Figure~\ref{fig:coeff_l1} displays the wavelet coefficients
 of the first signal in Figure~\ref{fig:l1} 
 before the first iteration (noisy data) and after the final iteration (reconstruction)
 for the first order interpolatory wavelet and the third order DD wavelet.
 We observe that also in the manifold setup the method has a shrinkage effect on the coefficients:
 only a few coefficients remain significantly larger than zero after
 the final iteration.

In Figure~\ref{fig:l0}, the $\ell^0$ variant of the proposed method is applied
to the same signals as in Figure~\ref{fig:l1}.
We observe that the noise is reduced significantly
with a signal-to-noise-ratio improvement which is comparable to the $\ell^1$ variant.	
We observe that  $\ell^0$ regularization and the $\ell^1$ variant perform equally well
with only slight differences depending on the particular signal under consideration.

 Figure~\ref{fig:deconv} shows the results of a deconvolution experiment. 
Here the original signal was convolved with the manifold analogue of a Gaussian kernel 
with standard deviation $2$ (set to zero outside a window of length $13$)  
before corrupting it with noise.
Then the signal was reconstructed using the proposed wavelet sparse regularization method
with the data term including the manifold-valued convolution operator as forward operator.
Here, we use generalized forward backward algorithm with Gau\ss-Seidel type update scheme
\eqref{eq:algFBwithGSupdateAndTraj}.
As the point deconvolution and denoising task is more involved than pure denoising
we used lower noise levels for these experiments; more precisely, for the noise generation, we used the parameters $\kappa=20,$
$\sigma = 0.1,$ and $\eta = 60$ for the $S^1,$ the $S^2,$ and the $\Pos_3$ data, respectively. 
The first regularization parameter $\lambda_1$
was set equal to $6,$ $4$ and $2,$ respectively.
As in the direct measurement case, the first order interpolatory wavelet and the third order DD wavelet
yield comparable reconstruction results. 

We proceed with experiments considering bivariate data.
We focus on denoising using $\ell^1$ wavelet regularization based on the manifold analogue of the tensor product third order DD wavelet.
 In Figure~\ref{fig:s1_2D} 
 a synthetic $S^1$-valued image was corrupted by von Mises noise ($\kappa=10$)
 and reconstructed using $\lambda_1 = 32.$
 In Figure~\ref{fig:Pos3_2D}, a synthetic $\Pos_3$-valued image was corrupted by Rician noise of level $\eta = 40$
 and reconstructed using $\lambda_1 = 8.$
 As in the univariate case, the proposed method reduces the noise significantly.
 We observe some block artifacts at the roundish structures which can be attributed to the tensor product structure.
Eventually we show the results on a real diffusion tensor image of a human brain
provided by the Camino project \cite{cook2006camino}\footnote{Data available at \url{http://camino.cs.ucl.ac.uk/}}.
The original tensors were computed from the diffusion weighted images by a least squares fit
based on the Stejskal-Tanner equation, and invalid tensors were filled by averages of their neighboring pixels.
The image shows a cutout of the axial slice number $20.$
Figure~\ref{fig:camino} shows the smoothing effect of the proposed method using $\lambda_1 = 12.$

\begin{figure}[!tp]
\def\figfolderA{experiments/exp_deconv_1D_traj/}
\def\figfolderB{experiments/exp_deconv_1D_S2_traj/}
\def\figfolderC{experiments/exp_deconv_1D_Pos3_traj/}
\def\hs{\hfill}
\def\vs{\vspace{0.03\textwidth}}
\def\figurewidth{0.3\textwidth}
\def\figurewidthB{0.15\textwidth}
\centering
{
\footnotesize
\begin{tabular}{ccc}
Data & 1st order interpol. wavelet  & DD wavelet \\[1ex]
\includegraphics[width=\figurewidth]{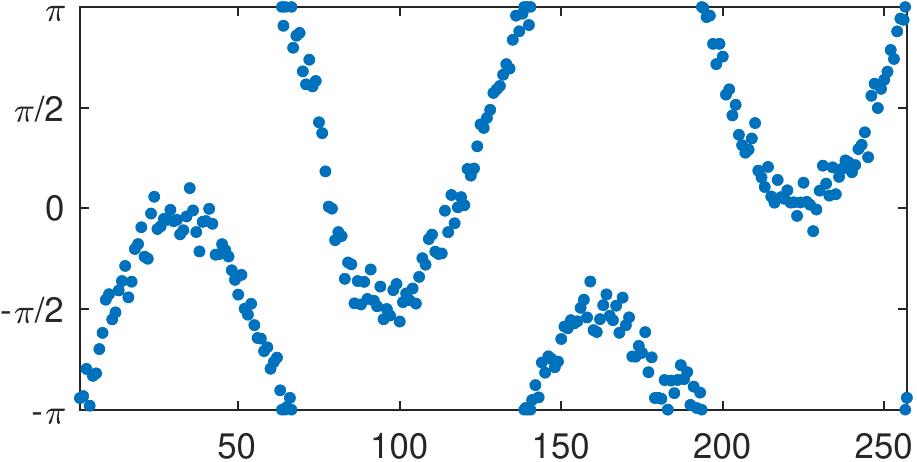}   &
\includegraphics[width=\figurewidth]{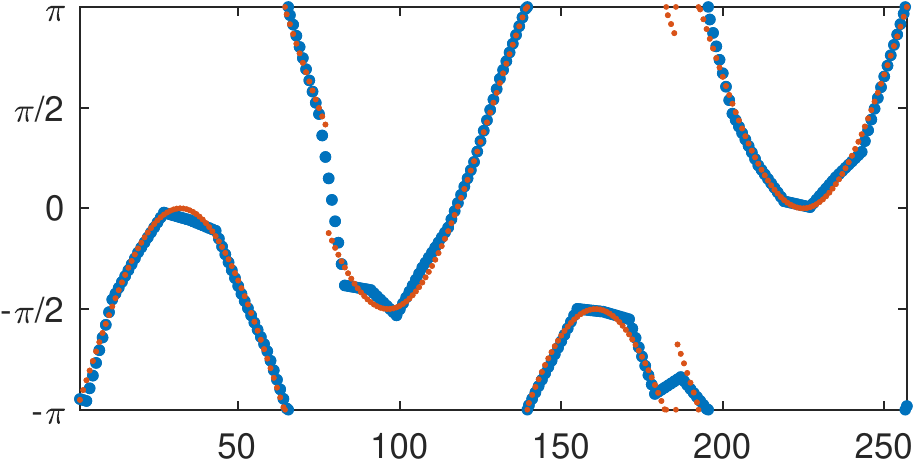} &
\includegraphics[width=\figurewidth]{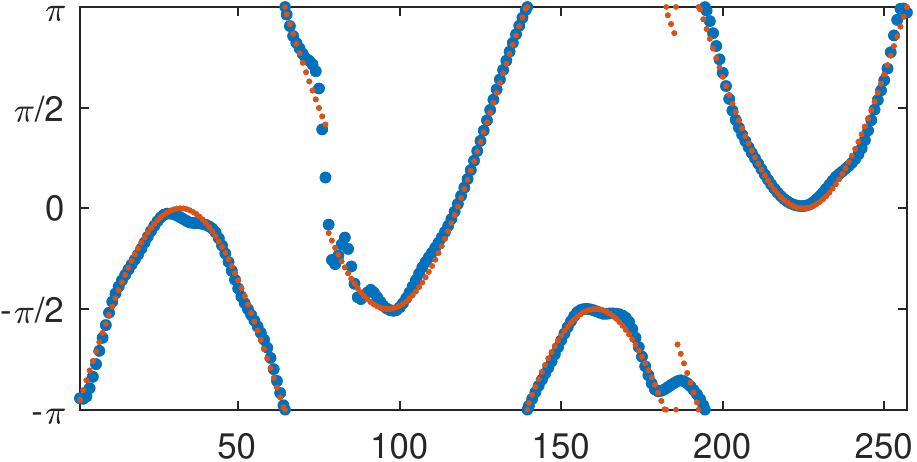}  \\
 &
$\deltaSNR$: $\input{\figfolderA exp_deconv_1D_traj_Haar_deltaSNR_1.txt}$ dB &
$\deltaSNR$: $\input{\figfolderA exp_deconv_1D_traj_Spline_deltaSNR_1.txt}$ dB \\[2ex]
\includegraphics[width=\figurewidth]{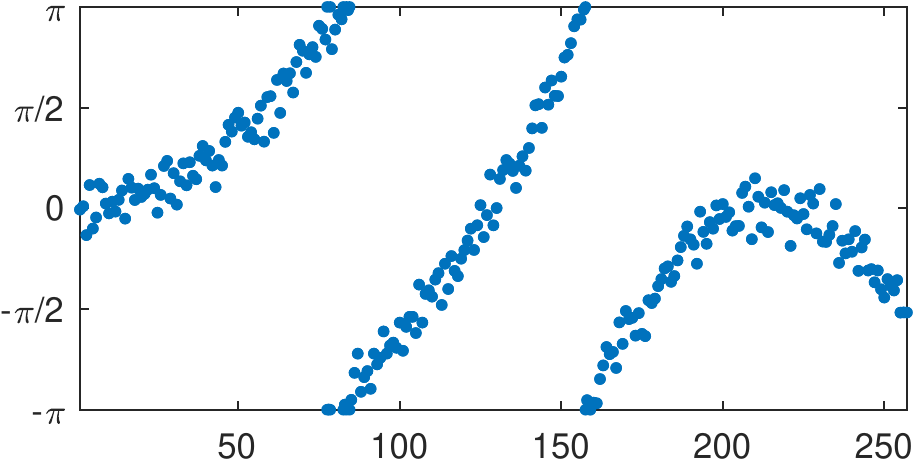}  &
\includegraphics[width=\figurewidth]{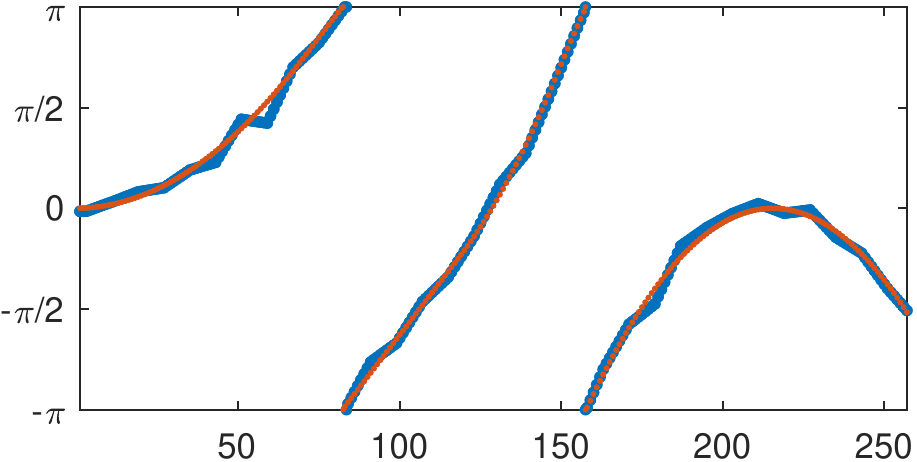} &
\includegraphics[width=\figurewidth]{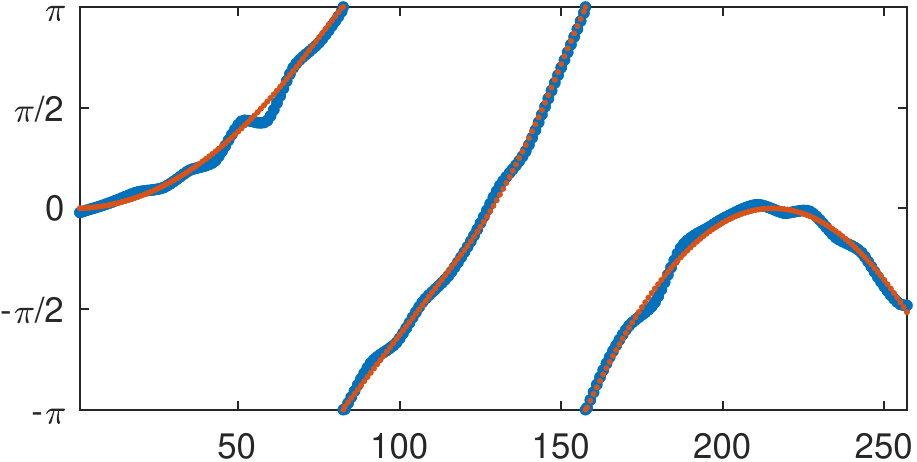} \\
&
$\deltaSNR$: $\input{\figfolderA exp_deconv_1D_traj_Haar_deltaSNR_2.txt}$ dB &
$\deltaSNR$: $\input{\figfolderA exp_deconv_1D_traj_Spline_deltaSNR_2.txt}$ dB \\[2ex]
\includegraphics[width=\figurewidth]{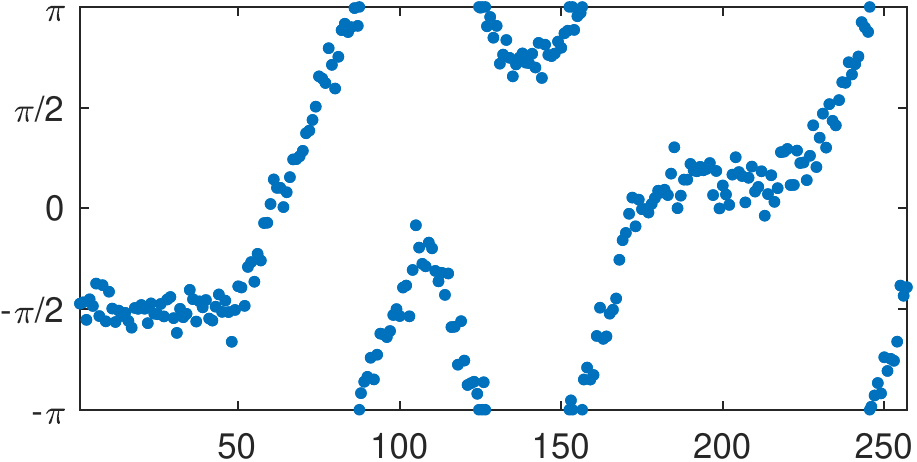}  &
\includegraphics[width=\figurewidth]{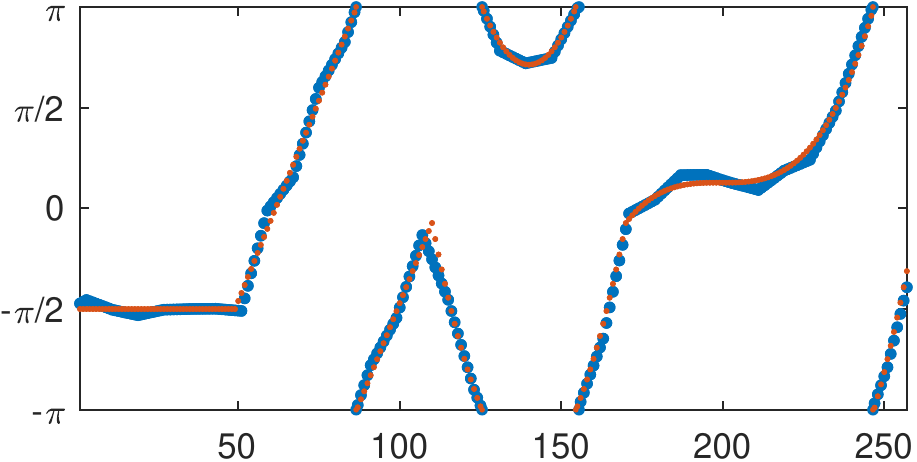} &
\includegraphics[width=\figurewidth]{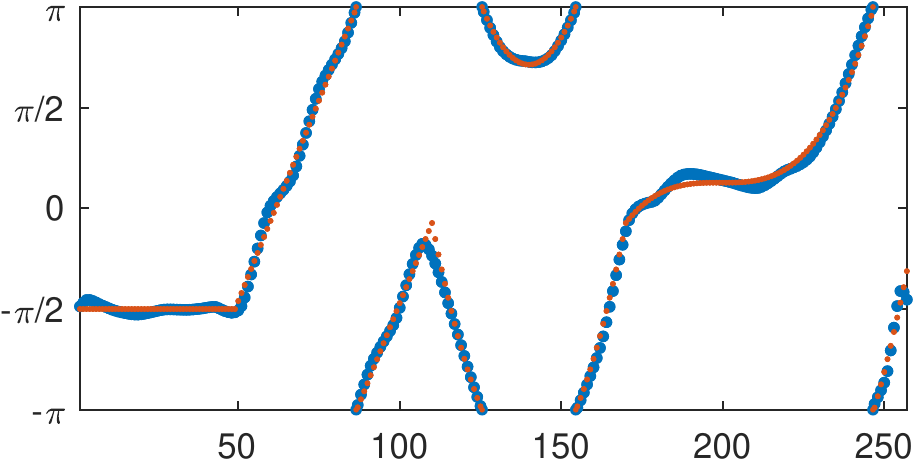} \\
&
$\deltaSNR$: $\input{\figfolderA exp_deconv_1D_traj_Haar_deltaSNR_3.txt}$ dB &
$\deltaSNR$: $\input{\figfolderA exp_deconv_1D_traj_Spline_deltaSNR_3.txt}$ dB \\[4ex]
\includegraphics[width=\figurewidthB]{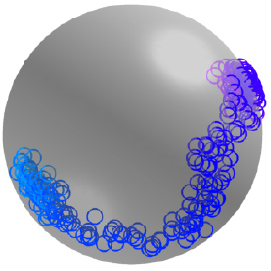}   &
\includegraphics[width=\figurewidthB]{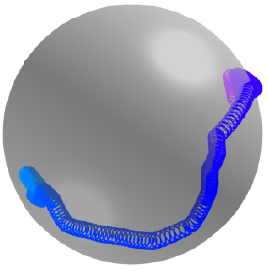} &
\includegraphics[width=\figurewidthB]{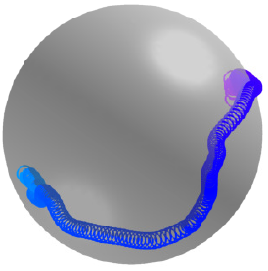}  
 \\
 &
$\deltaSNR$: $\input{\figfolderB exp_deconv_1D_S2_traj_Haar_deltaSNR_1.txt}$ dB &
$\deltaSNR$: $\input{\figfolderB exp_deconv_1D_S2_traj_Spline_deltaSNR_1.txt}$ dB \\[4ex]
\includegraphics[width=\figurewidthB]{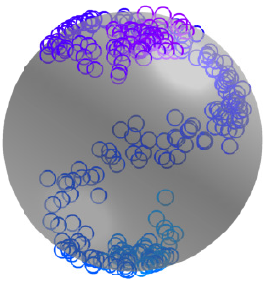}   &
\includegraphics[width=\figurewidthB]{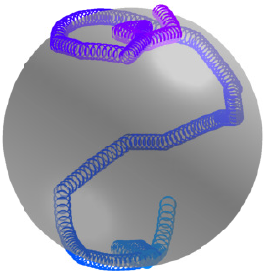} &
\includegraphics[width=\figurewidthB]{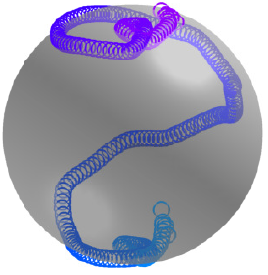}  
 \\
 &
$\deltaSNR$: $\input{\figfolderB exp_deconv_1D_S2_traj_Haar_deltaSNR_2.txt}$ dB &
$\deltaSNR$: $\input{\figfolderB exp_deconv_1D_S2_traj_Spline_deltaSNR_2.txt}$ dB \\[4ex]
\includegraphics[width=\figurewidth]{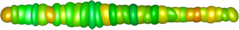}   &
\includegraphics[width=\figurewidth]{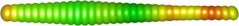} &
\includegraphics[width=\figurewidth]{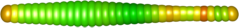}  
 \\
 &
$\deltaSNR$: $\input{\figfolderC exp_deconv_1D_Pos3_traj_Haar_deltaSNR_1.txt}$ dB &
$\deltaSNR$: $\input{\figfolderC exp_deconv_1D_Pos3_traj_Spline_deltaSNR_1.txt}$ dB \\[4ex]
\includegraphics[width=\figurewidth]{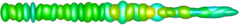}   &
\includegraphics[width=\figurewidth]{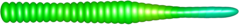} &
\includegraphics[width=\figurewidth]{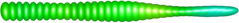}  
 \\
 &
$\deltaSNR$: $\input{\figfolderC exp_deconv_1D_Pos3_traj_Haar_deltaSNR_2.txt}$ dB &
$\deltaSNR$: $\input{\figfolderC exp_deconv_1D_Pos3_traj_Spline_deltaSNR_2.txt}$ dB 
\end{tabular}
}
\caption{
Results of the proposed variant for joint deconvolution and denoising. The given convolved noisy data is shown on the left, 
the result using the manifold analogue of the first order interpolatory wavelet and the third order Deslaurier-Dubuc (DD) wavelet
		along with the corresponding signal-to-noise-ratio improvement are shown in the center and on the right, respectively.
}
\label{fig:deconv}
\end{figure}

\begin{figure}[!tp]
\def\figfolderA{experiments/exp_denoising_2D_S1_2/}
\def\hs{\hfill}
\def\vs{\vspace{0.03\textwidth}}
\def\figurewidth{0.3\textwidth}
\def\figurewidthB{0.15\textwidth}
\centering
{
\footnotesize
\begin{tabular}{ccc}
\includegraphics[width=\figurewidth]{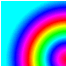} &
\includegraphics[width=\figurewidth]{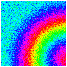} &
\includegraphics[width=\figurewidth]{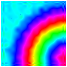} 
\end{tabular}
}
	\caption{Denoising of a synthetic $S^1$-valued image.
	\emph{Left:} Original, \emph{Center:} Noisy data, \emph{Right:} Denoising using $\ell^1$ regularization and the DD wavelet
	 ($\deltaSNR$: $\protect\input{\figfolderA exp_S1_deltaSNR.txt}$ dB).
	}
	\label{fig:s1_2D}
\end{figure}

\begin{figure}[!tp]
\def\figfolderB{experiments/exp_denoising_2D_Pos3_2/}
\def\hs{\hfill}
\def\vs{\vspace{0.03\textwidth}}
\def\figurewidth{0.3\textwidth}
\def\figurewidthB{0.15\textwidth}
\centering
{
\footnotesize
\begin{tabular}{ccc}
\includegraphics[width=\figurewidth]{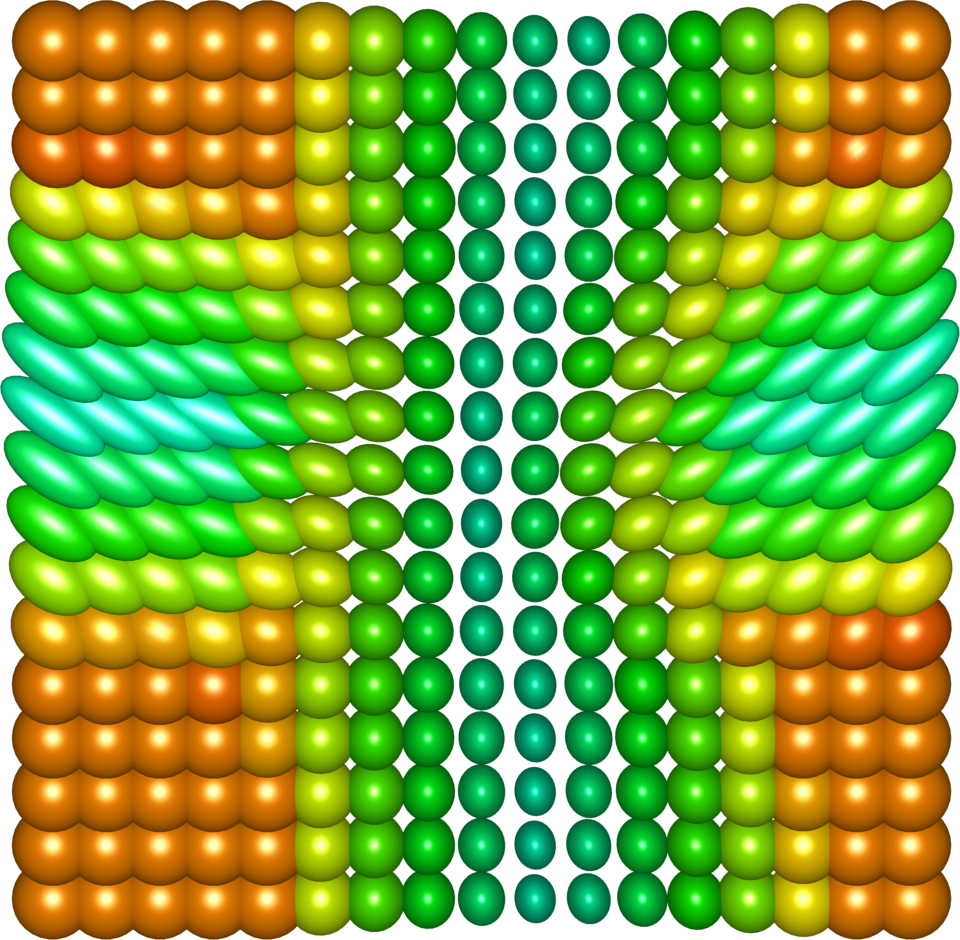} &
\includegraphics[width=\figurewidth]{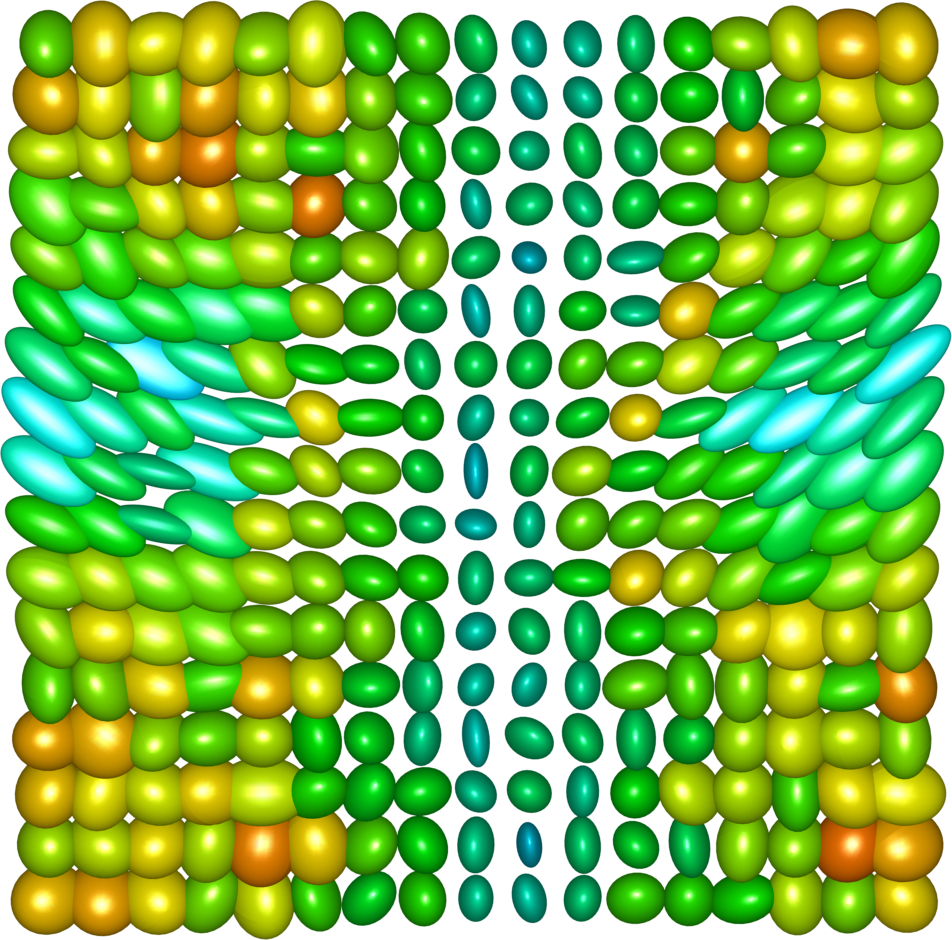} &
\includegraphics[width=\figurewidth]{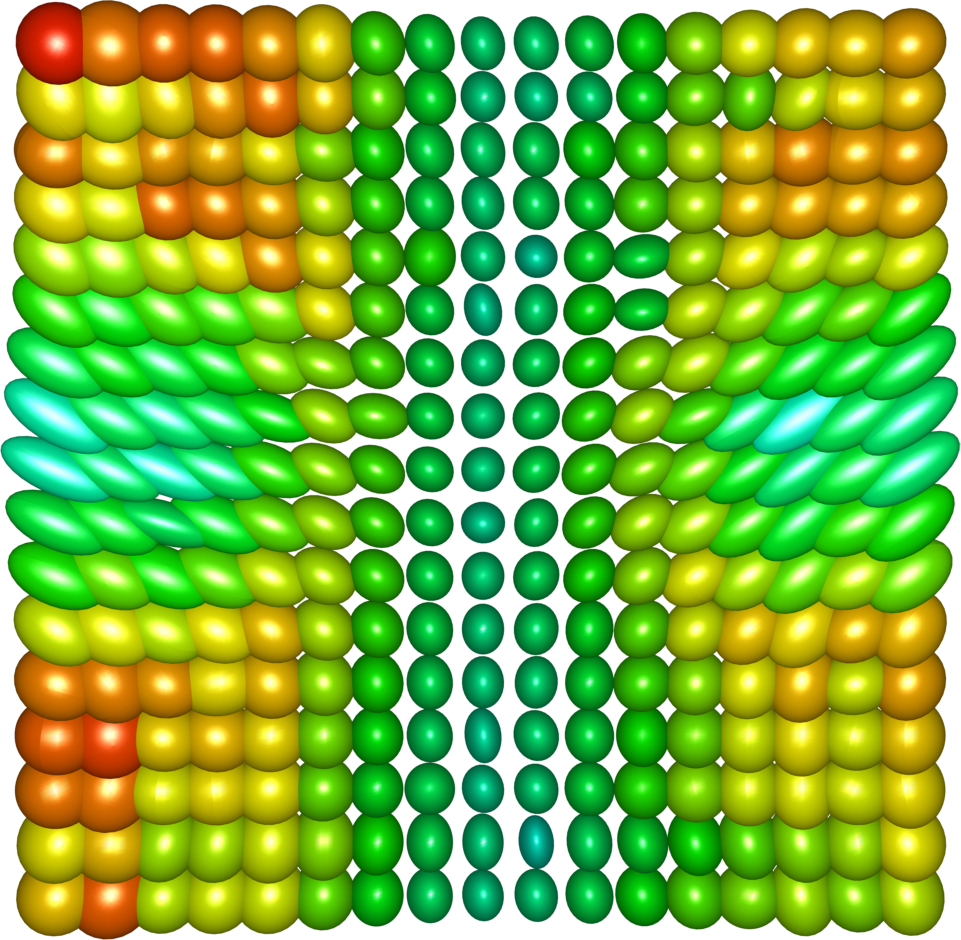} 
\end{tabular}
}
	\caption{Denoising of a synthetic $\Pos_3$-valued image.
	\emph{Left:} Original, \emph{Center:} Noisy data, \emph{Right:} Denoising using $\ell^1$ regularization and the DD wavelet
 ($\deltaSNR$: $\protect\input{\figfolderB exp_Pos3_deltaSNR.txt}$ dB).
}
	\label{fig:Pos3_2D}
\end{figure}

\begin{figure}
	\def\figfolderA{experiments/exp_camino_haar_2/}
\def\hs{\hfill}
\def\vs{\vspace{0.03\textwidth}}
\def\figurewidth{0.45\textwidth}
\def\figurewidthB{0.15\textwidth}
\centering
{
\footnotesize
\begin{tabular}{ccc}
\includegraphics[width=\figurewidth]{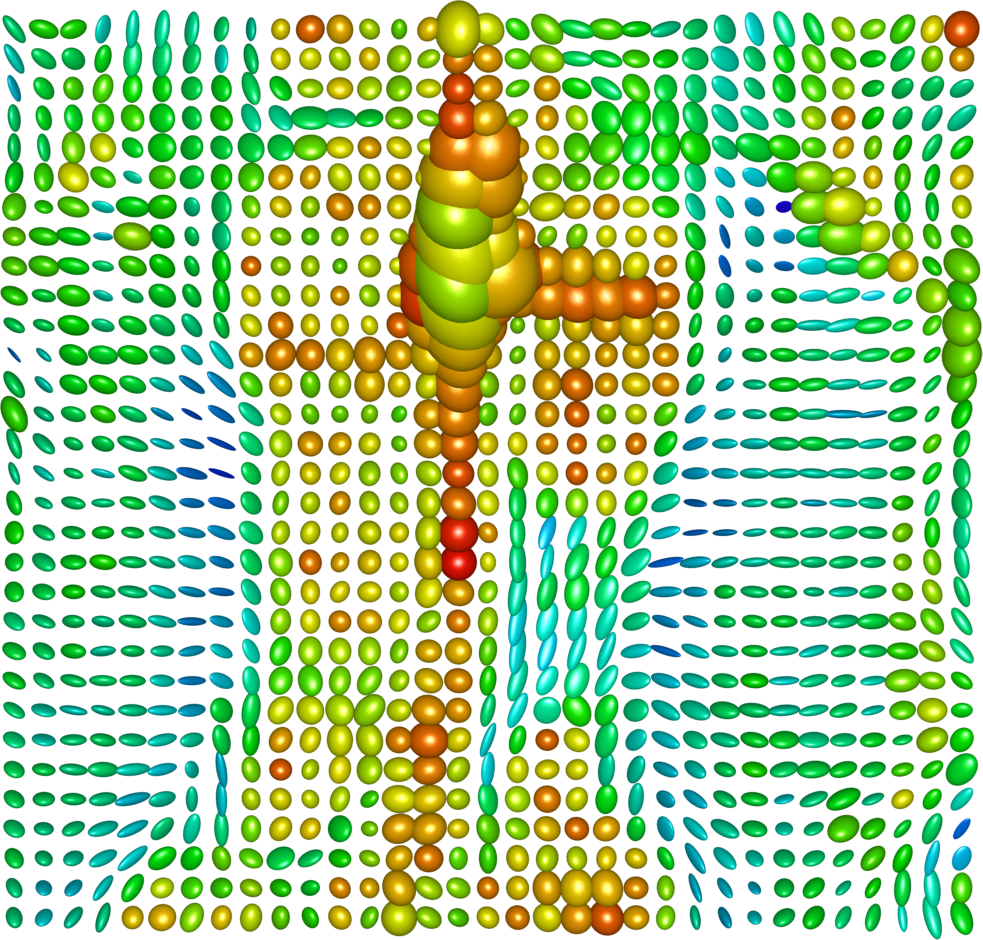} & &
\includegraphics[width=\figurewidth]{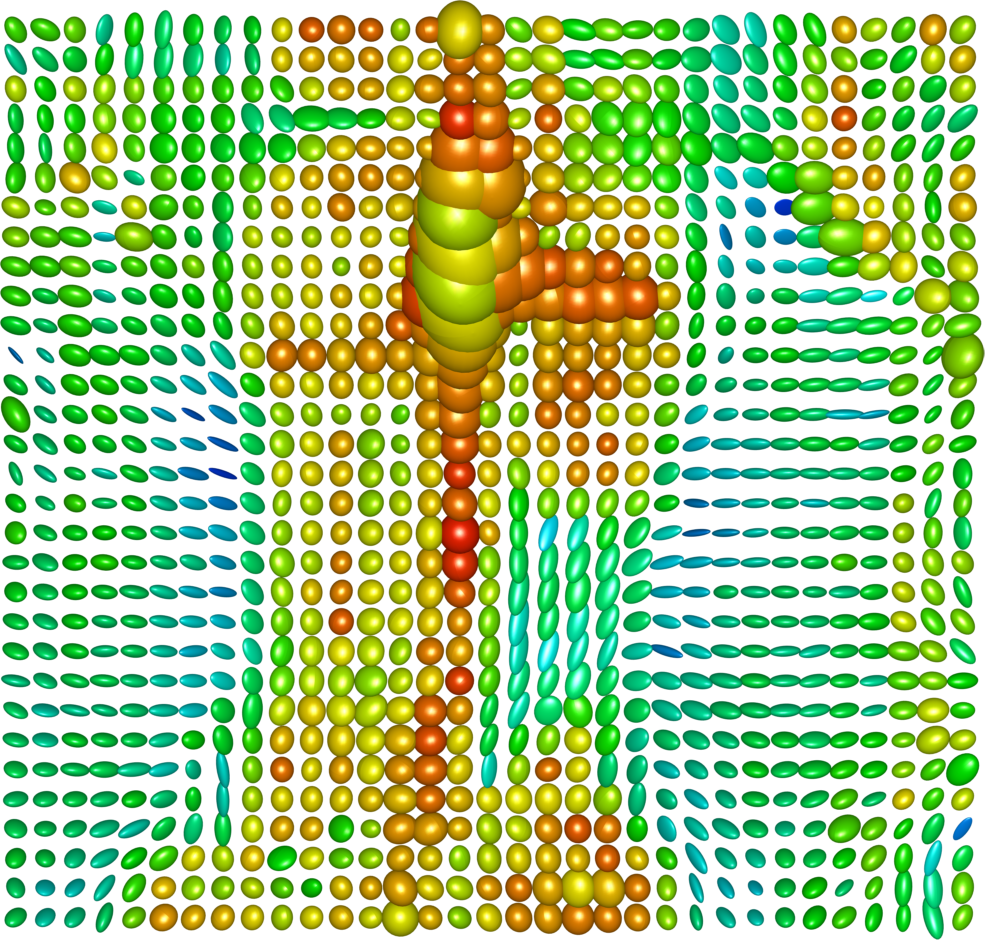}
\end{tabular}
}
	\caption{\emph{Left:} Section of a diffusion tensor image of the human brain; \emph{Right:} Results of the proposed method (using $\ell^1$ sparse regularization based on the DD wavelet).}
	\label{fig:camino}
\end{figure}

\section{Conclusion}\label{sec:Conclusion}

In this paper 
we have introduced and studied a variational model for wavelet sparse regularization in the manifold setup.
We have proposed a variational scheme using manifold valued interpolatory wavelets in the regularizing term, and, in particular, we consider a sparsity promoting $\ell^1$ type 
as well as an $\ell^0$ type regularizing term.  
We have shown the existence of minimizers for the proposed models.
Further, we have provided algorithms for the proposed models.
We have shown how to implement  the concepts of a generalized forward backward-scheme with Gauss-Seidel type update and a trajectory method as well as the well-established concept of a cyclic proximal point algorithm
for wavelet sparse regularization in the manifold setup.  
To this end, we have derived expressions for 
the (sub)gradients and proximal mappings of the atoms of the wavelet regularizing terms.
This includes the manifold analogues of $\ell^1$ and $\ell^0$ sparse wavelet regularization.
Finally, we have shown the potential of the proposed algorithms in the experimental section by applying them to data living in the unit circle, in the two-dimensional sphere as well as in the space of positive matrices.

\section*{Acknowledgment}

M. Storath was supported by the German Research Foundation DFG Grant STO1126/2-1.
A. Weinmann was supported by the German Research Foundation DFG Grants WE5886/4-1, WE5886/3-1.

{\small
	\bibliographystyle{plainnat}
	\bibliography{intwavvar}

\begin{thebibliography}{93}
\providecommand{\natexlab}[1]{#1}
\providecommand{\url}[1]{\texttt{#1}}
\expandafter\ifx\csname urlstyle\endcsname\relax
  \providecommand{\doi}[1]{doi: #1}\else
  \providecommand{\doi}{doi: \begingroup \urlstyle{rm}\Url}\fi

\bibitem[Absil et~al.(2004)Absil, Mahony, and Sepulchre]{absil2004riemannian}
P.-A. Absil, R.~Mahony, and R.~Sepulchre.
\newblock Riemannian geometry of {G}rassmann manifolds with a view on
  algorithmic computation.
\newblock \emph{Acta Applicandae Mathematicae}, 80\penalty0 (2):\penalty0
  199--220, 2004.

\bibitem[Absil et~al.(2009)Absil, Mahony, and Sepulchre]{absil2009optimization}
P.-A. Absil, R.~Mahony, and R.~Sepulchre.
\newblock \emph{Optimization algorithms on matrix manifolds}.
\newblock Princeton University Press, 2009.

\bibitem[Adato et~al.(2011)Adato, Zickler, and Ben-Shahar]{adato2011polar}
Y.~Adato, T.~Zickler, and O.~Ben-Shahar.
\newblock A polar representation of motion and implications for optical flow.
\newblock In \emph{IEEE Conference on Computer Vision and Pattern Recognition
  (CVPR)}, pages 1145--1152, 2011.

\bibitem[Afsari(2011)]{afsari2011riemannian}
B.~Afsari.
\newblock Riemannian $l^p$ center of mass: existence, uniqueness, and
  convexity.
\newblock \emph{Proceedings of the American Mathematical Society}, 139\penalty0
  (2):\penalty0 655--673, 2011.

\bibitem[Afsari et~al.(2013)Afsari, Tron, and Vidal]{afsari2013convergence}
B.~Afsari, R.~Tron, and R.~Vidal.
\newblock On the convergence of gradient descent for finding the riemannian
  center of mass.
\newblock \emph{SIAM Journal on Control and Optimization}, 51\penalty0
  (3):\penalty0 2230--2260, 2013.

\bibitem[Alexander et~al.(2007)Alexander, Lee, Lazar, Boudos, DuBray, Oakes,
  Miller, Lu, Jeong, McMahon, et~al.]{alexander2007diffusion}
A.~Alexander, J.~Lee, M.~Lazar, R.~Boudos, M.~DuBray, T.~Oakes, J.~Miller,
  J.~Lu, E.-K. Jeong, W.~McMahon, et~al.
\newblock Diffusion tensor imaging of the corpus callosum in autism.
\newblock \emph{Neuroimage}, 34\penalty0 (1):\penalty0 61--73, 2007.

\bibitem[Arnaudon et~al.(2013)Arnaudon, Barbaresco, and
  Yang]{arnaudon2013medians}
M.~Arnaudon, F.~Barbaresco, and L.~Yang.
\newblock Medians and means in {R}iemannian geometry: Existence, uniqueness and
  computation.
\newblock In \emph{Matrix Information Geometry}, pages 169--197. Springer,
  2013.

\bibitem[Azagra and Ferrera(2005)]{azagra2005proximal}
D.~Azagra and J.~Ferrera.
\newblock Proximal calculus on {R}iemannian manifolds.
\newblock \emph{Mediterranean Journal of Mathematics}, 2:\penalty0 437--450,
  2005.

\bibitem[Ba{\v{c}}{\'a}k(2013)]{bavcak2013proximal}
M.~Ba{\v{c}}{\'a}k.
\newblock The proximal point algorithm in metric spaces.
\newblock \emph{Israel Journal of Mathematics}, 194\penalty0 (2):\penalty0
  689--701, 2013.

\bibitem[Ba{\v{c}}{\'a}k(2014{\natexlab{a}})]{bacak2014convex}
M.~Ba{\v{c}}{\'a}k.
\newblock \emph{Convex analysis and optimization in Hadamard spaces}.
\newblock de Gruyter, 2014{\natexlab{a}}.

\bibitem[Ba{\v{c}}{\'a}k(2014{\natexlab{b}})]{bavcak2013computing}
M.~Ba{\v{c}}{\'a}k.
\newblock Computing medians and means in {H}adamard spaces.
\newblock \emph{SIAM Journal on Optimization}, 24\penalty0 (3):\penalty0
  1542--1566, 2014{\natexlab{b}}.

\bibitem[Ba{\v{c}}{\'a}k et~al.(2016)Ba{\v{c}}{\'a}k, Bergmann, Steidl, and
  Weinmann]{bavcak2016second}
M.~Ba{\v{c}}{\'a}k, R.~Bergmann, G.~Steidl, and A.~Weinmann.
\newblock A second order non-smooth variational model for restoring
  manifold-valued images.
\newblock \emph{SIAM Journal on Scientific Computing}, 38\penalty0
  (1):\penalty0 A567--A597, 2016.

\bibitem[Basser et~al.(1994)Basser, Mattiello, and LeBihan]{basser1994mr}
P.~Basser, J.~Mattiello, and D.~LeBihan.
\newblock {M}{R} diffusion tensor spectroscopy and imaging.
\newblock \emph{Biophysical Journal}, 66\penalty0 (1):\penalty0 259--267, 1994.

\bibitem[Basu et~al.(2006)Basu, Fletcher, and Whitaker]{basu2006rician}
S.~Basu, T.~Fletcher, and R.~Whitaker.
\newblock Rician noise removal in diffusion tensor {MRI}.
\newblock In \emph{Medical Image Computing and Computer-Assisted Intervention
  2006}, pages 117--125. Springer, 2006.

\bibitem[Baust et~al.(2015)Baust, Demaret, Storath, Navab, and
  Weinmann]{baust2015total}
M.~Baust, L.~Demaret, M.~Storath, N.~Navab, and A.~Weinmann.
\newblock Total variation regularization of shape signals.
\newblock In \emph{IEEE Conference on Computer Vision and Pattern Recognition
  (CVPR)}, pages 2075--2083, 2015.

\bibitem[Baust et~al.(2016)Baust, Weinmann, Wieczorek, Lasser, Storath, and
  Navab]{baust2016combined}
M.~Baust, A.~Weinmann, M.~Wieczorek, T.~Lasser, M.~Storath, and N.~Navab.
\newblock Combined tensor fitting and {TV} regularization in diffusion tensor
  imaging based on a {R}iemannian manifold approach.
\newblock \emph{IEEE Transactions on Medical Imaging}, 35\penalty0
  (8):\penalty0 1972--1989, 2016.

\bibitem[Berens(2009)]{berens2009circstat}
P.~Berens.
\newblock Circstat: a {MATLAB} toolbox for circular statistics.
\newblock \emph{J Stat Softw}, 31\penalty0 (10):\penalty0 1--21, 2009.

\bibitem[Bergmann et~al.(2014)Bergmann, Laus, Steidl, and
  Weinmann]{bergmann2014second}
R.~Bergmann, F.~Laus, G.~Steidl, and A.~Weinmann.
\newblock Second order differences of cyclic data and applications in
  variational denoising.
\newblock \emph{SIAM Journal on Imaging Sciences}, 7\penalty0 (4):\penalty0
  2916--2953, 2014.

\bibitem[Bergmann et~al.(2016)Bergmann, Chan, Hielscher, Persch, and
  Steidl]{steidl16half_quadratic}
R.~Bergmann, R.~H. Chan, R.~Hielscher, J.~Persch, and G.~Steidl.
\newblock Restoration of manifold-valued images by half-quadratic minimization.
\newblock \emph{Inverse Problems and Imaging}, 10:\penalty0 281–304, 2016.

\bibitem[Berkels et~al.(2013)Berkels, Fletcher, Heeren, Rumpf, and
  Wirth]{berkels2013discrete}
B.~Berkels, P.~Fletcher, B.~Heeren, M.~Rumpf, and B.~Wirth.
\newblock Discrete geodesic regression in shape space.
\newblock In \emph{International Workshop on Energy Minimization Methods in
  Computer Vision and Pattern Recognition}, pages 108--122. Springer, 2013.

\bibitem[Bertsekas(2011)]{Bertsekas2011in}
D.~Bertsekas.
\newblock Incremental proximal methods for large scale convex optimization.
\newblock \emph{Mathematical Programming}, 129:\penalty0 163--195, 2011.

\bibitem[Bhattacharya and Patrangenaru(2003)]{bhattacharya2003large}
R.~Bhattacharya and V.~Patrangenaru.
\newblock Large sample theory of intrinsic and extrinsic sample means on
  manifolds {I}.
\newblock \emph{Annals of Statistics}, 31\penalty0 (1):\penalty0 1--29, 2003.

\bibitem[Bhattacharya and Patrangenaru(2005)]{bhattacharya2005large}
R.~Bhattacharya and V.~Patrangenaru.
\newblock Large sample theory of intrinsic and extrinsic sample means on
  manifolds {II}.
\newblock \emph{Annals of Statistics}, 33\penalty0 (3):\penalty0 1225--1259,
  2005.

\bibitem[Boumal et~al.(2014)Boumal, Mishra, Absil, and Sepulchre]{manopt}
N.~Boumal, B.~Mishra, P.-A. Absil, and R.~Sepulchre.
\newblock {M}anopt, a {M}atlab toolbox for optimization on manifolds.
\newblock \emph{Journal of Machine Learning Research}, 15:\penalty0 1455--1459,
  2014.
\newblock URL \url{http://www.manopt.org}.

\bibitem[Bredies et~al.(2018)Bredies, Holler, Storath, and
  Weinmann]{bredies2017total}
K.~Bredies, M.~Holler, M.~Storath, and A.~Weinmann.
\newblock Total generalized variation for manifold-valued data.
\newblock \emph{SIAM Journal on Imaging Sciences}, 11\penalty0 (3):\penalty0
  1785--1848, 2018.

\bibitem[Chambolle et~al.(1998)Chambolle, De~Vore, Lee, and
  Lucier]{chambolle1998nonlinear}
A.~Chambolle, R.~De~Vore, N.~Lee, and B.~Lucier.
\newblock Nonlinear wavelet image processing: variational problems,
  compression, and noise removal through wavelet shrinkage.
\newblock \emph{IEEE Transactions on Image Processing}, 7\penalty0
  (3):\penalty0 319--335, 1998.

\bibitem[Chan et~al.(2001)Chan, Kang, and Shen]{chan2001total}
T.~Chan, S.~Kang, and J.~Shen.
\newblock Total variation denoising and enhancement of color images based on
  the {CB} and {HSV} color models.
\newblock \emph{Journal of Visual Communication and Image Representation},
  12:\penalty0 422--435, 2001.

\bibitem[Cheeger and Ebin(1975)]{cheeger1975comparison}
J.~Cheeger and D.~Ebin.
\newblock \emph{Comparison theorems in Riemannian geometry}, volume~9.
\newblock North-Holland, 1975.

\bibitem[Chefd'Hotel et~al.(2004)Chefd'Hotel, Tschumperl{\'e}, Deriche, and
  Faugeras]{chefd2004regularizing}
C.~Chefd'Hotel, D.~Tschumperl{\'e}, R.~Deriche, and O.~Faugeras.
\newblock Regularizing flows for constrained matrix-valued images.
\newblock \emph{Journal of Mathematical Imaging and Vision}, 20\penalty0
  (1-2):\penalty0 147--162, 2004.

\bibitem[Cohen et~al.(1993)Cohen, Daubechies, and Vial]{cohen1993wavelets}
A.~Cohen, I.~Daubechies, and P.~Vial.
\newblock Wavelets on the interval and fast wavelet transforms.
\newblock \emph{Applied and computational harmonic analysis}, 1993.

\bibitem[Cook et~al.(2006)Cook, Bai, Nedjati-Gilani, Seunarine, Hall, Parker,
  and Alexander]{cook2006camino}
P.~Cook, Y.~Bai, S.~Nedjati-Gilani, K.~Seunarine, M.~Hall, G.~Parker, and
  D.~Alexander.
\newblock Camino: Open-source diffusion-{MRI} reconstruction and processing.
\newblock In \emph{14th Scientific Meeting of the International Society for
  Magnetic Resonance in Medicine}, page 2759, 2006.

\bibitem[Deslauriers and Dubuc(1989)]{deslauriers1989symmetric}
G.~Deslauriers and S.~Dubuc.
\newblock Symmetric iterative interpolation processes.
\newblock In \emph{Constructive approximation}, pages 49--68. Springer, 1989.

\bibitem[do~Carmo(1992)]{do1992riemannian}
M.~do~Carmo.
\newblock \emph{Riemannian geometry.}
\newblock Birkhauser, 1992.

\bibitem[Donoho(1992)]{donoho1992interpolating}
D.~Donoho.
\newblock Interpolating wavelet transforms.
\newblock \emph{Preprint, Department of Statistics, Stanford University},
  2\penalty0 (3), 1992.

\bibitem[Donoho(1995)]{donoho1995noising}
D.~Donoho.
\newblock De-noising by soft-thresholding.
\newblock \emph{IEEE Transactions on Information Theory}, 41\penalty0
  (3):\penalty0 613--627, 1995.

\bibitem[Dyn et~al.(1987)Dyn, Levin, and Gregory]{dyn19874}
N.~Dyn, D.~Levin, and J.~Gregory.
\newblock A 4-point interpolatory subdivision scheme for curve design.
\newblock \emph{Computer aided geometric design}, 4\penalty0 (4):\penalty0
  257--268, 1987.

\bibitem[Ferreira and Oliveira(2002)]{ferreira2002proximal}
O.~Ferreira and P.~Oliveira.
\newblock Proximal point algorithm on {R}iemannian manifolds.
\newblock \emph{Optimization}, 51:\penalty0 257--270, 2002.

\bibitem[Ferreira et~al.(2013)Ferreira, Xavier, Costeira, and
  Barroso]{ferreira2013newton}
R.~Ferreira, J.~Xavier, J.~Costeira, and V.~Barroso.
\newblock Newton algorithms for {R}iemannian distance related problems on
  connected locally symmetric manifolds.
\newblock \emph{IEEE Journal of Selected Topics in Signal Processing},
  7:\penalty0 634--645, 2013.

\bibitem[Fletcher and Joshi(2007)]{fletcher2007riemannian}
P.~Fletcher and S.~Joshi.
\newblock {R}iemannian geometry for the statistical analysis of diffusion
  tensor data.
\newblock \emph{Signal Processing}, 87:\penalty0 250--262, 2007.

\bibitem[Fletcher et~al.(2004)Fletcher, Lu, Pizer, and
  Joshi]{fletcher2004principal}
P.~Fletcher, C.~Lu, S.~Pizer, and S.~Joshi.
\newblock Principal geodesic analysis for the study of nonlinear statistics of
  shape.
\newblock \emph{IEEE Transactions on Medical Imaging}, 23\penalty0
  (8):\penalty0 995--1005, 2004.

\bibitem[Foong et~al.(2000)Foong, Maier, Clark, Barker, Miller, and
  Ron]{foong2000neuropathological}
J.~Foong, M.~Maier, C.~Clark, G.~Barker, D.~Miller, and M.~Ron.
\newblock Neuropathological abnormalities of the corpus callosum in
  schizophrenia: a diffusion tensor imaging study.
\newblock \emph{Journal of Neurology, Neurosurgery \& Psychiatry}, 68\penalty0
  (2):\penalty0 242--244, 2000.

\bibitem[Giaquinta and Mucci(2006)]{GM06}
M.~Giaquinta and D.~Mucci.
\newblock The {BV}-energy of maps into a manifold: relaxation and density
  results.
\newblock \emph{Annali della Scuola Normale Superiore di Pisa-Classe di
  Scienze}, 5\penalty0 (4):\penalty0 483--548, 2006.

\bibitem[Giaquinta and Mucci(2007)]{GM07}
M.~Giaquinta and D.~Mucci.
\newblock Maps of bounded variation with values into a manifold: total
  variation and relaxed energy.
\newblock \emph{Pure and Applied Mathematics Quarterly}, 3\penalty0
  (2):\penalty0 513--538, 2007.

\bibitem[Giaquinta et~al.(1993)Giaquinta, Modica, and Sou{\v c}ek]{GMS93}
M.~Giaquinta, G.~Modica, and J.~Sou{\v c}ek.
\newblock Variational problems for maps of bounded variation with values in
  {$S^1$}.
\newblock \emph{Calculus of Variations and Partial Differential Equations},
  1\penalty0 (1):\penalty0 87--121, 1993.

\bibitem[Grohs(2008)]{grohs2008smoothness}
P.~Grohs.
\newblock Smoothness analysis of subdivision schemes on regular grids by
  proximity.
\newblock \emph{SIAM Journal on Numerical Analysis}, 46\penalty0 (4):\penalty0
  2169--2182, 2008.

\bibitem[Grohs(2010{\natexlab{a}})]{grohs2010general}
P.~Grohs.
\newblock A general proximity analysis of nonlinear subdivision schemes.
\newblock \emph{SIAM Journal on Mathematical Analysis}, 42\penalty0
  (2):\penalty0 729--750, 2010{\natexlab{a}}.

\bibitem[Grohs(2010{\natexlab{b}})]{grohs2010stability}
P.~Grohs.
\newblock Stability of manifold-valued subdivision schemes and multiscale
  transformations.
\newblock \emph{Constructive Approximation}, 32\penalty0 (3):\penalty0
  569--596, 2010{\natexlab{b}}.

\bibitem[Grohs and Hosseini(2016)]{grohs2016varepsilon}
P.~Grohs and S.~Hosseini.
\newblock $\varepsilon$-subgradient algorithms for locally {L}ipschitz
  functions on {R}iemannian manifolds.
\newblock \emph{Advances in Computational Mathematics}, 42\penalty0
  (2):\penalty0 333--360, 2016.

\bibitem[Grohs and Sprecher(2016)]{grohs2016total}
P.~Grohs and M.~Sprecher.
\newblock Total variation regularization on {R}iemannian manifolds by
  iteratively reweighted minimization.
\newblock \emph{Information and Inference}, 5\penalty0 (4):\penalty0 353--378,
  2016.

\bibitem[Grohs and Wallner(2009)]{grohs2009interpolatory}
P.~Grohs and J.~Wallner.
\newblock Interpolatory wavelets for manifold-valued data.
\newblock \emph{Applied and Computational Harmonic Analysis}, 27\penalty0
  (3):\penalty0 325--333, 2009.

\bibitem[Grohs et~al.(2015)Grohs, Hardering, and Sander]{grohs2013optimal}
P.~Grohs, H.~Hardering, and O.~Sander.
\newblock Optimal a priori discretization error bounds for geodesic finite
  elements.
\newblock \emph{Foundations of Computational Mathematics}, 15\penalty0
  (6):\penalty0 1357--1411, 2015.

\bibitem[Han(2002)]{HanComputingSmooth}
B.~Han.
\newblock Computing the smoothness exponent of a symmetric multivariate
  refinable function.
\newblock \emph{SIAM J. Matrix Anal. Appl.}, 24:\penalty0 693--714, 2002.
\newblock ISSN 0895-4798.
\newblock \doi{http://dx.doi.org/10.1137/S0895479801390868}.

\bibitem[Han(2006)]{HanSol}
B.~Han.
\newblock Solutions in {S}obolev spaces of vector refinement equations with a
  general dilation matrix.
\newblock \emph{Adv. Comput. Math.}, 24:\penalty0 375--403, 2006.

\bibitem[Han and Jia(1998)]{HanConvergence}
B.~Han and R.-Q. Jia.
\newblock Multivariate refinement equations and convergence of subdivision
  schemes.
\newblock \emph{SIAM J. Math. Anal.}, 29:\penalty0 1177--1199, 1998.
\newblock ISSN 0036-1410.
\newblock \doi{http://dx.doi.org/10.1137/S0036141097294032}.

\bibitem[Hawe et~al.(2013)Hawe, Seibert, and Kleinsteuber]{hawe2013separable}
S.~Hawe, M.~Seibert, and M.~Kleinsteuber.
\newblock Separable dictionary learning.
\newblock In \emph{Proceedings of the IEEE Conference on Computer Vision and
  Pattern Recognition}, pages 438--445, 2013.

\bibitem[Johansen-Berg and Behrens(2009)]{johansen2009diffusion}
H.~Johansen-Berg and T.~Behrens.
\newblock \emph{Diffusion MRI: From quantitative measurement to in-vivo
  neuroanatomy}.
\newblock Academic Press, London, 2009.

\bibitem[Karcher(1977)]{karcher1977riemannian}
H.~Karcher.
\newblock {R}iemannian center of mass and mollifier smoothing.
\newblock \emph{Communications on Pure and Applied Mathematics}, 30:\penalty0
  509--541, 1977.

\bibitem[Kendall(1990)]{kendall1990probability}
W.~Kendall.
\newblock Probability, convexity, and harmonic maps with small image {I}:
  uniqueness and fine existence.
\newblock \emph{Proceedings of the London Mathematical Society}, 3:\penalty0
  371--406, 1990.

\bibitem[Kimmel and Sochen(2002)]{kimmel2002orientation}
R.~Kimmel and N.~Sochen.
\newblock Orientation diffusion or how to comb a porcupine.
\newblock \emph{Journal of Visual Communication and Image Representation},
  13\penalty0 (1):\penalty0 238--248, 2002.

\bibitem[Lai and Osher(2014)]{lai2014splitting}
R.~Lai and S.~Osher.
\newblock A splitting method for orthogonality constrained problems.
\newblock \emph{Journal of Scientific Computing}, 58\penalty0 (2):\penalty0
  431--449, 2014.

\bibitem[Lellmann et~al.(2013)Lellmann, Strekalovskiy, Koetter, and
  Cremers]{LSKC13}
J.~Lellmann, E.~Strekalovskiy, S.~Koetter, and D.~Cremers.
\newblock Total variation regularization for functions with values in a
  manifold.
\newblock In \emph{International Conference on Computer Vision (ICCV)}, pages
  2944--2951, 2013.

\bibitem[Mardia and Jupp(2009)]{mardia2009directional}
K.~Mardia and P.~Jupp.
\newblock \emph{Directional statistics}, volume 494.
\newblock John Wiley \& Sons, 2009.

\bibitem[Massonnet and Feigl(1998)]{massonnet1998radar}
D.~Massonnet and K.~Feigl.
\newblock Radar interferometry and its application to changes in the earth's
  surface.
\newblock \emph{Reviews of Geophysics}, 36\penalty0 (4):\penalty0 441--500,
  1998.

\bibitem[Michor and Mumford(2007)]{Michor07}
P.~Michor and D.~Mumford.
\newblock An overview of the {R}iemannian metrics on spaces of curves using the
  {H}amiltonian approach.
\newblock \emph{Applied and Computational Harmonic Analysis}, 23\penalty0
  (1):\penalty0 74 -- 113, 2007.
\newblock ISSN 1063-5203.

\bibitem[Moreau(1962)]{moreau1962fonctions}
J.-J. Moreau.
\newblock Fonctions convexes duales et points proximaux dans un espace
  hilbertien.
\newblock \emph{Comptes Rendus de l'Académie des Sciences. Series A
  Mathematics.}, 255:\penalty0 2897--2899, 1962.

\bibitem[Oller and Corcuera(1995)]{oller1995intrinsic}
J.~Oller and J.~Corcuera.
\newblock Intrinsic analysis of statistical estimation.
\newblock \emph{Annals of Statistics}, 23\penalty0 (5):\penalty0 1562--1581,
  1995.

\bibitem[Pennec(2006)]{pennec2006intrinsic}
X.~Pennec.
\newblock Intrinsic statistics on {R}iemannian manifolds: {B}asic tools for
  geometric measurements.
\newblock \emph{Journal of Mathematical Imaging and Vision}, 25\penalty0
  (1):\penalty0 127--154, 2006.

\bibitem[Pennec et~al.(2006)Pennec, Fillard, and Ayache]{pennec2006riemannian}
X.~Pennec, P.~Fillard, and N.~Ayache.
\newblock A {R}iemannian framework for tensor computing.
\newblock \emph{International Journal of Computer Vision}, 66:\penalty0 41--66,
  2006.

\bibitem[Rao(1945)]{radhakrishna1945information}
C.~Rao.
\newblock Information and accuracy attainable in the estimation of statistical
  parameters.
\newblock \emph{Bulletin of the Calcutta Mathematical Society}, 37\penalty0
  (3):\penalty0 81--91, 1945.

\bibitem[Rosman et~al.(2012)Rosman, Bronstein, Bronstein, Wolf, and
  Kimmel]{rosman2012group}
G.~Rosman, M.~Bronstein, A.~Bronstein, A.~Wolf, and R.~Kimmel.
\newblock Group-valued regularization framework for motion segmentation of
  dynamic non-rigid shapes.
\newblock In \emph{Scale Space and Variational Methods in Computer Vision},
  pages 725--736. Springer, 2012.

\bibitem[Sander(2015)]{sander2015geodesic}
O.~Sander.
\newblock Geodesic finite elements of higher order.
\newblock \emph{IMA Journal of Numerical Analysis}, 36\penalty0 (1):\penalty0
  238--266, 2015.

\bibitem[Schultz(1990)]{schultz1990circular}
H.~Schultz.
\newblock A circular median filter approach for resolving directional
  ambiguities in wind fields retrieved from spaceborne scatterometer data.
\newblock \emph{Journal of Geophysical Research: Oceans}, 95\penalty0
  (C4):\penalty0 5291--5303, 1990.

\bibitem[Spivak(1975)]{spivak1975differential}
M.~Spivak.
\newblock \emph{Differential Geometry}.
\newblock Publish or Perish, Berkeley, 1975.

\bibitem[Stefanoiu et~al.(2016)Stefanoiu, Weinmann, Storath, Navab, and
  Baust]{stefanoiu2016joint}
A.~Stefanoiu, A.~Weinmann, M.~Storath, N.~Navab, and M.~Baust.
\newblock Joint segmentation and shape regularization with a generalized
  forward--backward algorithm.
\newblock \emph{IEEE Transactions on Image Processing}, 25\penalty0
  (7):\penalty0 3384--3394, 2016.

\bibitem[Storath and Weinmann(2018{\natexlab{a}})]{storath2017fastMedian}
M.~Storath and A.~Weinmann.
\newblock Fast median filtering for phase or orientation data.
\newblock \emph{IEEE Transactions on Pattern Analysis and Machine
  Intelligence}, 40\penalty0 (3):\penalty0 639--652, 2018{\natexlab{a}}.

\bibitem[Storath and Weinmann(2018{\natexlab{b}})]{storath2018variational}
M~Storath and A~Weinmann.
\newblock Variational regularization of inverse problems for manifold-valued
  data.
\newblock \emph{arXiv preprint arXiv:1804.10432}, 2018{\natexlab{b}}.

\bibitem[Storath et~al.(2016)Storath, Weinmann, and Unser]{storath2016exact}
M.~Storath, A.~Weinmann, and M.~Unser.
\newblock Exact algorithms for {$L^1$}-{TV} regularization of real-valued or
  circle-valued signals.
\newblock \emph{SIAM Journal on Scientific Computing}, 38\penalty0
  (1):\penalty0 A614--A630, 2016.

\bibitem[Storath et~al.(2017)Storath, Weinmann, and Unser]{storath2017jump}
M.~Storath, A.~Weinmann, and M.~Unser.
\newblock Jump-penalized least absolute values estimation of scalar or
  circle-valued signals.
\newblock \emph{Information and Inference}, 6\penalty0 (3):\penalty0 225--245,
  2017.

\bibitem[Strekalovskiy and Cremers(2011)]{SC11}
E.~Strekalovskiy and D.~Cremers.
\newblock {Total variation for cyclic structures: Convex relaxation and
  efficient minimization}.
\newblock In \emph{IEEE Conference on Computer Vision and Pattern Recognition
  (CVPR)}, pages 1905--1911, 2011.

\bibitem[Tibshirani(1996)]{tibshirani1996regression}
R.~Tibshirani.
\newblock Regression shrinkage and selection via the lasso.
\newblock \emph{Journal of the Royal Statistical Society. Series B
  (Methodological)}, pages 267--288, 1996.

\bibitem[Tron et~al.(2013)Tron, Afsari, and Vidal]{tron2013riemannian}
R.~Tron, B.~Afsari, and R.~Vidal.
\newblock Riemannian consensus for manifolds with bounded curvature.
\newblock \emph{IEEE Transactions on Automatic Control}, 58\penalty0
  (4):\penalty0 921--934, 2013.

\bibitem[Tschumperl{\'e} and Deriche(2001)]{tschumperle2001diffusion}
D.~Tschumperl{\'e} and R.~Deriche.
\newblock Diffusion tensor regularization with constraints preservation.
\newblock In \emph{IEEE Conference on Computer Vision and Pattern Recognition
  (CVPR)}, pages I948--I953, 2001.

\bibitem[Ur~Rahman et~al.(2005)Ur~Rahman, Drori, Donoho, and
  Schr{\"o}der]{rahman2005multiscale}
I.~Ur~Rahman, V.~Drori, I.and~Stodden, D.~Donoho, and P.~Schr{\"o}der.
\newblock Multiscale representations for manifold-valued data.
\newblock \emph{Multiscale Modeling \& Simulation}, 4\penalty0 (4):\penalty0
  1201--1232, 2005.

\bibitem[Vese and Osher(2002)]{vese2002numerical}
L.~Vese and S.~Osher.
\newblock Numerical methods for p-harmonic flows and applications to image
  processing.
\newblock \emph{SIAM Journal on Numerical Analysis}, 40:\penalty0 2085--2104,
  2002.

\bibitem[Wallner and Dyn(2005)]{wallner2005convergence}
J.~Wallner and N.~Dyn.
\newblock Convergence and {C1} analysis of subdivision schemes on manifolds by
  proximity.
\newblock \emph{Computer Aided Geometric Design}, 22\penalty0 (7):\penalty0
  593--622, 2005.

\bibitem[Wallner et~al.(2011)Wallner, Yazdani, and
  Weinmann]{wallner2011convergence}
J.~Wallner, E.~Yazdani, and A.~Weinmann.
\newblock Convergence and smoothness analysis of subdivision rules in
  {R}iemannian and symmetric spaces.
\newblock \emph{Advances in Computational Mathematics}, 34\penalty0
  (2):\penalty0 201--218, 2011.

\bibitem[Weaver et~al.(1991)Weaver, Xu, Healy, and
  Cromwell]{weaver1991filtering}
J.~Weaver, Y.~Xu, D.~Healy, and L.~Cromwell.
\newblock Filtering noise from images with wavelet transforms.
\newblock \emph{Magnetic Resonance in Medicine}, 21\penalty0 (2):\penalty0
  288--295, 1991.

\bibitem[Weinmann(2010)]{weinmannConstrApprox}
A.~Weinmann.
\newblock Nonlinear subdivision schemes on irregular meshes.
\newblock \emph{Constructive Approximation}, 31\penalty0 (3):\penalty0
  395--415, 2010.

\bibitem[Weinmann(2012{\natexlab{a}})]{weinmann2012interpolatory}
A.~Weinmann.
\newblock Interpolatory multiscale representation for functions between
  manifolds.
\newblock \emph{SIAM Journal on Mathematical Analysis}, 44:\penalty0 162--191,
  2012{\natexlab{a}}.

\bibitem[Weinmann(2012{\natexlab{b}})]{weinmann2012subdivision}
A.~Weinmann.
\newblock Subdivision schemes with general dilation in the geometric and
  nonlinear setting.
\newblock \emph{Journal of Approximation Theory}, 164\penalty0 (1):\penalty0
  105--137, 2012{\natexlab{b}}.

\bibitem[Weinmann et~al.(2014)Weinmann, Demaret, and
  Storath]{weinmann2014total}
A.~Weinmann, L.~Demaret, and M.~Storath.
\newblock Total variation regularization for manifold-valued data.
\newblock \emph{SIAM Journal on Imaging Sciences}, 7\penalty0 (4):\penalty0
  2226--2257, 2014.

\bibitem[Weinmann et~al.(2016)Weinmann, Demaret, and
  Storath]{weinmann2015mumford}
A.~Weinmann, L.~Demaret, and M.~Storath.
\newblock {M}umford--{S}hah and {P}otts regularization for manifold-valued
  data.
\newblock \emph{Journal of Mathematical Imaging and Vision}, 55\penalty0
  (3):\penalty0 428--445, 2016.

\bibitem[Xie and Yu(2008)]{xie2008smoothness}
G.~Xie and T.~Yu.
\newblock Smoothness equivalence properties of general manifold-valued data
  subdivision schemes.
\newblock \emph{Multiscale Modeling $\&$ Simulation}, 7\penalty0 (3):\penalty0
  1073--1100, 2008.

\end{thebibliography}
}
\end{document}